\documentclass[11pt]{amsart}
\usepackage[colorlinks=true, pdfstartview=FitV, linkcolor=blue, citecolor=blue, urlcolor=blue, breaklinks=true]{hyperref}
\usepackage{amsmath,amsfonts,amssymb,amsthm,mathrsfs,amscd,comment,paralist,stmaryrd,etoolbox,mathtools,enumitem}
\usepackage[usenames,dvipsnames]{color}
\usepackage{booktabs}
\usepackage{enumitem}
\usepackage[all]{xy}
\usepackage{tikz}
\usepackage{tikz-cd}
\usetikzlibrary{decorations.pathmorphing}
\usepackage{bbm}
\usepackage{multicol}
\usepackage{array}
\usepackage{longtable,array,booktabs}
\theoremstyle{plain}
\newtheorem{thm}{Theorem}[section]
\newtheorem{prop}[thm]{Proposition}
\newtheorem{lemma}[thm]{Lemma}
\newtheorem{cor}[thm]{Corollary}

\newtheorem{rem}[thm]{Remark}

\theoremstyle{definition}
\newtheorem{defn}[thm]{Definition}

\newtheorem{ex}[thm]{Example}
\newcommand{\Z}{\mathbb{Z}}
\newcommand{\C}{\mathbb{C}}
\newcommand{\N}{\mathbb{N}}
\newcommand{\g}{\mathfrak{g}}
\newcommand{\h}{\mathfrak{h}}

\newcommand{\n}{\mathfrak{n}}
\newcommand{\q}{\mathfrak{q}}

\newcommand{\RN}[1]{%
  \textup{\uppercase\expandafter{\romannumeral#1}}%
}

\newcommand{\A}{\mathbf{A}_{\lambda}^{\Gamma}}
\newcommand{\w}{w_{\lambda}^{\Gamma}}
\newcommand{\R}{\mathbf{R}_{\lambda}^{\Gamma}}
\newcommand{\lw}{w_{\psi}^{\Gamma}}
\newcommand{\LW}{W_{loc}^{\Gamma}(\psi)}

\DeclareMathOperator{\Ann}{Ann}
\DeclareMathOperator{\Hom}{Hom}

\DeclareMathOperator{\Ext}{Ext}

\DeclareMathOperator{\Supp}{Supp}

\DeclareMathOperator{\het}{ht}
\DeclareMathOperator{\wt}{wt}
\newtoggle{comments}
\newtoggle{details}
\newtoggle{prelimnote}
\newtoggle{detailsnote}

\toggletrue{detailsnote}   

\iftoggle{comments}{%
  \newcommand{\comments}[1]{
    \begin{center}
      \parbox{6.5 in}{
        \color{red}
          {\footnotesize \textbf{Comments:} #1}
        \color{black}}
    \end{center}}
}{%
  \newcommand{\comments}[1]{}
}

\iftoggle{details}{%
  \newcommand{\details}[1]{
      \ \\
      \color{OliveGreen}
        \begin{footnotesize}
          \textbf{Details:} #1
        \end{footnotesize}
      \color{black}
      \\
  }
}{%
  \newcommand{\details}[1]{}
}

\iftoggle{prelimnote}{%
  
}{%
  
}

\begin{document}

\title{Weyl modules for Equivariant map Lie superalgebras}

\author[Lakshmi S K]{Lakshmi S K}
\address{Lakshmi S K, Department of Mathematics, National Institute of Technology Calicut, NIT Campus P.O., Kozhikode-673 601, India}
\email{lakshmi\_p250569ma@nitc.ac.in}

\author[Nayak]{Saudamini Nayak}
\address{Saudamini Nayak, Department of Mathematics, National Institute of Technology Calicut, NIT Campus P.O., Kozhikode-673 601, India}
\email{saudamini@nitc.ac.in}

\begin{abstract}
 Equivariant map superalgebras are Lie superalgebras of algebraic maps from a scheme to a target finite dimensional Lie superalgebra that are equivariant with respect to the action of a (cyclic) group. In this paper, we extend the notion of Weyl modules, previously defined for the untwisted case, to the case of equivariant(twisted) map superalgebras. Consider $\mathbb{K}$ be an algebraically closed field of characteristic $0$. We define global Weyl modules, and Weyl functors for equivariant map Lie superalgebras $(\g\otimes A)^{\Gamma}$, where $\g$ is a basic classical Lie $\mathbb{K}$-superalgebra and $A$ is an associative commutative unital $\mathbb{K}$-algebra. Under certain assumption on the triangular decomposition of $\g$, we prove that global Weyl modules are universal objects in certain category. We introduce a commutative algebra $\mathbf{A}_{\lambda}^{\Gamma}$ and further prove that global Weyl modules are finitely generated $\mathbf{A}_{\lambda}^{\Gamma}$-modules when $A$ is finitely generated. Finally we define the local Weyl modules for $(\g\otimes A)^{\Gamma}$, where $\g$ is basic classical, using Weyl functors. We show that they are finite dimensional irrespective of the triangular decomposition of $\g^{\Gamma}$. Finally it has been shown that twisted local Weyl modules of $(\g\otimes A)^{\Gamma}$ are precisely the image of the untwisted local Weyl modules under the twisting functor $\textbf{T}_{\mathbf{x}}$.
\end{abstract}


\maketitle
\section{Introduction}\label{sec-intro}
The concept of Lie superalgebras became mainstream  area of research in the 1970s, motivated largely by the idea of supersymmetry in theoretical Physics. Lie superalgebras $\g=\g_{\bar{0}}\oplus\g_{\bar{1}}$ are generalizations of Lie algebras in the sense that $\g$ is a Lie algebra when the odd part $\g_{\bar{1}}=0$.  The theory of Lie superalgebras and their representations have a wide range of application in many areas of Physics and Mathematics such as string theory, quantum mechanics, conformal field theory, algebraic geometry, category theory and number theory. In 1975, Kac offered a comprehensive description of the mathematical theory of Lie superalgebras and established the classification of all finite dimensional simple Lie superalgebras $\g$ over an algebraically closed field of characteristic zero \cite{Kac1977}. These are divided into three groups namely: basic(classical and exceptional) Lie superalgebras, the strange ones(periplectic and queer) and the ones of Cartan type. In recent times there has been much interest in understanding  map Lie superalgebras $\g\otimes A$, where $\g$ is a simple finite dimensional Lie superalgebra and $A$ is some commutative associative unital algebra with Lie bracket is defined pointwise. They form a large class of Lie superalgebras with examples include loop $(A=\mathbb{C}[t^{\pm 1}])$, multi loop $(A=\mathbb{C}[t_1^{\pm 1}, \cdots, t_n^{\pm 1}])$ and current $(A=\mathbb{C}[t])$ superalgebras which play a very significant role in the theory of affine Kac-Moody Lie superalgebras. Kac classified the simple finite dimensional representations of basic classical Lie superalgebras \cite{Kac77a,Kac1978}. Classification of finite dimensional modules of map superalgebras is an extremely active area of research (\cite{Rao13, Sav14, CMS16, RZ04}). But the category of finite dimensional modules of map Lie algebras isn't semisimple. So
 Weyl modules play an important role in the representation theory of infinite dimensional (map) Lie algebras.

 Given a simple finite dimensional Lie algebra $\g$ over $\mathbb{C}$, Weyl modules for loop algebras $\g\otimes \mathbb{C}[t^{\pm 1}]$ were introduced by Chari and Pressley in \cite{CP01}. These modules are indexed by dominant integral weights of $\g$ and are closely related to certain irreducible modules for quantum affine algebras. Instead of the loop algebra, Feigin and Loktev \cite{FL04} extended the notion of Weyl modules for the higher dimensional case, in which they worked with the Lie algebra $\g\otimes A$ where A is the coordinate ring of an algebraic variety and obtained results analogous to \cite{CP01}. Later, in \cite{CFK10}, Chari et. al., considered a more general functorial approach to (global and local) Weyl modules associated with the algebras $\g\otimes A$, where $A$ is commutative associative unital algebra over $\mathbb{C}$. It was proved that the global Weyl module is a projective object in a suitable category and its weight spaces are right modules for a certain commutative algebra. Concepts of Weyl functors were introduced and the local Weyl modules along with their homological properties were also studied.

In \cite{CFS08, FMS13}, $\g\otimes\mathbb{C}[t^{\pm 1}]$ was replaced by the twisted loop algebra, denoted, by $(\g\otimes\mathbb{C}[t^{\pm 1}])^{\Gamma}$ which is the fixed point subalgebra of $\g\otimes\mathbb{C}[t^{\pm1}]$ under the action of a group $\Gamma$ of automorphisms of $\g$, generated by a Dynkin diagram automorphism. This acts on $\mathbb{C}[t^{\pm1}]$ by scaling $t$ by roots of unity. They were able to describe an identification of the Weyl modules for the twisted affine algebras with suitably chosen Weyl modules for the corresponding untwisted affine algebras. This identification allowed them to use known results in the untwisted case to compute the dimension and characters of the Weyl modules for the twisted algebras. In \cite{FKKS12}, local Weyl module for equivariant map algebras were defined under the assumption that the scheme is of finite type, group is abelian and the action on the scheme is free. The key ingredient was to introduce the notions of twisting and untwisting functors that relate the representation theory of map and equivarinat map algebras. In \cite{FMS15}, the global Weyl modules for equivariant map algebras were defined in terms of generators and relations. The notion of Weyl functors were also extended to the twisted settings. 

Then Zhang in \cite{Zha14}, first define and study the Weyl modules in the spirit of Chari-Pressley for a quantum analogue in the loop case for $\g = \mathfrak{sl}(m, n)$. In \cite{CLS19}, Calixto, Lemay and Savage  study Weyl modules for Lie superalgebras of the form $\g\otimes_{\C} A$, where A is an associative commutative unital $\C$-algebra and $\g$ is a basic complex Lie superalgebra or $\mathfrak{sl}(n, n), n \geq 2$. Particularly, they define (global and local) Weyl modules for the Lie superalgebras $\g\otimes_{\C} A$ and prove that global Weyl modules are universal highest weight objects in a certain category and local Weyl modules are finite-dimensional provided $A$ is finitely generated. Furthermore recently, Bagci, Calixto, and Macedo \cite{BCM19} studied (global and local) Weyl modules and Weyl functors for superalgebras $\g\otimes_{\C} A$, where $\g$ is either $\mathfrak{sl}(n, n),\ n\geq 2$, or any finite-dimensional complex simple Lie superalgebra not of type $\q(n)$, and where A is an associative commutative unital $\C$-algebra. In \cite{Nayak25}, second author studied global and local Weyl modules for Lie superalgebras $\q(n)\otimes_{\C} A$, where $\q(n)$ is the queer Lie superalgebra. Savage classified irreducible finite dimensional representation of equivariant map Lie superalgebras $(\g\otimes A)^{\Gamma}$, with the assumption that $\g$ is finite dimensional basic Lie superalgebra, $A$ is finitely generated and $\Gamma$ is an abelian group \cite{Sav14}. In this paper, we define and study global and local Weyl modules for equivariant map Lie superalgebras for certain basic classical Lie superalgebras. 

 It is worth mentioning here that Weyl modules for Lie superalgebras have many analog results as their non-super part. However, there are some striking differences. The Borel Lie superalgebras of basic Lie superalgebras are not conjugate under the action of Weyl group. Hence  the notation of Weyl modules depend upon the choice of simple root systems, which is in contrast to the situation on finite dimensional simple Lie algebras. Further, the category of finite dimensional modules for basic Lie superalgebras is not semisimple in general. Hence Kac-modules play an important role in the representation theory, which are maximal finite dimensional modules of a given highest weight. The Weyl modules defined in this paper can be viewed as the unification of several types of modules. If $\Gamma$ is trivial, then the Weyl modules defined here are Weyl modules for (untwisted) map Lie superalgebras. If $\g$ is a simple Lie algebra, then the definition reduces to the Weyl module in the non-super setting. On the other hand, if $A=\mathbb{C}$ then the global and local Weyl modules are the same and coincide with the Kac-module. In addition, if $\g$ is a simple Lie algebra, this is the irreducible module of a given highest weight.  
 
 The paper is structured as follows. In Section \ref{sec-prelim}, we provide basic definitions and results related to a $\mathbb{K}-$ Lie superalgbra which we use throughout. In Section \ref{sec:1}, we  define a triangular decomposition for $\g$ that is compatible with that of its invariant subalgebra $\g^{\Gamma}$ and we define highest weight module over $\g^{\Gamma}$. In Section \ref{sec:2} we recall the definitions of the map superalgebras, its highest weight modules and extend this concept to the equivariant map Lie superalgebras. We also give the structure and root space decomposition of $\g^{\Gamma}$.
 In Section \ref{sec:5}, we develop the theory of global Weyl modules for $(\g\otimes A)^{\Gamma}$. In Sections \ref{sec:6} and \ref{sec:7}, we discuss the construction of twisted Weyl functor, and the structure of global Weyl modules. We prove that, if $A$ is a finitely generated associative, unital $\mathbb{K}-$algebra, then the global Weyl module is a finitely generated $\A$-module. In Section \ref{sec:8}, we define the local Weyl module, corresponding to a dominant integral weight of $\g^{\Gamma}$, using the generator relations. We also prove that for any triangular decomposition of $\g^{\Gamma}$, the local Weyl module is finite dimensional. 
 In Section \ref{sec:9}, we prove that a twisted local Weyl module can be defined as an image of the untwisted local Weyl module under the action of the twisting functor. 

 For the reader's convenience, we give here some important notation used in this paper. 
  \begin{center}
      {\sc Index of Notation}
  \end{center} 
    \noindent
\begin{minipage}[t]{0.48\textwidth}
\centering
    \begin{tabular}{p{5cm}p{2.5cm}}
       $\mathbf{U}(\g)$  & Page \pageref{pg:1} \\
       $\Phi$, $\phi$  & Page \pageref{pg:2}\\
       $\g^{\alpha}$ & Equation \ref{eq:9}\\
       $R,~R_{\bar{0}}^{\pm},~R_{\bar{1}}^{\pm},~\Delta,~\Delta_{\bar{0}},~\Delta_{\bar{1}}$ & Page \pageref{pg:3}\\
       $P^{+}$ & Page \pageref{pg:4}\\
       $\Lambda$, $\Lambda^{+}$ & Page \pageref{pg:4} and equation \ref{eq:10}\\
       $Q_{\Gamma}$, $Q_{\Gamma}^{+}$ & Page \pageref{pg:6} and equation \ref{eq:10}\\
       $\mathcal{I}$, $\mathcal{I}^{\Gamma}$  & Subsection \ref{sub:1}\\
       $\mathcal{I}_{\lambda}$, $\mathcal{I}_{\lambda}^{\Gamma}$ & Page \pageref{pg:5}\\
       $W^{\Gamma}(\lambda)$ & Definition \ref{def:4}\\
       $\A$ & Equation \ref{eq:11}\\
       $\A-mod$ & Subsection \ref{sub:2}\\
       $W^{\Gamma}_{\lambda}$ & Definition \ref{def:5}\\
       $R_{\lambda}^{\Gamma}$ & Definition \ref{def:6}\\
       $W^{\Gamma}_{loc}(\psi)$ & Definition \ref{def:7}\\
       $\mathcal{L}(\h\otimes A)^{\Gamma}$ & Definition \ref{def:8}\\
       $\Supp J$ & Definition \ref{def:9}\\ $\Ann_{A}^{\Gamma}V$,  $\Supp^{\Gamma}_{A}V$ & Definition \ref{def:10}\\ 
        $X_{*}$ & Page \pageref{pg:7}\\
       \end{tabular}
\end{minipage}
\hfill
\begin{minipage}[t]{0.48\textwidth}
\centering
      \begin{tabular}{p{5cm}p{2.5cm}}

       $\mbox{ev}_{\textbf{x}}^{\Gamma}$, $\mbox{ev}^{\Gamma}_{\textbf{x}}(\rho_{x})_{x\in\textbf{x}}$ & Definition \ref{def:11}\\
       $\mbox{ev}^{\Gamma}_{\textbf{m}_{1}^{n_{1}},\cdots,\textbf{m}_{l}^{n_{l}}}$, $\mbox{ev}^{\Gamma}_{\textbf{m}_{1}^{n_{1}},\cdots,\textbf{m}_{l}^{n_{l}}}(\rho_{1},\cdots,\rho_{l})$ & Definition \ref{def:12}\\
       $I_{\eta}$ & Page \pageref{pg:7}\\
       $R(\g)$,  $\mbox{ev}_{\Psi}$ & Page \pageref{pg:8}\\
        $\Psi$, $\varepsilon(X,\g)=\varepsilon$, $\varepsilon(X,\g)^{\Gamma}=\varepsilon^{\Gamma}$ & Page \pageref{pg:8} or \pageref{pg:10}\\
        $\varepsilon_{\lambda}$, $\varepsilon^{\Gamma}_{\lambda}$ & Page \pageref{pg:12}\\ 
        $\mbox{ev}_{\Psi}^{\Gamma}$ & Definition \ref{def:13}\\
       $V_{\eta}$, $\mathcal{F}$, $\mathcal{F}^{\Gamma}$ & Page \pageref{pg:9}\\
       $\mathcal{F}_{\textbf{x}}$, $\mathcal{F}_{\textbf{x}}^{\Gamma}$ & Definition \ref{def:14}\\
       $\textbf{T}$, $\textbf{T}_{\textbf{x}}$ & Definition \ref{def:15}\\
       $\textbf{U}_{\textbf{x}}$ & Definition \ref{def:16}\\
       $Y_{\Gamma}$, $\Psi_{\textbf{x}}$, $\wt(\Psi)$, $\wt_{\Gamma}(\Psi)$, $ht\Psi$, $\het_{\Gamma}\Psi$ & Page \pageref{pg:11}-\pageref{pg:13}\\
       $W(\Psi)$ & Definition \ref{def:1} and \ref{def:3}\\
       $V(\Psi)$, $V^{\Gamma}(\Psi)$ & Definition \ref{def:17}\\
       $V_{\lambda}^{\Gamma}$ & Definition \ref{def:18}\\
       $M(\Psi)$, $M^{\Gamma}(\Psi)$ & Definition \ref{def:2}\\
       $W^{\Gamma}(\Psi)$ & Definition \ref{def:3}\\
       
       \end{tabular}
\end{minipage}

\section{Preliminaries}\label{sec-prelim}
Throughout the paper, the ground field is $\mathbb{K}$. All the vector superspaces, modules, algebras, Lie superalgebras etc. considered are over $\mathbb{K}.$ Set $\mathbb{Z}_{2}=\mathbb{Z}/ 2 \mathbb{Z}$.
\subsection{Basic classical Lie superalgebras}
  A vector superspace is a vector space that is endowed with a ${\Z}_2$- gradation: $V= V_{\bar{0}}\oplus V_{\bar{1}}$. The dimension of the vector superspace $V$ is the tuple $(\dim V_{\bar{0}} \mid \dim V_{\bar{1}})$. The parity/degree of a homogeneous element $a\in V$ is denoted by $|a|=i$ where $i\in \{0,1\}$. The element in $V_{\bar{0}}$ is called even and the element in $V_{\bar{1}}$ is called odd.

A \emph{Lie superalgebra} is a vector superspace ${\g}={\g_{\bar 0}}\oplus {\g_{\bar 1}}$ with bilinear multiplication $[\cdot,\cdot]$ that satisfies the following axioms:
  \begin{enumerate}[label=(\alph*)]
    \item The multiplication respects the grading: $[\g_i,\g_j] \subseteq \g_{i+j}$ for all $i, j \in \Z_2$.
    \item Skew-supersymmetry: $[a, b]=-(-1)^{|a||b|}[b,a]$, for all homogeneous elements $a, b\in \g$.
    \item Super Jacobi identity: $[a,[b, c]]=[[a, b],c]+(-1)^{|a||b|}[b,[a, c]]$, for all homogeneous elements $a, b, c\in \g$.
  \end{enumerate}

\begin{ex}
Let $A$ be any associative superalgebra. Then we can make $A$ into a Lie superalgebra by defining $[a, b]:=a b -(-1)^{|a||b|}b a$ for all homogeneous elements $a, b \in A$ and extending $[., .]$ by linearity. We call this is the Lie superalgebra associated with $A$. A concrete example is the general linear Lie superalgebra $\mathfrak{g l}(V)$ associated with associative superalgebra $End(V)$ of all linear operators on a vector superspace $V$.
\end{ex}
By a {\it homomorphism} between superspaces $f: V \rightarrow W $ of degree $|f|\in \mathbb{Z}_{2}$, we mean a linear map satisfying $f(V_{\alpha})\subseteq W_{\alpha+|f|}$ for $\alpha \in \mathbb{Z}_{2}$. In particular, if $|f| = \bar{0}$, then the homomorphism $f$ is called homogeneous linear map of even degree.
A homomorphism $\rho$ between  Lie superalgebras is a map which preserves the structure in them. Precisely $\rho: \g \longrightarrow \g'$ is an even linear map with $\rho([x, y])=[\rho x, \rho y]$ for all $x, y \in \g.$
A representation of Lie superalgebra $\g$ is a homomorphism $\rho: \g \longrightarrow \mathfrak{gl}(V)$, i.e., $\rho$ is an even linear map 
with $\rho[x, y]=\rho(x) \rho(y)- (-1)^{ |x| |y|} \rho(y) \rho(x) $ for all $x, y \in \g$. Alternatively $V$ is called a $\g$-module and $V$ is irreducible if there are no sub-modules other than $0$ and $V$ itself.

Observe that ${\g_{\bar 0}}$ inherits the structure of a Lie algebra and that ${\g_{\bar 1}}$ inherits the structure of a ${\g_{\bar 0}}$-module with respect to the adjoint representation. A Lie superalgebra is said to be simple if there are no non-zero proper ideals, i.e., there are no non-zero proper supersubspaces $i\subset \g$ such that $[i,\g]\subseteq i$. A finite dimensional simple Lie superalgebra ${\g}={\g_{\bar 0}}\oplus {\g_{\bar 1}}$ is said to be classical if the ${\g_{\bar 0}}$-module ${\g_{\bar 1}}$ is completely reducible. A simple Lie superalgebra is classical if and only if its even part ${\g_{\bar 0}}$ is a reductive Lie algebra.

 If ${\g}$ is the classical Lie superalgebra, then the adjoint representation of ${\g_{\bar 0}}$ on ${\g_{\bar 1}}$ is either
\begin{enumerate}[label=(\alph*)]
    \item irreducible, in which case we say that ${\g}$ is of type $\textbf{II}$, or
    \item the direct sum of two irreducible representations, in which case we say that ${\g}$ is of type $\textbf{I}.$
\end{enumerate}
A bilinear form $(., .)$ on a Lie superalgebra $\g$ is called consistent if $(x, y)=0$ for all $x\in \g_{\bar{0}}$ and $y \in \g_{\bar{1}}.$ It is called supersymmetric if $(x, y)=(-1)^{|x||y|} (y, x)$ for all $x, y \in \g $ and it is invariant if $([[x, y], z])=([x, [y, z]])$ for all $x, y, z \in \g.$ Two invariant bilinear forms on a simple Lie superalgebra $\g$ are  \cite{Mus12,Frappat2000}. An invariant bilinear form on a simple Lie superalgebra $\g$ is either non-degenerate or identically zero.
A classical Lie superalgebra ${\g}$ having a non-degenerate invariant bilinear form is called basic (otherwise it is called strange). All basic classical Lie superalgebras are perfect, that is, $[\g,\g]=\g$. The even part of every basic classical Lie superalgebra is either semisimple or reductive with one dimensional center. The bilinear form associated to the adjoint representation of a Lie superalgebra $\g$, is called the Killing form and is denoted as $K(x, y)$. It is defined by $K(x, y)= \mbox{str}(\mathrm{ad}_x, \mathrm{ad}_y)$ for all $x, y \in \g$. The Killing form is consistent, supersymmetric and invariant bilinear form on $\g$. In addition, $K(\sigma(x), \sigma(y))=K(x, y)$ for all $\sigma $ in the automorphism group of $\g$ and $x, y \in \g.$

For any basic classical Lie superalgebra $\g$, there exists a distinguished ${\Z}$-grading $\g=\oplus_{i\in{\Z}}\g_i$ that is compatible with the ${\Z}_2-$ grading and such that \cite{BCM19}
\begin{enumerate}[label=(\alph*)]
\item if $\g$ is of type $\textbf{I}$, then $\g_i=0$ for $|i|>1, {\g_{\bar 0}}=\g_0, {\g_{\bar1 }}=\g_{-1}\oplus \g_1$
\item if $\g$ is of type $\textbf{II}$, then $\g_i=0$ for $|i|>2, \g_{\bar{0}}=\g_{-2}\oplus\g_0\oplus\g_{2}, \g_{\bar{1}}=\g_{-1}\oplus\g_1$.
\end{enumerate}

The list of basic classical Lie superalgebras of type $\textbf{I}$ consists of $\mathfrak{osp}(2 \mid 2n), \mathfrak{sl}(m \mid n)$  for $m \neq n$ and $\mathfrak{psl}(n \mid n)$ for $n \geq 1.$
  The list of basic classical Lie superalgebras of type $\textbf{II}$ consists of $\mathfrak{osp}(m \mid 2n)$ for $m \neq 2$, $ D(2,1; \alpha), F(4)$ and $G(3)$.
\begin{lemma} [\cite{Sav14}]
Suppose $\g$ is a Lie superalgebra and $V$ is an irreducible $\g$-module such that $Iv=0$ for some ideal $I$ of $\g$ and non-zero vector $v\in V$. Then $IV=0$.
\end{lemma}
Given a Lie superalgebra $\g$, we denote $\mathbf{U}(\g)$ \label{pg:1} its {\em universal enveloping superalgebra}. The universal enveloping superalgebra $\mathbf{U}(\g)$ is constructed from the tensor algebra $T(\g)$ by factoring out  the ideal generated by elements $[u, v] - u\otimes v + (-1)^{|u||v|} v\otimes u$, for homogeneous elements $u, v$ in $\g$. Now we state an analogous of Poincare-Birkhoff-Witt(PBW) Theorem in super setting, which ensures that $\g \mapsto \mathbf{U}(\g)$ is an inclusion by precisely giving a basis for $\mathbf{U}(\g)$.

\begin{lemma}[\cite{Mus12}, Theorem 6.1.1] \label{lem:0}
Let $\g = \g_{\bar0} \oplus \g_{\bar1}$ be a Lie superalgebra. If $x_1, \ldots, x_m$ be a basis of $\g_{\bar0}$ and $y_1, \ldots, y_n$ be a basis of $\g_{\bar1}$, then the monomials
\[x_1^{a_1}\cdots x_m^{a_m} y_1^{b_1}\cdots y_n^{b_n}, \quad a_1, \ldots, a_m \geq 0, \quad \mbox{and}\quad b_1, \ldots, b_n \in \{0, 1\},\]
form a basis of $\mathbf{U}(\g)$. This is called the Poincare-Birkhoff-Witt(PBW) basis. In particular, if $\g$ is finite dimensional and $\g_{\bar 0} = 0$, then $\mathbf{U}(\g)$ is finite dimensional.
\end{lemma}
\subsection{Category $C_{(\g, \mathbf{t})}$}
Let $\g$ be a Lie superalgebra. Then a $\g$-module $V$, is said to be finitely semisimple if it is equal to the direct sum of its finite dimensional irreducible submodules. Given a sub(super)algebra $\mathbf{t}\subseteq\g,$ let $C_{(\g, \mathbf{t})}$ denote the full subcategory of the category of all $\g$-modules whose objects are $\g$-modules which are finitely semisimple as $\mathbf{t}$-modules.

\begin{lemma}[\cite{Kumar2002}]\label{lem:1} 
The category $C_{(\g, \mathbf{t})}$ is closed under taking submodules, quotients, arbitrary direct sums and finite tensor products.
\end{lemma}
Given a Lie superalgebra $\g$, Lie sub(super)algebra $\mathbf{t}\subseteq\g$ and a $\mathbf{t}$-module $M$, we define the induced module 
\begin{equation*}
  \mathrm{ind}^{\g}_\mathbf{t}M=\mathbf{U}(\g)\otimes_{\mathbf{U}(\mathbf{t})}M  
\end{equation*}
with the action induced by the left multiplication.
\begin{lemma}[\cite{BCM19}]\label{lem:2}
Let $\g$ be a Lie superalgebra, $\mathbf{t}\subseteq\g$ be a Lie subalgebra and $M$ be a $
\mathbf{t}$-module. If $\g$ (via the restriction of adjoint representation) and $M$ are finitely semisimple $\mathbf{t}$-modules, then $\mathrm{ind}^{\g}_\mathbf{t} M$ is an object in $C_{(\g, \mathbf{t})}$.
\end{lemma}
\begin{lemma}[\cite{BCM19}]
Let $\g$ be a Lie superalgebra, $\mathbf{t}\subseteq\g$ be a Lie subalgebra. If $M$ is a cyclic $\mathbf{t}$-module given as the quotient of $\mathbf{U}(\mathbf{t})$ by a left ideal $J\subset\mathbf{U}(\mathbf{t})$, then $\mathrm{ind}^{\g}_\mathbf{t}M$ is a cyclic $\g$-module given as the quotient of $\mathbf{U}(\g)$ by the left ideal generated by $J$ in $\mathbf{U}(\g)$.
\end{lemma}
\subsection{Root space and triangular decomposition}
Let $\g$ be a simple classical Lie superalgebra. Let $\h_{\bar{0}}$ be a Cartan subalgebra of $\g_{\bar{0}}$. Then $\h$ is taken to be the centralizer of $\h_{\bar{0}}$ in $\g$. If $\Phi_{\bar{i}}$, for $i=0,1$, denotes the set of roots of $\g_{\bar{i}}$ with respect to $\h_{\bar{0}}$, then $\Phi=\Phi_{\bar{0}}\cup \Phi_{\bar{1}}$. Thus \label{pg:2},
\begin{equation*}
    \Phi=\{\alpha\in\h_{\bar{0}}^*\mid \alpha\neq 0, (\g_{\bar{i}})_{\alpha}\neq 0\}
\end{equation*}
where $(\g_{\bar{i}})_{\alpha}=\{x\in\g_{\bar{i}}\mid [h, x]=\alpha(h)x,\forall, h\in\h_{\bar{0}}\}$.
Once $\h_{\bar{0}}$ is chosen, the canonical root space decomposition is given by
\begin{equation*}
     \g=\h\oplus\bigoplus_{\alpha\in \Phi}\g_{\alpha}.
\end{equation*}
The triangular decomposition for $\g$ is given by 
\begin{equation*}
    \g=\n^-\oplus\h\oplus\n^+
\end{equation*}
where $\n^{\pm}=\oplus_{\alpha\in\pm\Phi^{+}}\g_{\alpha}$. Here $\mathfrak{b}=\h\oplus\n^+$ is the Borel subalgebra of $\g$. If $\g$ is a basic classical Lie superalgebra, then the Cartan subalgebra $\h$ of $\g$ is precisely the Cartan subalgebra $\h_{\bar{0}}$ of $\g_{\bar{0}}.$ We  denote $\phi$ to be a base for $\Phi$.

\subsection{Diagram automorphisms}(\cite{Mus12})
For a simple Lie superalgebra $\g=\g_{\bar{0}}\oplus \g_{\bar{1}}$, an automorphism $f$ on $\g$ is a bijective homomorphism on $\g$. Let $\Gamma$ denote the automorphism group of $\g$. Suppose that $\g$ is a basic classical Lie superalgebra. Then $\g$ is generated by the elements $e_i$ and $f_i$. Let $I=\{1,2,\ldots, n\}$ and $S_n$ be the permutation group on $I$. For $\sigma\in S_n$, we say that $\sigma$ is a diagram automorphism if there is a non zero scalar $\lambda$ such that 
\begin{equation}\label{eq1}
    a_{i, j}=\lambda a_{\sigma(i),\sigma(j)} \quad for\;\; j=1,\ldots,n.
\end{equation}
\begin{lemma}\label{lem9}
If $\sigma$ satisfies \eqref{eq1}, there is an automorphism $\nu$ of $\g$ such that 
    \begin{equation}
        \nu(e_i)=\lambda e_{\sigma(i)}, \quad \nu(f_i)= f_{\sigma(i)}, \quad \nu(h_i)=\lambda h_{\sigma(i)}.
    \end{equation}
\end{lemma}
\section{The invariant subalgebra $\g^{\Gamma}$}\label{sec:1}
We begin with defining the invariant sub(super)algebras $\g^{\Gamma}$ which are further classified as regular and singular \cite{Frappat2000}. Further this section focuses on defining a triangular decomposition for $\g$ that is compatible with that of $\g^{\Gamma}$ and provides a root theoretic framework, necessary for the study of the Weyl modules. For a sub(super)algebra $\mathbf{t}\subseteq\g^{\Gamma}$ and a $\mathbf{t}$-module $M$, consider the induced module
  $\mathrm{ind}^{\g^{\Gamma}}_\mathbf{t}M=\mathbf{U}(\g^{\Gamma})\otimes_{\mathbf{U}(\mathbf{t})}M $
with the action induced by the left multiplication.
\subsection{Category $C_{(\g^{\Gamma},\mathbf{t})}$}
\begin{defn}(Invariant sub(super)algebra $\g^{\Gamma}$)
  When an automorphism group $\Gamma$ acts on a basic classical Lie superalgebra, we can talk about the fixed subalgebra $\g^{\Gamma}$ occurring due to the action of the automorphism group. We define $\g^{\Gamma}$ as
\begin{equation*}
    \g^{\Gamma}=\{x\in\g \mid \sigma(x)=x, \forall \sigma\in\Gamma\}.
\end{equation*}  
\end{defn}

Note that the fixed subalgebra $\g^{\Gamma}$ formed depends on the automorphism group acting on it. Accordingly, it can be mainly of two types:
\begin{enumerate}[label=(\alph*)]
\item Regular invariant sub(super)algebras:
Let $\g$ be a basic classical Lie superalgebra and consider its canonical root space decomposition 
\begin{equation*}
        \g=\h\oplus\bigoplus_{\alpha\in\Phi}\g^{\alpha}
\end{equation*}
where $\h$ is the Cartan subalgebra of $\g$ and $\Phi$ is its corresponding root system. We say a subalgebra $\g'$ of $\g$ is regular if it has a root space decomposition 
\begin{equation*}
\g'=\h'\oplus\bigoplus_{\alpha'\in\Phi'}\g'^{\alpha'}
\end{equation*}
where $\h'\subset\h$ and $\Phi\subset\Phi'$.
\item Singular invariant sub(super)algberas: Let $\g$ be basic Lie superalgbera and $\g'$ be a subsuperalgebra of $\g$. Then $\g'$ is said to a singular subsuperalgebra if it is not regular. These singular sub(super)algebras can sometimes be found using the folding technique. 
\end{enumerate}

\begin{lemma}\label{lem10}
Let $\mathbf{t}$ be the sub(super)algebra of the Lie superalgebra $\g^{\Gamma}$. Let $C_{(\g^{\Gamma},\mathbf{t})}$ denote the full subcategory of the category of all $\g^{\Gamma}$-modules whose objects are finitely semisimple as $\mathbf{t}$-modules. Then this category is closed under taking submodules, quotients, arbitrary direct sums, and finite tensor products. 
\end{lemma}
\begin{proof}
The proof follows from Lemma $\ref{lem:1}$.
\end{proof}
\begin{lemma}\label{lem11}
Suppose $\mathbf{t}\subseteq\g^{\Gamma}$ is a Lie sub(super)algebra and $M$ is a $\mathbf{t}$-module. If $\g^{\Gamma}$(via the restriction of adjoint representation) and $M$ are finitely semisimple as $\mathbf{t}$-modules, then $\mathrm{ind}^{\g^{\Gamma}}_{\mathbf{t}}M$ is an object in $C_{(g^{\Gamma}, \mathbf{t})}$.
\end{lemma}
\begin{proof}
The proof follows from Lemma $\ref{lem:2}$.
\end{proof}
\subsection{Triangular decomposition of $\g^{\Gamma}$}
Let $\g$ be the basic classical Lie superalgebra  and $\Gamma$ be the finite group acting on $\g$ by diagram automorphisms. Let $\mathcal{J}$ and $\mathcal{J}^{\Gamma}$ denote the set of nodes of Dynkin diagrams of $\g$ and $\g^{\Gamma}$ respectively. Let $\h_{\Gamma}$ be the Cartan subalgebra of $\g^{\Gamma}$.
Fix a triangular decomposition $\g^{\Gamma}=\n^{-}_{\Gamma}\oplus\h_{\Gamma}\oplus\n^{+}_{\Gamma}$ of $\g^{\Gamma}$. Let $Q_{\Gamma}^{+}$ denote the positive root lattice of $\g^{\Gamma}$ associated with the triangular decomposition and $Q_{\Gamma}^{-}$ be the negative root lattice. We have a root space decomposition for $\g$ with respect to $\h_{\Gamma}$ given by, 
\begin{equation}\label{eq:9}
\g=\bigoplus_{\alpha\in\h_{\Gamma}^{*}}\g^{\alpha},\quad ~\g^{\alpha}=\{x\in\g \mid [h, x]=\alpha(h)x, \forall  h\in\h_{\Gamma}\}
\end{equation}
with only finitely many $\g^{\alpha}$ non zero. Let
\begin{equation*}
    \n^{-}=\bigoplus_{\alpha\in Q_{\Gamma}^{-}\setminus\{0\}}\g^{\alpha}, \quad \n^{+}=\bigoplus_{\alpha\in Q_{\Gamma}^{+}}\g^{\alpha}.
\end{equation*}
Then $\g=\n^{-}\oplus\g^{0}\oplus\n^{+}$ (vector space direct sum) where $\g^{0}$ and $\n^{\pm}$ are Lie subalgebras of $\g$. Also, $\n_{\Gamma}^{-}=\g^{\Gamma}\cap\n^{-}$, $\n_{\Gamma}^{+}=\g^{\Gamma}\cap\n^{+}$ are Lie subalgebras of $\g^{\Gamma}$. This gives us a triangular decomposition for $\g$ that is consistent with the triangular decomposition for $\g^{\Gamma}$. Clearly, $\g^{0}$ is the centralizer of $\h_{\Gamma}$ in $\g$. The root spaces $\g^{\alpha}$ are preserved by the action of $\Gamma$.
\begin{rem}
Suppose $\g^{\alpha}=\{x\in\g\mid [h, x]=\alpha(h)x,~\forall h\in\h_{\Gamma}\}$ and for $\sigma\in\Gamma$ and $x\in\g^{\alpha}$ we have
$\sigma([h, x])=\sigma(\alpha(h)x)$.
Clearly,
\begin{equation*}
     [h,\sigma(x)] =  [\sigma(h),\sigma(x)]=\alpha(h)\sigma(x)
\end{equation*}
which implies that $\sigma(x) \in \g^{\alpha}$.    
\end{rem}
 We have  $\Gamma(\g^{\alpha})= \g^{\alpha}$ for each $\alpha$ and hence $\Gamma \n ^{\pm}=\n^{\pm}$ and $\Gamma \g^{0}=\g^{0}.$ The next lemma shows that $\g^{0}$ is a self-normalizing subalgebra.
\begin{lemma}
$\g^{0}$ is a self normalizing subalgebra.
\end{lemma}
\begin{proof}
Let  $N_{\g}(\g^{0})$ be the normalizer of $\g^{0}$ in  $\g$. Then $N_{\g}(\g^{0})=\{x\in \g \mid [x,y]\in\g^{0},~~~\forall y\in\g^{0}\}$. Our claim is $N_{\g}(\g^{0})=\g^{0}$. Clearly $\g^{0} \subseteq N_{\g}(\g^{0}).$ To show the reverse inclusion, let $y\in\g$ be such that $[y, \g^{0} ]\subset \g^{0}$. Suppose $y \in \g^{\alpha}$ for some $\alpha$, then $[h, y]=\alpha(h)y$ for all $h \in \h_{\Gamma} \subseteq \g^{0}$ as it is the centralizer of $\h_{\Gamma}$ in $\g.$  By assumption $[h, y]=-[y, h]$ is in $\g^{0}$ and hence $y \in \g^{0}.$
\end{proof}
Let $\Gamma$ be a cyclic group of order $m$, such that $\Gamma=<\sigma>\cong\mathbb{Z}/m\mathbb{Z}$, be an automorphism group acting on $\g$. Let $\zeta$ be the primitive $m^{th}$ root of unity. We say $\zeta$ is an eigenvalue corresponding to the eigenspace
 \begin{equation}
     \g_{s}=\{x\in\g\mid \sigma(x)=\zeta^{s}x\}.
 \end{equation}
Hence we obtain the following $\mathbb{Z}_{m}$-gradation for $\g$
\begin{equation}\label{eq:7}
    \g=\bigoplus_{s=0}^{m-1}\g_{s}.\
\end{equation}
\begin{lemma}\label{lem3.3}
\begin{enumerate}[label=(\alph*)]
\item Let $(\cdot \mid \cdot)$ be a non-degenerate, supersymmetric, consistent and invariant bilinear form on $\g$ which is also invariant under the automorphism group of $\g$. Then \[(\g_{i} \mid\g_{j})=0\quad
\mathrm{if}\;\; i+j \not\equiv 0 {\pmod m}.\] Otherwise, they are non-degenerately paired.
\item The centralizer of $\h_{\Gamma}$ in $\g$ is the  Cartan subalgebra of $\g$. 
\end{enumerate}
\end{lemma}
\begin{proof}
\begin{enumerate}[label=(\alph*)]
\item 
Consider the $\mathbb{Z}_{m}$-gradation for $\g$. Let  $x\in\g_{i}, y \in \g_{j}$ then 
 \[(x \mid y)=(\sigma(x) \mid \sigma(y))=\zeta^{i+j}(x \mid y).\]
 If $i+j \not\equiv 0\, {\pmod m}$, then $(x \mid y)=0$. If $i+j \equiv 0\, {\pmod m}$ then $(\g_{i} \mid \g_{j})$ are non-degenerately paired
 (as $(\cdot \mid \cdot)$ is non-degenerate).
 \item
Let $\h$ be the Cartan subalgebra of $\g$ and let us denote the centralizer of $\h_{\Gamma}$ in $\g$ as $\textbf{z}$. Our claim is $\h=\textbf{z}.$ Suppose $\textbf{z}=\h +\Sigma \g_{\alpha}$ where $\g_{\alpha}$ is the root space with respect to $\h$ and we take the roots such that $\alpha \mid_{\h_{\Gamma}}=0$. Hence $\h_{\Gamma} \subset \h$. Then $\textbf{z}=\h + \textbf{s}$ where $\textbf{s}$ is $\sigma$-invariant semisimple subalgebra. Clearly $\textbf{s} \cap \g_{0}=\{0\}$. Consider the $\mathbb{Z}_{m}$-gradation  $ \textbf{s}=\bigoplus_{k=0}^{m-1}\textbf{s}_{k}$ for $\textbf{s}$ such that $\textbf{s}_{0}=\{0\}$.

Let $N_{m}=\{0, 1, \cdots, m-1\}$ and $\textbf{s}_{a}=\textbf{s}_{b}$ if $b \in N_{m}$ and $a \equiv b\; {\pmod m}$. We want to prove $\textbf{s}=0$, and we will achieve this by showing $\textbf{s}_{n}=0$ for each integer $n$. We induct on $n$; for $n=0$ we have $\textbf{s}_{0}=0$. Let $n>0$ and $x \in \textbf{s}_{n}$. Then $(\mathrm{ad}_x)^{r} \textbf{s}_{i} \subseteq \textbf{s}_{nr+i}$. Choose $n \in \mathbb{N}$ such that $n(r-1) < m-i$ which implies $n r+i < m+n$. For some $0 \leq t <n$ we have $n r+i=m+t$. Hence $\textbf{s}_{n r+i}=\textbf{s}_{m+t}=\textbf{s}_{t}=0$ and the last equality holds by using the induction hypothesis. We find that $\mathrm{ad}_{x}$ is nilpotent. Similarly, $\mathrm{ad}_{y}$ is nilpotent for $y \in \textbf{s}_{-n}.$ But $[\textbf{s}_{n}, \textbf{s}_{-n}]\subset \textbf{s}_{0}=0$, so $x, y$ commutes. Now consider two cases; let both $x, y \in \g_{\bar{0}}$. Then \[0=\mathrm{ad}_{[x, y]}=[\mathrm{ad}_{x}, \mathrm{ad}_{y}]=\mathrm{ad}_{x} \mathrm{ad}_{y}-\mathrm{ad}_{y} \mathrm{ad}_{x},\] i.e. $\mathrm{ad}_{x}$ and $\mathrm{ad}_{y}$ commutes. If  $x, y \in \g_{\bar{1}}$. Then
\[ 0=\mathrm{ad}_{[x, y]}=[\mathrm{ad}_{x}, \mathrm{ad}_{y}]=\mathrm{ad}_{x} \mathrm{ad}_{y}+\mathrm{ad}_{y} \mathrm{ad}_{x}\] i.e., $\mathrm{ad}_{x}$ and $\mathrm{ad}_{y}$ anti commutes. In either cases $\mathrm{ad}_{x} \mathrm{ad}_{y}$ is a nilpotent operator. By using Engel's theorem, there is a basis for $\g$ with respect to which we can write all the nilpotent operators as strictly upper triangular matrices. Then trace of $\mathrm{ad}_{x} \mathrm{ad}_{y}$ is $0$ which means $K(x, y)=0$.
  
Any bilinear form $(\cdot \mid \cdot)$ invariant under automorphism group of $\g$ is a scalar multiple of the killing form $K(x, y)$ and hence $(x\mid y)=0$. But $x\in \textbf{s}_{n}$ and $y\in \textbf{s}_{-n}$. We have $\textbf{s}_{n}$ and $\textbf{s}_{-n}$ are non-degenerately paired, hence $x=y=0$. This means $\textbf{s}_{n}=0$ for each $n$ which proves the second part.
\end{enumerate}
\end{proof}
\begin{rem}
    The authors later discovered the first part of Lemma \ref{lem3.3} had already been there in \cite{vdL}.
\end{rem}
 \begin{lemma}\label{lem 3.4}
If $\Gamma$ is a cyclic group, then $\g^{0}=\h$, where $\h$ is the Cartan subalgebra of $\g$. In particular, if $\g$ is simple classical Lie superalgebra, then $\g^{0}=\h=\h_{\bar{0}}$.
\end{lemma}
\begin{proof}
It follows from Lemma \ref{lem3.3}, as $\g^{0}$ is centralizer of $\h_{\Gamma}$ in $\g$.
\end{proof}
\begin{rem}
It may happen that $\g^{\Gamma}=0$, in that case $\h_{\Gamma}=0$ and so $\g^{0}=\g$ is simple. However, from the Lemma \ref{lem 3.4} it is clear that if $\Gamma$ is a cyclic group, then $\g^{0}=\h_{\bar{0}}$ is an abelian subalgebra and hence $\g^{\Gamma}\neq 0$.
\end{rem}

Let $R =\{\alpha\in
\h_{\Gamma}^{*}\setminus \{0\}\mid \g^{\Gamma}_{\alpha}\neq 0\}$ \label{pg:3} be the set of roots, where $\g^{\Gamma}_{\alpha}=\{x\in\g^{\Gamma}\mid [h, x]=\alpha(h)x, \forall h\in \h_{\Gamma}\}$. Note that $\g^{\Gamma}_{\alpha}=(\g^{\alpha})^{\Gamma}$. For $\alpha\in R$, we have $\g^{\Gamma}_{\alpha}$ is either purely even, if $(\g^{\Gamma})_{\alpha}\subset\g_{\bar{0}}^{\Gamma}$ or, purely odd if $(\g^{\Gamma})_{\alpha}\subset\g_{\bar{1}}^{\Gamma}$. Let $R_{\bar{0}}$ be the set of even roots and $R_{\bar{1}}$ be the set of odd roots. Hence, we get $R=R_{\bar{0}}\cup R_{\bar{1}}$.
Let $\Delta\subset R$ denote the set of simple roots. Every choice of a set of simple roots $\Delta\subseteq R$ will yield a decomposition $R=R^{+}(\Delta)\cup R^{-}(\Delta)$ where $R^{+}(\Delta)$ denotes the positive roots and $R^{-}(\Delta)$ denotes the set of negative roots defined in the usual way. For ease of notation, we write $R^{+}(\Delta)=R^{+}$ and $R^{-}(\Delta)=R^{-}$. Define
\begin{equation*}
\Delta_{\bar{0}}=\Delta\cap R_{\bar{0}},\;\; \Delta_{\bar{1}}=\Delta\cap R_{\bar{1}},
    \;\; R^{+}_{\bar{0}}=R_{\bar{0}}\cap R^+,\;\; R^{-}_{\bar{0}}=R_{\bar{0}}\cap R^-,\;\;  R^{+}_{\bar{1}}=R_{\bar{1}}\cap R^+,\;\; 
    R^{-}_{\bar{1}}=R_{\bar{1}}\cap R^-.
\end{equation*}

\subsection{\textbf{Highest weight module over $\g^{\Gamma}$}}
 A $\g$-module $V$ is called a weight module if it admits a weight space decomposition 
\[V = \bigoplus_{\mu \in \h_{\bar 0}^{*}} V_{\mu}, \;\;\mbox{where}\;\; V_{\mu} = \{v \in V \mid hv = \mu(h)v, \forall h \in \h_{\bar 0}\}. \]
An element $\mu \in \h_{\bar 0}^{*}$ such that $V_{\mu} \neq 0$ is called a {\em weight} of $V$ and $V_{\mu}$ is called weight space. A vector $v\in V_{\mu}\setminus \{0\}$ is said to be the highest weight vector with highest weight $\mu$, if $\n^+v=0$. Similarly, $\lambda\in\h^*_{\bar{0}}$ is said to be lowest weight of a $\g$-module V, if $V_\lambda\neq\{0\}$ and $\n^-V_\lambda=\{0\}$ (\cite{BCM19}). Every irreducible finite dimensional $\g$-module is a highest weight module.
\begin{rem}
For a basic classical Lie superalgebra $\g$, the Cartan subalgebra $\h$ of $\g$ is the same as the Cartan subalegbra of $\g_{\bar{0}}$, i.e., $\h=\h_{\bar{0}}$. Hence $\h^*=\h^*_{\bar{0}}$ and $\h^*_{\Gamma}=(\h^*_{\Gamma})_{\bar{0}}.$
\end{rem} 
A $\g^{\Gamma}$-module $V$ is called a weight module if it has the weight space decomposition
$V = \bigoplus_{\mu \in \h^*_{\Gamma}} V_{\mu}, \;\;\mbox{where}\;\; V_{\mu} = \{v \in V \mid hv = \mu(h)v, \; \forall\; h \in \h_{\Gamma}\}$.
Here $\mu\in\h^*_{\Gamma}$ corresponding to $V_{\mu}\neq0$ is the weight of $V$ and $V_{\mu}$ is called the weight space. A vector $v\in V_{\mu}\setminus \{0\}$ is said to be the highest weight vector, with weight $\mu$, if $\n^+_{\Gamma}v=0$ and $\lambda\in\h^{*}_{\Gamma}$ is said to be the lowest weight of $\g^{\Gamma}$-module $V$ if $V_{\lambda}\neq0$ and $\n^-_{\Gamma}V_{\lambda}=\{0\}$.

\section{Equivariant map superalgebras}\label{sec:2}
In this section, first we recall the definitions of the map superalgebras, its highest weight modules, equivariant map Lie superalgebras and their highest weight modules \cite{Sav14}. We give the structure and root space decomposition of $\g^{\Gamma}$ and define one important condition which we use in defining the properties of the twisted global Weyl module. Recall $\g=\g_{\bar{0}} \oplus \g_{\bar{1}}$ is a Lie superalgebra over $\mathbb{K}$ and $A$ be a commutative associative unital $\mathbb{K}-$algebra.
\begin{defn}(Map superalgebras \cite{BCM19,Sav14})
The map superalgebra, $\g\otimes A$, is a vector space endowed with the $\mathbb{Z}_{2}$ grading given by $\g\otimes A=(\g\otimes A)_{\bar{0}}\oplus(\g\otimes A)_{\bar{1}}=(\g_{\bar{0}}\otimes A)\oplus(\g_{\bar{1}}\otimes A)$ and the Lie superbracket extending bilinearly
\begin{equation*}
    [u_1\otimes f_1,u_2\otimes f_2]=[u_1,u_2]\otimes f_1f_2    \quad u_1,u_2\in\g \quad f_1,f_2\in A.
\end{equation*}
    
\end{defn}
\begin{defn}(Weight Modules for map Lie superalgebras)
  A $\g\otimes A$-module $V$, is said to be a weight module, if its restriction to $\g$ is a weight module, that is, if
  \begin{equation*}
      V=\bigoplus_{\lambda\in\h^*}V_{\lambda},~~V_{\lambda}=\{v\in V|hv=\lambda(h)v,~\forall h\in\h\}.
  \end{equation*}
  The $\lambda\in\h^*$ such that $V_{\lambda}\neq0$, are called weights of $V$. A non zero element  $v\in V_{\lambda}$, for $\lambda\in\h^*$, is called as the weight vector of weight $\lambda$.
\end{defn}
\begin{defn}(Highest Weight Modules for map Lie superalgebras)
A $\g\otimes A$-module $V$ is called the {\em highest weight module}, if there exists a non zero vector $v\in V$ such that \[(\n^+\otimes A)v=0,  \quad   \mathbf{U}(\h\otimes A)v= \mathbb{K}v, \quad\mathbf{U}(\g\otimes A)v=V.\] Such a vector $v$ is called the highest weight vector corresponding to the highest weight $\lambda$ and $V_{\lambda}$ is called the highest weight space.
\end{defn}

Every highest weight module is a weight module as the highest weight vector is a weight vector and it generates the whole module. We identify Lie superalgebra $\g$ with the subsuperalgebra $\g\otimes \mathbb{K}\subseteq \g\otimes A$.
\begin{lemma}[\cite{Sav14}]\label{lem:3}
    Every irreducible finite dimensional $(\g\otimes A)$-module is a highest weight module.
\end{lemma}

 Let $\Gamma$ be a group acting on A and Lie superalgebra $\g$ by automorphisms. Then $\Gamma$ acts naturally on $\g\otimes A$ by extending the map $\gamma(u\otimes f)=(\gamma u)\otimes(\gamma f), \gamma\in\Gamma,u\in\g, f\in A$ by linearity. For a cyclic group $\Gamma$ acting on $\g$ ,we have already seen that the $\mathbb{Z}_{m}$-gradation for $\g$ is given by $\g=\bigoplus_{s=0}^{m-1}\g_{s}$. The action of $\Gamma$ on $A$ gives the gradation of $A$ to be $A=\bigoplus_{s=0}^{m-1}A_{s}$.

 \begin{defn}(Equivariant map superalgebras)
The equivariant map Lie superalgebra, denoted by $(\g\otimes A)^{\Gamma}$, is defined to be
 \begin{equation*}
     (\g\otimes A)^\Gamma = \{\mu\in\g\otimes A \mid \gamma\mu=\mu, \;\forall\quad \gamma\in\Gamma\},
 \end{equation*}
 which is the superalgebra of points fixed under the  action of $\Gamma$.
 \end{defn}
 Precisely, these are going to be elements from $\g_{s}\otimes A_{-s}$, i.e., $(\g\otimes A)^{\Gamma}=\bigoplus_{s=0}^{m-1}\g_s\otimes A_{-s}$ since, $u\otimes f\in \g_{s}\otimes A_{-s}~\Leftrightarrow ~\gamma(u\otimes f)=(\gamma u)\otimes(\gamma f)=\zeta^{s} u\otimes\zeta^{-s}f =u\otimes f$.

Since $\Gamma$ respects the triangular decomposition defined for $\g=\n^-\oplus\h \oplus\n^+$, we have the decomposition
\begin{equation}
    (\g\otimes A)^\Gamma=(\n^-\otimes A)^\Gamma\oplus(\h\otimes A)^\Gamma\oplus(\n^+\otimes A)^\Gamma.
\end{equation}
We give another equivalent way of defining the above gradation. Let $\Xi$ be the character group of $\Gamma$. This is an abelian group, whose group operation we will write additively. Hence 0 is the character of the trivial one dimensional representation. If $\xi$ is a character corresponding to an irreducible representation then $-\xi$ is the character of the dual representation.
 Now $(\g\otimes A)^{\Gamma}$ can also be written as
\begin{equation}
    (\g\otimes A)^\Gamma=\bigoplus_{\xi\in\Xi}\g_\xi\otimes A_{-\xi}
\end{equation}
where $\g_{\xi}=\{x\in\g~|~\sigma(x)=\xi x\}$ and $A_{\xi}=\{a\in A~|~\sigma(a)=\xi a\}$.
We say that $\g=\bigoplus_\xi g_\xi$ and $A=\bigoplus_\xi A_\xi$ are $\Xi-$ graded and $(\g_\xi\otimes A_{\xi'})^\Gamma=0$ if $\xi'\neq-\xi$.
\begin{defn}(Weight Modules and Highest Weight Modules)
    A $(\g\otimes A)^{\Gamma}$-module $V$ is called a weight module if its restriction to $\g^{\Gamma}$ is a weight module.
    \begin{equation*}
        V=\bigoplus_{\lambda\in\h_{\Gamma}^*}V_{\lambda},\quad V_{\lambda}=\{v\in V \mid hv=\lambda(h)v, \; \forall~h\in\h_{\Gamma}\}.
    \end{equation*}
    For $\lambda\in\h_{\Gamma}^*$, with $V_{\lambda}\neq0$, are called the weights of $V$ and $v\in V_{\lambda}$, such that $v\neq0$, is called the weight vector corresponding to the weight $\lambda$.
    \smallskip
A $(\g\otimes A)^{\Gamma}$-module $V$ is said to be the highest weight module, if there exists a non zero vector $v\in V$, called the highest weight vector,  such that $(\n^+\otimes A)^{\Gamma}v=0$, $\mathbf{U}(\h\otimes A)^{\Gamma}v=\mathbb{K}v$.
 \end{defn}
\begin{lemma}[\cite{Sav14}]\label{lem:4}
    Every finite dimensional $(\g\otimes A)^{\Gamma}$-module $V$ is the restriction of a $(\g\otimes A)$-module $\bar{V}$. Furthermore, $V$ is irreducible if and only if $\bar{V}$ is irreducible.
\end{lemma}
\begin{lemma}[\cite{Sav14}]
Suppose $\g$ is a finite dimensional simple Lie superalgebra. Then all ideals of $(\g\otimes A)^\Gamma$ are of the form $(\g\otimes\mathbf{I})=\bigoplus_{\zeta\in\Xi}\g_\zeta\otimes\mathbf{I}_{-\zeta}$ where $\mathbf{I}=\bigoplus_{\zeta\in\Xi}\mathbf{I_\zeta}$ is a $\Gamma$-invariant ideal of $A$.
 \end{lemma}
\subsection{More on Structure of $\g^{\Gamma}$.}
Now we define how we arrive at the invariant subalgebra $\g^{\Gamma}$ in this paper. Let $\Gamma$ be the diagram automorphism group generated by the automorphism $\sigma$. The action of a diagram automorphism, and hence the invariant subalgebra $\g^{\Gamma}$, depends on the symmetry of the Dynkin diagram of $\g$. For $\g= A(2k,2l-1), A(2k-1,2l-1)(k,l)\neq(1,1), C(l+1), D(k+1,l)$ and $G(3)$, we consider the Dynkin diagram as described in Table 4 of  \cite{vdL}. It is clear from the Table in \cite{vdL} that, the automorphism $\sigma$ has order $2$. Hence $\Gamma\cong\mathbb{Z}_{2}$. Thus, we have $m=2$ in equation(\ref{eq:7}), and this gives us the following $\mathbb{Z}_{2}$ gradation of $\g$ 
\begin{equation*}
    \g=\g_{0}\oplus\g_{1}.
\end{equation*}
Here $\g_{0}=\g^{\Gamma}$ is the eigenspace corresponding to the eigenvalue $1$ and $\g_{1}$ is the eigenspace corresponding to the eigenvalue $-1$. Moreover, $\g_{1}$ is an irreducible $\g_{0}$ module. We obtain $\g^{\Gamma}$ using the folding technique on the Dynkin diagram of $\g$. The resulting fixed point Lie superalgebras are as given in table \ref{tab:1}.
\begin{table}[ht]
    \centering
    \begin{tabular}{|c|c|}
          \hline
    Superalgebra $\g$ & Invariant subalgebra $\g^{\Gamma}$\\   
    \hline
        $A(2k, 2l-1)$ & $B(k,l)$  \\
        $A(2k-1, 2l-1)$ & $D(k,l)$\\
        $C(l+1)$ & $B(0,l)$\\
        $D(k+1,l)$ & $B(k,l)$\\
        $G(3)$ & $D(2,1;\frac{-3}{4}$)\\
    \hline
    \end{tabular}
    \caption{}
    \label{tab:1}
\end{table}

Since we obtain these subsuperalgebras using the folding technique, they are singular invariant subsuperalgebras of $\g$. For $\g=B(k,l)$ and $F(4)$, the Dynkin diagrams do not admit non trivial symmetries. Hence, the diagram automorphism group is trivial and so $\g^{\Gamma}=\g$. Therefore, they are regular subalgebras of $\g$. We choose the root systems for $\g$ such that we obtain a distinguished root system for $\g^{\Gamma}$. This is necessary to arrive at the $\mathbf{C}$ condition discussed in the next section.

Hereafter, when we consider $\g$, we mean \[\g=A(2k,2l-1), A(2k-1,2l-1)((k,l)\neq(1,1)), C(l+1), D(k+1,l), B(k,l), F(4)\; \mbox{or}\; G(3)\] as taken above, unless otherwise stated. This would mean that 
\[\g^{\Gamma}=B(k,l),D(k,l), B(0,l), F(4)\; \mbox{or}\;\; D\left(2,1;\frac{-3}{4}\right).\]  
Consider
$\g=\g_{\bar{0}}\oplus \g_{\bar{1}}$ to be a simple basic classical Lie superalgebra having the triangular decomposition, induced by simple root system. So it is generated by $X_{\alpha}^{+}\in\g^{\alpha},~X_{\alpha}^{-}\in\g^{-\alpha}$ such that $[X_{\alpha}^{+},X_{\alpha}^{-}]=H_{\alpha}$ for all $\alpha$ in the simple root system of $\g$. Here $\g_{\bar{0}}$ is reductive, so for each even root $\{X_{\alpha}^{+}, X_{\alpha}^{-}, H_{\alpha}\}$ forms a $\mathfrak{sl}(2)$-triple. We denote the irreducible highest weight $\g$-module with highest weight $\lambda^{'}$ by $L_{\mathfrak{b}}(\lambda^{'}).$ Define \label{pg:4} $P^+=\{\lambda^{'} \in\h^* \mid L_{\mathfrak{b}}(\lambda^{'})~ \mbox{is finite dimensional}\}.$ Note that $\lambda^{'}(H_{\alpha}) \in \mathbb{N}$ for all even simple roots of $\g.$ 
\smallskip

Now $\g^{\Gamma}$, as described above, is also a basic classical Lie superalgebra (type $\textbf{II}$), it is generated by $x_{\alpha}^{+}\in(\g^{\Gamma})_{\alpha}$, $x^{-}_{\alpha}\in(\g^{\Gamma})_{-\alpha}$ such that $[x^{+}_{\alpha},x^{-}_{\alpha}]=h_{\alpha},~h_{\alpha}\in\h_{\Gamma}$ for all $\alpha\in \Delta$. We can also observe that $\g_{\bar{0}}^{\Gamma}$ is a finite dimensional semisimple Lie algebra, and hence, for each $\alpha\in R_{\bar{0}}$, there exist elements $x^{+}_{\alpha}\in(\g_{\bar{0}}^{\Gamma})_{\alpha}$, $x_{\alpha}^{-}\in(\g_{\bar{0}}^{\Gamma})_{-\alpha}$ and $h_{\alpha}\in\h_{\Gamma}$ such that the subalgebra generated by $\{x^{+}_{\alpha},x_{\alpha}^{-},h_{\alpha}\}$ is isomorphic to $\mathfrak{sl}(2)$. Thus, we get the following relation
\begin{equation}
    [x^{+}_{\alpha},x_{\alpha}^{-}]=h_{\alpha}, \quad [h_{\alpha},x_{\alpha}^{+}]=2x_{\alpha}^{+}, \quad [h_{\alpha},x_{\alpha}^{-}]=-2x_{\alpha}^{-} \quad \forall~\alpha\in R_{\bar{0}}.
\end{equation}
Let $L_\mathfrak{b}(\lambda)$ denote the unique irreducible $\g^{\Gamma}$-module of highest weight $\lambda$ and set
\begin{equation*}
\Lambda^+=\Lambda^+(\mathfrak{b})=\{\lambda\in\h^*_{\Gamma}|~L_\mathfrak{b}(\lambda)~ \mbox{is finite dimensional}\}.
\end{equation*}
Here $\mathfrak{b}=\h_{\Gamma}\oplus \n^+_{\Gamma}$ denotes the Borel subalgebra of $\g^{\Gamma}$. The existence of such an $L_{\mathfrak{b}}(\lambda)$ can be guaranteed from Lemma $\ref{lem:3}$ and $\ref{lem:4}$.

 For $\lambda\in\Lambda^{+}$, let $L_{\mathfrak{b}}(\lambda)$ be the finite dimensional $\g^{\Gamma}$ module. So, $L_{\mathfrak{b}}(\lambda)$ is a finite dimensional $\g_{\bar{0}}^{\Gamma}$ module and hence, $\lambda(h_{\alpha})\in\mathbb{N}$ for $\alpha\in\Delta_{\bar{0}}$.
\subsection{The ${\bf C}$ Condition}
 Let $-\theta$ be the lowest root of $\g^{\Gamma}$. Then the $\mathbf{C}$ condition is as follows:
 \[\mathbf{C}:-\theta \mbox{ is a root of} \;\g_{\bar{0}}^{\Gamma}.\]
  In this paper, we are interested in a triangular decomposition satisfying condition $\mathbf{C}$.
 In order to achieve this we choose the triangular decomposition for $\g^{\Gamma}$ such that the underlying simple root system is a distinguished root system as mentioned above. Once we have fixed the distinguished triangular root system for $\g^{\Gamma}$, we construct the triangular decomposition for $\g$ so that the two decompositions are consistent with each other, as mentioned in Section \ref{sec:1}.

 Let $\Delta_{dis}=\{\gamma_1,\ldots,\gamma_n\}$ be the set of distinguished simple roots for $\g^{\Gamma}$ and let $\gamma_s$ denote the unique odd root in $\Delta_{dis}$. With this simple root system, we can define a $\mathbb{Z}$-gradation for $\g^{\Gamma}$. Since $\g^{\Gamma}$ is one of the type $\textbf{II}$ basic classical Lie superalgbera, the $\mathbb{Z}$-gradation is as follows \cite{BCM19}:
 \begin{equation*}
     \g_{\bar{0}}^{\Gamma}=(\g^{\Gamma})_{-2}\oplus(\g^{\Gamma})_{0}\oplus(\g^{\Gamma})_{2}\quad\mbox{and}\quad\g_{\bar{1}}^{\Gamma}=(\g^{\Gamma})_{-1}\oplus(\g^{\Gamma})_{1}.
 \end{equation*} 
 The induced triangular decomposition for $\g^{\Gamma}$ would be
 \begin{equation*}
     \g^{\Gamma}=\n_{\Gamma}^{-}(\Delta_{dis})\oplus\h_{\Gamma}\oplus\n_{\Gamma}^{+}(\Delta_{dis}) \quad \mbox{where} \quad \n_{\Gamma}^{\pm}(\Delta_{dis})=(\n_{\Gamma}^{\pm})_{0}\oplus\biggl(\bigoplus_{i>0}(\g^{\Gamma})_{\pm i}\biggl).
 \end{equation*}
 \begin{lemma}
     Consider the Lie superalgebra $\g$. Let $\g^{\Gamma}$ be the fixed subalgebra and let $\Delta_{dis}$ be the set of distinguished simple root system for $\g^{\Gamma}$. Then $\g^{\Gamma}$ satisfies the $\mathbf{C}$ condition.
 \end{lemma}
 \begin{proof}
     $\g^{\Gamma}$ has the $\mathbb{Z}$-gradation given as $\g^{\Gamma}=\g^{\Gamma}_{-2}\oplus\g^{\Gamma}_{-1}\oplus\g^{\Gamma}_{0}\oplus\g^{\Gamma}_{1}\oplus\g^{\Gamma}_{2}$. Let $x^{-}_{\theta}$ be the lowest weight of $\g^{\Gamma}_{-2}$ as a $\g_{0}^{\Gamma}$-module. Then $[x^{-}_{\theta},(\n^{-}_{\Gamma})_{0}]=0$, $\n^-_{\Gamma}(\Delta_{dis})=\g^{\Gamma}_{-2}\oplus\g^{\Gamma}_{-1}\oplus(\n^-_{\Gamma})_{0}$ and $[\g^{\Gamma}_{-2},\g^{\Gamma}_{-2}\oplus\g^{\Gamma}_{-1}]=0$. Hence in particular $[x^{-}_{\theta},\g^{\Gamma}_{-2}\oplus\g^{\Gamma}_{-1}]=0$. That is, the lowest weight of $\g^{\Gamma}_{-2}$ as a $\g^{\Gamma}_{0}$ module is also the lowest weight of $\g^{\Gamma}$. Hence $\g^{\Gamma}_{-\theta}\subseteq\g^{\Gamma}_{-2}$ where $-\theta$ is the lowest root. Since $\g_{-2}^{\Gamma}\subseteq\g^{\Gamma}_{\bar{0}}$, we obtain that $-\theta$ is a root of $\g_{\bar{0}}^{\Gamma}$.
 \end{proof}
  \section{Twisted Global Weyl Modules}\label{sec:5}
  In this section we develop the theory of global Weyl modules for $(\g\otimes A)^{\Gamma}$, where $\g$ is the Lie superalgebra and $A$ is the commutative associative unital algebra. We also observe the universal property of the global Weyl module in the category we define. Let $\mathcal{I}$ be the full subcategory of the category of $(\g\otimes A)$-modules which when restricted to $\g_{\bar{0}}$ is finitely semisimple.
 \subsection{Category $\mathcal{I}^\Gamma$}\label{sub:1}
 We define $\mathcal{I}^\Gamma$ to be the full subcategory of the category of $(\g\otimes A)^\Gamma$-modules which, when restricted to $\g_{\bar{0}}^{\Gamma}$ is finitely semisimple. From Lemma $\ref{lem:1}$ and $\ref{lem:2}$, we can see that this category is going to be closed taking submodules, quotients, arbitrary direct sums, and finite tensor products.
 \begin{defn}(The module $\bar{V}(\lambda)$)
     For $\lambda\in \Lambda^+$, we define $\bar{V}(\lambda)$ to be the $\g^{\Gamma}$-module generated by a vector $v$ with defining relations
     \begin{equation}\label{eq:1}
         \n^+_{\Gamma}v=0 \quad hv=\lambda(h)v \quad (x_{\alpha}^{-})^{\lambda(h_\alpha)+1}v=0 \quad \forall h\in \h_{\Gamma}, \alpha \in \Delta_{\bar{0}}.
     \end{equation}
\end{defn}
\begin{prop}
For all $\lambda \in \Lambda^+$, the module $\bar{V}(\lambda)$ is finite dimensional.
 \end{prop}
 \begin{proof}
 Let $x_1, \ldots, x_m,y_1, \ldots, y_n$ be a homogeneous basis for $\g^{\Gamma}$ with $x_1, \ldots, x_m \in \g_{\bar0}^{\Gamma}$ and  $y_1, \ldots, y_n \in \g_{\bar1}^{\Gamma}$. Then by Lemma \ref{lem:0} the monomials
\[x_1^{a_1}\cdots x_m^{a_m} y_1^{b_1}\cdots y_n^{b_n}, \quad a_1, \ldots, a_m \geq 0, \quad \mbox{and}\quad b_1, \ldots, b_n \in \{0, 1\},\]
form a basis of $\mathbf{U}(\g^{\Gamma})$. Since $\{y_1^{b_1}\cdots y_n^{b_n} \mid b_j=0,1\}$ is a finite set, to prove that $\bar{V}(\lambda)$ is finite dimensional, it is enough to show that $U(\g_{\bar{0}}^{\Gamma})v$ is finite dimensional. 

Consider an irreducible $\g_{\bar{0}}^{\Gamma}$-module $L(\lambda)$ with highest weight $\lambda \in \Lambda^+$. Since $\g^{\Gamma}$ is one of the types $\textbf{II}$ superalgebra, we know that $\g_{\bar{0}}^{\Gamma}$ is a semisimple Lie algebra. Since $\lambda(h_\alpha) \in \mathbb{N}$ for all $\alpha \in \Delta_{\bar{0}}$, we have that $L(\lambda)$ is finite dimensional. Moreover, $L(\lambda)$ is isomorphic to the $\g_{\bar{0}}^{\Gamma}$ module generated by a vector $u_{\lambda}$ with the defining relations
     \begin{equation}
         x_{\alpha}^{+}u_\lambda=0 \quad hu_\lambda=\lambda(h)u_\lambda \quad (x_{\alpha}^{-})^{\lambda(h_\alpha)+1}u_\lambda=0 \quad \forall h\in \h_{\Gamma}, \alpha \in \Delta_{\bar{0}}.
     \end{equation}
     Let $V'=\mathbf{U}(\g_{\bar{0}}^{\Gamma})v \subset\bar{V}(\lambda)$ be the $\g_{\bar{0}}^{\Gamma}$-submodule of $\bar{V}(\lambda)$ generated by $v$. Then the map given by 
     \begin{equation}
         f:L(\lambda)\longrightarrow V', \quad xu_\lambda \longrightarrow xv \quad \forall x \in \mathbf{U}(\g_{\bar{0}}^{\Gamma})
     \end{equation}
     is a well defined epimorphism of $\g_{\bar{0}}^{\Gamma}$ modules. Thus, $V'$ is finite dimensional. 
 \end{proof}
 The proof for the following proposition is similar to the proof of Lemma 2.8 in \cite{CLS19}.
 \begin{prop}\label{prop:4'}
Suppose $V$ is a finite dimensional $\g^{\Gamma}$ module generated by a highest weight vector of weight $\lambda\in \Lambda^+$. Then there exists a unique submodule $W$ of $\bar{V}(\lambda)$ such that $\bar{V}(\lambda)/W\cong V$ is as $\g^{\Gamma}$ modules.
 \end{prop}
 Every simple finite dimensional $\g^{\Gamma}$-module is a highest weight module with some highest weight $\lambda$. Hence using Proposition \ref{prop:4'} any simple finite dimensional module is a quotient of $\bar{V}({\lambda}).$
 For a $\g^{\Gamma}$ module $V$, define 
 \begin{equation}
     P{^\Gamma}(V)=\mathbf{U}((\g\otimes A)^\Gamma)\otimes_{\mathbf{U(\g^{\Gamma})}}V. 
 \end{equation}
 We can view $V$ as a $\g^{\Gamma}$-submodule of $P^\Gamma(V)$ via the natural identification $V\cong \mathbb{K}\otimes V\subset P^\Gamma(V)$.
\begin{lemma}[\cite{FMS15}]\label{lem:14}
If $V$ is the direct sum of irreducible finite dimensional $\g$ modules (where $\g$ is a reductive Lie algebra), then so is the tensor algebra $T(V)=\bigoplus_{n=0}^{\infty}V^{\otimes n}$.
\end{lemma}
 \begin{lemma}
 Let $V\in C_{(\g^{\Gamma},\g_{\bar{0}}^{\Gamma})}$, that is, a  $\g^{\Gamma}$-module whose restriction to $\g_{\bar{0}}^{\Gamma}$ is finitely semisimple. Then $P^\Gamma(V)\in\mathcal{I}^\Gamma$. 
 \end{lemma}
 \begin{proof}
Since $\g_{\bar{0}}^{\Gamma}$ is a semisimple Lie algebra, $\g$ is a completely reducible $\g_{\bar{0}}^{\Gamma}$-module. Consider the action of $\g_{\bar{0}}^{\Gamma}$ on $\g\otimes A$ given by the adjoint action on the first factor. As $\g_{\bar{0}}^{\Gamma}$ module, $\g\otimes A\cong\g$. Hence, it follows that $\g\otimes A$ can be written as the direct sum of irreducible finite dimensional $\g_{\bar{0}}^{\Gamma}$ modules. The action of $\g_{\bar{0}}^{\Gamma}$ on $(\g\otimes A)^{\Gamma}$ is also by the left adjoint multiplication with $\g$. Our claim is that $(\g\otimes A)^\Gamma$ is invariant under action of $\g_{\bar{0}}^{\Gamma}$. For this, we need to show that, $\gamma(x(u\otimes a))=x(u\otimes a)$, where $\gamma\in\Gamma$, $x\in\g_{\bar{0}}^{\Gamma}$ and $u\otimes a\in (\g\otimes A)^{\Gamma}$. Let $[x,u]\otimes a\in\g_{\bar{0}}^{\Gamma}(\g\otimes A)^{\Gamma},  x\in\g_{\bar{0}}^{\Gamma}, (u\otimes a)\in(\g\otimes A)^{\Gamma}$.
\begin{align*}
  \gamma([x,u]\otimes a)&=\gamma([x,u])\otimes\gamma a \\
  [\gamma x,\gamma u]\otimes\gamma a&=[x,\zeta^{s}u]\otimes\zeta^{-s}a\quad \\
  &=[x,u]\otimes a.
\end{align*}
The second equality holds $x\in\g_{\bar{0}}^{\Gamma}~\mbox{and}~u\otimes a\in(\g\otimes A)^{\Gamma}$.
Hence $(\g\otimes A)^{\Gamma}$ can be written as the direct sum of irreducible finite dimensional $\g_{\bar{0}}^{\Gamma}$-modules. Using Lemma \ref{lem:14}, $T((\g\otimes A)^\Gamma)$ and hence $\mathbf{U}(\g\otimes A)^\Gamma$ are the direct sum of irreducible finite dimensional $\g_{\bar{0}}^{\Gamma}$ modules. Hence $P^\Gamma(V)\in\mathcal{I}^\Gamma$.
 \end{proof}
 \begin{prop}
If $\lambda\in \Lambda^+$, then $P^\Gamma(\bar{V}(\lambda))$ is generated as a $\mathbf{U}(\g\otimes A)^\Gamma$ module by an element $v_\lambda$ with the defining relations
\begin{equation}\label{eq:2}
         \n^+_{\Gamma}v_\lambda=0 \quad hv_\lambda=\lambda(h)v_\lambda \quad (x_{\alpha}^{-})^{\lambda(h_\alpha)+1}v_\lambda=0 \quad \forall h\in \h_{\Gamma}, \alpha \in \Delta_{\bar{0}}.
\end{equation}
\end{prop}
 \begin{proof}
Since $v\in \bar{V}(\lambda)$ satisfies the relation \eqref{eq:2}, its image $v_\lambda=1\otimes v$ in $P^\Gamma(\bar{V}(\lambda))$ also satisfies the above relation. Now to show that these are the only relations satisfied, consider $W$ to be a $(\g\otimes A)^\Gamma$-module generated by a vector $w$ with defining relation \eqref{eq:2}. Then we have the surjective homomorphism of $(\g\otimes A)^\Gamma$-modules, $\Pi_1:W\longrightarrow P^\Gamma(\bar{V}(\lambda))$, which maps $w$ to $v_\lambda$. Since $w\in W$ satisfies the relation \eqref{eq:2}, from Proposition \ref{prop:4'}, there will exist a $\g^\Gamma$ submodule of $W$ which is isomorphic to the quotient of $\bar{V}(\lambda)$. Thus there exist an epimorphism
\begin{equation*}
         \Pi_2:P^\Gamma(\bar{V}(\lambda))\longrightarrow W, \quad u_1\otimes_{\mathbf{U}(\g^{\Gamma})}u_2v\longrightarrow u_{1}u_{2}w, \quad u_1\in \mathbf{U}((\g\otimes A)^\Gamma), u_2\in \mathbf{U}(\g^\Gamma).
\end{equation*}
Since $\Pi_1=\Pi_{2}^{-1}$, we have $W\cong P^\Gamma(\bar{V}(\lambda))$.
 \end{proof}
 For $\nu\in \Lambda^+$ and $V\in\mathcal{I}^\Gamma$, let $V^\nu$ be the unique maximal $(\g\otimes A)^\Gamma$-module quotient of V, such that the weights of $V^\nu$ lie in $\nu-Q^+_{\Gamma}$ \label{pg:6}, where $Q^+_{\Gamma}=\Sigma_{\alpha\in \Delta}\mathbb{N}\alpha$ is the positive root lattice of $\g^{\Gamma}$. In other words,
 \begin{equation*}
     V^\nu=V/\Sigma_{\mu\notin\nu-Q^+_{\Gamma}}\mathbf{U}((\g\otimes A)^\Gamma)V_\mu.
\end{equation*}
A morphism $f:V\longrightarrow W$ of objects in $\mathcal{I}^\Gamma$ induces a morphism $f^\nu:V^\nu\longrightarrow W^\nu$. Let $\mathcal{I}^{\Gamma}_{\nu}$ denote the full subcategory of $\mathcal{I}^\Gamma$ whose objects are those $V\in\mathcal{I}^\Gamma$ such that $V=V^\nu$. Therefore, $\mathcal{I}_{\lambda}^{\Gamma}$ \label{pg:5} is the full subcategory of $\mathcal{I}^{\Gamma}$ whose elements have maximum weight $\lambda$. When $\Gamma$ is trivial, $\mathcal{I}_{\lambda}$ is the category of $(\g\otimes A)$- modules which when restricted to $\g_{\bar{0}}$ is finitely semisimple and has maximum weight $\lambda$.
\begin{defn}(Twisted Global Weyl Module)\label{def:4}
We define the global Weyl module associated to $\lambda\in \Lambda^+$ is 
\begin{equation*}
    W^\Gamma(\lambda)=P^\Gamma(\bar{V}(\lambda))^\lambda.
\end{equation*}
    Denote $w^{\Gamma}_{\lambda}$ as the image of $v_\lambda$ in $W^\Gamma(\lambda)$.
\end{defn}
\begin{prop}\label{prop:1}
    For $\lambda\in \Lambda^+$, the twisted global Weyl module $W^\Gamma(\lambda)$ is generated by $w_{\lambda}^{\Gamma}$ with defining relations
    \begin{equation}\label{eq:3}
        (\n^+\otimes A)^\Gamma w^{\Gamma}_{\lambda}=0, \quad hw^{\Gamma}_{\lambda}=\lambda(h)w^{\Gamma}_{\lambda},\quad (x_{\alpha}^{-})^{\lambda(h_\alpha)+1}w^{\Gamma}_{\lambda}=0, \quad \forall h\in\h_{\Gamma},\alpha\in\Delta_{\bar{0}}.
    \end{equation}
\end{prop}
\begin{proof}
Since the weights of $W^\Gamma(\lambda)$ lie in $\lambda-Q^+_{\Gamma}$, it follows that $(\n^+\otimes A)^\Gamma w^{\Gamma}_{\lambda}=0$. The remaining relations are also satisfied because they are satisfied by $v_\lambda$. To show that these are the only relations in $W^\Gamma(\lambda)$, we consider W to be the module generated by $w$ satisfying the relations in \eqref{eq:3}. Hence there will exist an epimorphism $\Pi_1:W\longrightarrow W^\Gamma(\lambda)$ sending $w$ to $w^{\Gamma}_{\lambda}$. Since relations \eqref{eq:3} imply relations \eqref{eq:1}, the vector $w\in W$ generates a $\g^\Gamma$-submodule of W isomorphic to the quotient of $\bar{V}(\lambda)$. Thus we have a surjective homomorphism 
    \begin{equation*}
        \Pi_2:P^\Gamma(\bar{V}(\lambda))\longrightarrow W, \quad u_1\otimes_{\mathbf{U}(\g^{\Gamma})}u_2v\longrightarrow u_{1}u_{2}w, \quad u_1\in \mathbf{U}((\g\otimes A)^\Gamma), u_2\in \mathbf{U}(\g)^\Gamma.
    \end{equation*}
    Since the $\g^\Gamma$ weights of W are bounded by $\lambda$, it follows that $\Pi_2$ induces a map $W^{\Gamma}(\lambda)\longrightarrow W$ inverse to $\Pi_1$.
\end{proof}
\begin{prop}\label{prop:2}
The twisted global Weyl module $W^\Gamma(\lambda)$ is the unique object of $\mathcal{I}^\Gamma$ up to isomorphism, that is generated by a highest weight vector of weight $\lambda$ and admits a surjective homomorphism to any object of $\mathcal{I}^\Gamma$ also generated by a highest weight vector of weight $\lambda$.
\end{prop}
\begin{proof}
Let $V\in\mathcal{I}^\Gamma$ be generated by a highest weight vector $v'$ of weight $\lambda$. Then
    \begin{equation*}
        (\n^+\otimes A)^{\Gamma}v'=0 \quad hv'=\lambda(h)v', \quad \forall h\in\h_{\Gamma}.
    \end{equation*}
 Since the $\g_{\bar{0}}^{\Gamma}$-module generated by $v'$ is finite dimensional we have that $(x_{\alpha}^{-})^{\lambda(h_\alpha)+1}v'=0$ for all $\alpha\in\Delta_{\bar{0}}$. Thus by Proposition $\ref{prop:1}$, there exists a surjective homomorphism $W^\Gamma(\lambda)\longrightarrow V$ such that $w^{\Gamma}_{\lambda}\mapsto v'$.
    
Suppose that $W$ is another object in $\mathcal{I}^\Gamma$ that is generated by the highest weight vector $w$ of weight $\lambda$ such that it admits a surjective homomorphism to any object of $\mathcal{I}^\Gamma$ also generated by the highest weight vector of weight $\lambda$, i.e, there exits a surjective homomorphism $\Pi_1:W\longrightarrow W^\Gamma(\lambda)$. It follows from PBW theorem that $W^\Gamma(\lambda)_\lambda=\mathbf{U}(\h\otimes A)^\Gamma w^{\Gamma}_{\lambda}$. The only elements of this weight space that generate $W^\Gamma(\lambda)$ are the $\mathbb{K}$ multiples of $w^{\Gamma}_{\lambda}$. After rescaling, we get $\Pi_1(w)=w^{\Gamma}_{\lambda}$. From the definition of $W$, we know that $w$ satisfies the relation \eqref{eq:3}. Thus there exists a homomorphism $\Pi_2:W^\Gamma(\lambda)\longrightarrow W$ sending $w^{\Gamma}_{\lambda}$ to $w$. It follows that $\Pi_1$ and $\Pi_2$ are mutually inverse homomorphisms and so $W\cong W^\Gamma(\lambda)$.
\end{proof}
\section{Twisted Weyl Functors}\label{sec:6}
Let $A$ be an associative commutative $\mathbb{K}$-algebra with unit and $\g$ be the $\mathbb{K}$-Lie superalgebra. Here, we define the algebra $\A$  which plays a crucial role in proving that the global Weyl modules are finitely generated. We also introduce the twisted Weyl functor and the twisted restriction functor.

For $\lambda\in \Lambda^+$, define 
\begin{equation*}
    \Ann_{(\g\otimes A)^\Gamma}(w_{\lambda}^{\Gamma})=\{u\in\mathbf{U}(\g\otimes A)^\Gamma\mid uw_{\lambda}^{\Gamma}=0\}
\end{equation*} and
\begin{equation*}
     \Ann_{(\h\otimes A)^\Gamma}(w_{\lambda}^{\Gamma})=Ann_{(\g\otimes A)^\Gamma}(w_{\lambda}^{\Gamma})\cap\mathbf{U}(\h\otimes A)^\Gamma.
\end{equation*}
Clearly $\Ann_{(\g\otimes A)^\Gamma}(w_{\lambda}^{\Gamma})$ is a left ideal of $\mathbf{U}(\g\otimes A)^\Gamma$. Since $\mathbf{U}(\h\otimes A)^\Gamma$ is a commutative algebra, $\Ann_{(\h\otimes A)^\Gamma}(w_{\lambda}^{\Gamma})$ is an ideal of $\mathbf{U}(\h\otimes A)^\Gamma$. Define the algebra $\A$ to be the quotient
\begin{equation}\label{eq:11}
    \A=\mathbf{U}(\h\otimes A)^\Gamma/\Ann_{(\h\otimes A)^\Gamma}(w_{\lambda}^{\Gamma}).
\end{equation}
By PBW theorem $W^\Gamma(\lambda)_\lambda=\mathbf{U}(\h\otimes A)^\Gamma w_{\lambda}^{\Gamma}$. Thus the unique homomorphism of $\mathbf{U}(\h\otimes A)^\Gamma$-modules satisfying
\begin{equation*}
     f:\mathbf{U}(\h\otimes A)^\Gamma\longrightarrow W^\Gamma(\lambda)_\lambda, \quad f(1)=w_{\lambda}^{\Gamma}
\end{equation*}
induces an isomorphism of $(\h\otimes A)^{\Gamma}$-modules between $W^\Gamma(\lambda)_\lambda$ and $\A$, i.e, $W^\Gamma(\lambda)_\lambda\cong\A$ as right $\A$-modules.
\begin{lemma}\label{lem:8}
    For all $\lambda\in \Lambda^+$ and $V\in \mathcal{I}^{\Gamma}_{\lambda}$, $(\Ann_{(\h\otimes A)^\Gamma}(w_{\lambda}^{\Gamma}))V_\lambda=0$. 
\end{lemma}
\begin{proof}
    Let $v\in V_\lambda$ and $W=\mathbf{U}(\g\otimes A)^{\Gamma}v$. Since $V$ is an object of $\mathcal{I}^{\Gamma}_{\lambda}$, the submodule $W$ is also an object in $\mathcal{I}^{\Gamma}_{\lambda}$. Moreover, since $v\in V_\lambda$, we have $(\n^+\otimes A)^{\Gamma}v=0$ and $hv=\lambda(h)v$ for all $h\in\h_{\Gamma}$. Thus by the universal property of $W^{\Gamma}(\lambda)$, there exists a unique (surjective) homomorphism of $(\g\otimes A)^{\Gamma}$-modules $\pi:W^{\Gamma}(\lambda)\rightarrow W$ satisfying $\pi(w_{\lambda}^{\Gamma})=v$. Since $\pi$ is a homomorphism of $(\g\otimes A)^{\Gamma}$-modules and $uw_{\lambda}^{\Gamma}=0$ for all $u\in \Ann_{(\h\otimes A)^\Gamma}(w_{\lambda}^{\Gamma})$, we conclude that $uv=\pi(uw_{\lambda}^{\Gamma})=0$ for all $u\in \Ann_{(\h\otimes A)^\Gamma}(w_{\lambda}^{\Gamma})$.
\end{proof}
Since $\mathbf{U}(\h\otimes A)$ is a commutative algebra, so is its subalgebra $\mathbf{U}(\h\otimes A)^{\Gamma}$ and hence every left $\mathbf{U}(\h\otimes A)^{\Gamma}$ module is also a right $\mathbf{U}(\h\otimes A)^{\Gamma}$ module. From Lemma $\ref{lem:8}$ the left action of $\mathbf{U}(\g\otimes A)^{\Gamma}$ on any object V of $\mathcal{I}_{\lambda}^{\Gamma}$ induces a left as well as right action of $\mathbf{A}_{\lambda}^{\Gamma}$ on $V_{\lambda}$. Since $W^{\Gamma}(\lambda)$ is an object of $\mathcal{I}_{\lambda}^{\Gamma}$ generated by $w_{\lambda}^{\Gamma}\in W^{\Gamma}(\lambda)_{\lambda}$ as a left $\mathbf{U}(\g\otimes A)^{\Gamma}$-module, we have a right action of $\mathbf{A}_{\lambda}^{\Gamma}$ on $W^{\Gamma}(\lambda)_{\lambda}$ that commutes with the left $\mathbf{U}(\g\otimes A)^{\Gamma}$ action, i. e.  
\begin{equation*}
    (uw_{\lambda}^{\Gamma})u'=uu'w_{\lambda}^{\Gamma} \quad \forall \quad u\in\mathbf{U}(\g\otimes A)^{\Gamma}\quad and \quad u'\in\mathbf{U}(\h\otimes A)^{\Gamma} .
\end{equation*}
Thus with these actions, $W^\Gamma(\lambda)$ is a $(\mathbf{U}(\g\otimes A)^{\Gamma},\mathbf{A}_{\lambda}^{\Gamma})$-bimodule (\cite{FMS15}).
\subsection{Category $\A$-mod}\label{sub:2}
Given $\lambda\in \Lambda^+$, let $\mathbf{A}_{\lambda}^{\Gamma}-mod$ denote the category of left $\mathbf{A}_{\lambda}^{\Gamma}$-modules and let $M\in\mathbf{A}_{\lambda}^{\Gamma}-mod$. Since $W^\Gamma(\lambda)$ be a finitely semisimple $\g_{\bar{0}}^{\Gamma}$-module and the action of $\g_{\bar{0}}^{\Gamma}$ on $W^\Gamma(\lambda)\otimes_{\mathbf{A}_{\lambda}^{\Gamma}} M$ is given by left multiplication, we have that $W^\Gamma(\lambda)\otimes_{\mathbf{A}_{\lambda}^{\Gamma}}M$ is finitely semisimple $\g_{\bar{0}}^{\Gamma}$ module. Since $id:W^\Gamma(\lambda)\rightarrow W^\Gamma(\lambda)$ is an even homomorphism of the $(\g\otimes A)^\Gamma$-modules, for every $M, M'\in \mathbf{A}_{\lambda}^{\Gamma}-mod$ and $\tilde f\in \Hom_{\mathbf{A}_{\lambda}^{\Gamma}}(M,M')$,
\begin{equation*}
    id\otimes \tilde f:W^\Gamma(\lambda)\otimes_{\mathbf{A}_{\lambda}^{\Gamma}} M\rightarrow W^\Gamma(\lambda)\otimes_{\mathbf{A}_{\lambda}^{\Gamma}} M'
\end{equation*}
is a homomorphism of $(\g\otimes A)^\Gamma$-modules.
\begin{defn}(Twisted Weyl Functor)\label{def:5}
  The twisted Weyl functor $\mathbf{W}_\lambda^{\Gamma}$  associated with $\lambda\in \Lambda^+$ is defined as 
\begin{equation*}
\mathbf{W}_\lambda^{\Gamma}:\mathbf{A}_{\lambda}^{\Gamma}-\mathrm{mod}\rightarrow\mathcal{I}_{\lambda}^{\Gamma}, \quad \mathbf{W}_\lambda^{\Gamma}M=W^\Gamma(\lambda)\otimes_{\mathbf{A}_{\lambda}^{\Gamma}} M, \quad \mathbf{W}_\lambda^{\Gamma} \tilde f=id\otimes \tilde f
\end{equation*}  
for all $M, M'$ in $\mathbf{A}_{\lambda}^{\Gamma}-\mathrm{mod}$ and $\tilde f\in \mathrm{Hom}_{\mathbf{A}_{\lambda}^{\Gamma}}(M,M')$.
\end{defn}
Given $\lambda\in \Lambda^+$, there is an isomorphism of $(\g\otimes A)^\Gamma$-modules $\mathbf{W}_\lambda^{\Gamma}\mathbf{A}_{\lambda}^{\Gamma}\cong W^\Gamma(\lambda)$. Also, for all $\mu\in\h^{*}_{\Gamma}$ and M in $\mathbf{A}_{\lambda}^{\Gamma}-\mathrm{mod}$ we have
\begin{equation}
(\mathbf{W}_{\lambda}^{\Gamma}M)_\mu=W^\Gamma(\lambda)_\mu\otimes_{\mathbf{A}_{\lambda}^{\Gamma}} M.
\end{equation}
\begin{defn}(Twisted Restriction Functor $\mathbf{R}_{\lambda}^{\Gamma}$)\label{def:6}
From Lemma \ref{lem:8}, we see that, for $V\in\mathcal{I}_{\lambda}^{\Gamma}$, the action of $\mathbf{U}(\g\otimes A)^{\Gamma}$ on $V$ induces an action of $\A$ on $V_{\lambda}$. Given $\pi\in \mathrm{Hom}_{\mathcal{I}_{\lambda}^{\Gamma}}(V,V')$, the restriction of $\pi$ to $V_{\lambda}$ induces a homomorphism of $\A$-modules $\pi_{\lambda}:V_{\lambda}\rightarrow V_{\lambda}'$. This is the motivation for the definition of $\mathbf{R}_{\lambda}^{\Gamma}$. For $\lambda\in\Lambda^{+}$, and for $V\in\mathcal{I}_{\lambda}^{\Gamma}$, we define the functor $\mathbf{R}_{\lambda}^{\Gamma}:\mathcal{I}_{\lambda}^{\Gamma}\rightarrow\A-\mathrm{mod}$, such that
\begin{equation*}
 \mathbf{R}_{\lambda}^{\Gamma}V=V_{\lambda}, \quad \mathbf{R}_{\lambda}^{\Gamma}\pi=\pi_{\lambda}.
\end{equation*}
\end{defn}
\section{The Structure of Global Weyl Modules}\label{sec:7}
In this section we assume that $A$ is finitely generated. This assumption, is important in proving that the global Weyl module is finitely generated, which is the main result here. Under the action of a cyclic group $\Gamma$, consider the gradation $A=\oplus_{s=0}^{m-1}A_{s}$of $A$ as mentationed earlier. 
\begin{lemma}[\cite{Bourbaki85}]\label{lem:5}
The algebra $A^{\Gamma}$, which is the subalgebra of $A$ consisting of fixed points, is finitely generated as an algebra and $A_{s},~s\in\{0,1,\ldots,m-1\}$, is finitely generated as an $A^{\Gamma}$ module.
\end{lemma}
\begin{lemma}
If $\lambda\in \Lambda^+$ and $\alpha\in R^{+}_{\bar{0}}$, then $(x_{\alpha}^{-})^{\lambda(h_\alpha)+1}w_{\lambda}^{\Gamma}=0$.
\end{lemma}
\begin{proof}
The vector $(x_{\alpha}^{-})^{\lambda(h_{\alpha})+1}w_{\lambda}^{\Gamma}$ has weight $\lambda-(\lambda(h_{\alpha})+1)\alpha$. Since the global Weyl module $W^{\Gamma}(\lambda)$ is an element of the category $\mathcal{I}_{\lambda}^{\Gamma}$, it can be written as the direct sum of finite dimensional irreducible $\g^{\Gamma}_{\bar{0}}$ modules. Hence the weights of $W^{\Gamma}(\lambda)$ remains invariant under the action of the Weyl group of $\g_{\bar{0}}^{\Gamma}$. Let $s_{\alpha}$ denote the reflection associated to the root $\alpha$. Then
    \begin{align*}
        s_{\alpha}(\lambda-(\lambda(h_{\alpha})+1)\alpha)&=(\lambda-(\lambda(h_{\alpha})+1)\alpha)-2\frac{(\alpha,(\lambda-(\lambda(h_{\alpha})+1)\alpha)}{(\alpha,\alpha)}\alpha\\
        &=(\lambda-(\lambda(h_{\alpha})+1)\alpha)-(\lambda-(\lambda(h_{\alpha})+1)\alpha)(h_{\alpha})\alpha \\
        &=\lambda-\lambda(h_{\alpha})\alpha-\alpha-\lambda(h_{\alpha})\alpha+\lambda(h_{\alpha})\alpha(h_{\alpha})\alpha+\alpha(h_{\alpha})\alpha \\
            &=\lambda-2\lambda(h_{\alpha})\alpha-\alpha+2\lambda(h_{\alpha})\alpha+2\alpha=\lambda+\alpha.
        \end{align*}
But the weights of $W^{\Gamma}(\lambda)$ are bounded above by $\lambda$. Hence $(x_{\alpha}^{-})^{\lambda(h_\alpha)+1}w_{\lambda}^{\Gamma}=0$.
\end{proof}
Given $a\in A^\Gamma$ and $\alpha\in R^{+}_{\bar{0}}$, define the power series in an indeterminate $u$ and with coefficients in $\mathbf{U}(\h_{\Gamma}\otimes A^{\Gamma})\subseteq\mathbf{U}(\h\otimes A)^{\Gamma}$ as follows:
\begin{equation}
    p(a,\alpha) = \exp \left( -\sum_{i=1}^{\infty} \frac{h_{\alpha} \otimes a^{i}}{i} u^{i} \right)
\end{equation}
where $h_{\alpha}\in\h_{\Gamma}$. For $i\geq 0$, let $p(a,\alpha)_i$ denote the coefficient of $u^i$ in $p(a,\alpha)$ and notice that $p(a,\alpha)_0=1$. The proof of the following Lemma follows from \cite[Lemma 7.5]{Garland1978}.
\begin{lemma}
Let $m\in\N$, $a\in A^{\Gamma}$ and $\alpha\in R_{\bar{0}}^+$. Then
\begin{equation*}
        (x^{+}_{\alpha}\otimes a)^{m}(x_{\alpha}^-)^{m+1}-(-1)^m\sum_{i=0}^{m}(x_{\alpha}^-\otimes a^{m-i})p(a,\alpha)_{i}\in\mathbf{U}(\mathfrak{sl}_{\alpha}\otimes A^{\Gamma})((\g^{\Gamma})_{\alpha} \otimes A^{\Gamma})
        \end{equation*}
where $\mathbf{U}(\mathfrak{sl}_{\alpha}\otimes A^{\Gamma})((\g^{\Gamma})_{\alpha}\otimes A^{\Gamma})$ denotes the left ideal of $\mathbf{U}(\mathfrak{sl}_{\alpha}\otimes A^{\Gamma})$ generated by $((\g^{\Gamma})_{\alpha}\otimes A^{\Gamma})=\mathbb{K} x^{+}_{\alpha}\otimes A^{\Gamma}$.
\end{lemma}
The following three Lemmas will be used to proof the main result of the section, Theorem \ref{thm:1}.
\begin{lemma}\label{lem:6}
Let $\lambda\in \Lambda^+$, $\alpha\in R_{\bar{0}}^+$ and $a_{1},\cdots,a_{t}\in A^{\Gamma}$. Then for every $m_{1},\ldots, m_{t}\in \N$ we have
\[(x_{\alpha}^-\otimes a_{1}^{m_1}\cdots a_{t}^{m_t})w_{\lambda}^{\Gamma}\in \mathrm{span}_{\mathbb{K}}\{(x_{\alpha}^-\otimes a^{l_1}_{1}\cdots a^{l_t}_{t})w_{\lambda}^{\Gamma}A_{\lambda}^{\Gamma}~|~0\leq l_{1}, l_{2},\ldots,l_{t}<\lambda(h_{\alpha}),~h_{\alpha}\in\h_{\Gamma}\}.\]
In particular, $(\g_{\bar{0}}^{\Gamma}\otimes A^{\Gamma})w_{\lambda}^{\Gamma}$ is finitely generated right $A_{\lambda}^{\Gamma}$-module.
\end{lemma}
\begin{proof}
The proof follows similarly as in \cite[Lemma 7.3]{BCM19}.
\end{proof}
\begin{lemma}{\label{lem:7}}
Let $\lambda\in\ \Lambda^{+}, \alpha\in R^{+}_{\bar{0}}, x_{1},\ldots, x_{k}\in n^{+}_{\Gamma}$ and $a_{1},\ldots, a_{t}\in A^{\Gamma}$. Then, for all $m_{1},\ldots, m_{t}\in \mathbb{N}$, the element $([x_{1},[x_{2},\cdots[x_{k},x_{\alpha}^{-}]\cdots]]\otimes a_{1}^{m_{1}}\cdots a_{t}^{m_{t}})w_{\lambda}^{\Gamma}$ is in
\begin{equation*}
        \mathrm{span}_{\mathbb{K}}\{([x_{1},[x_2,\cdots[x_{k},x_{\alpha}^{-}]\cdots]]\otimes a_{1}^{l_{1}}\cdots a_{t}^{l_{t}}w_{\lambda}^{\Gamma}A_{\lambda}^{\Gamma}~|~0\leq l_{1},l_{2},\cdots,l_{t}<\lambda(h_\alpha),~h_{\alpha}\in\h_{\Gamma}\}.
    \end{equation*}
\end{lemma}
\begin{proof}
 The proof follows similarly as in \cite[Lemma 7.4]{BCM19}.
\end{proof}

\begin{lemma}{\label{lem:9}}
As a right $\A$-module, $(\n_{\bar{1}}^{-}\otimes A)^{\Gamma}w_{\lambda}^{\Gamma}$ is finitely generated.
\end{lemma}
\begin{proof}
The automorphism group $\Gamma$ acting on the Lie superalgebra $\g$ and the associative algebra $A$ is of order $m=2$ and hence $(\g\otimes A)^{\Gamma}=\bigoplus_{s=0}^{1}\g_{s}\otimes A_{-s}$. Now consider the $\mathbb{Z}_{2}$-gradation $(\g\otimes A)^{\Gamma}=(\g\otimes A)^{\Gamma}_{\bar{0}}\bigoplus(\g\otimes A)^{\Gamma}_{\bar{1}}=(\g_{\bar{0}}\otimes A)^{\Gamma}\bigoplus(\g_{\bar{1}}\otimes A)^{\Gamma}$. We denote $\g_{s_{\bar{0}}}=\{x\in\g_{\bar{0}}~|~\sigma(x)=\zeta^{s}x\}$ and $\g_{s_{\bar{1}}}=\{x\in\g_{\bar{1}}~|~\sigma(x)=\zeta^{s}x\}$. Hence we get
    \begin{align*}
        (\g\otimes A)^{\Gamma} &=\bigoplus_{s=0}^{1}((\g_{s_{\bar{0}}}\otimes A_{-s})\oplus(\g_{s_{\bar{1}}}\otimes A_{-s}))\\
            &=(\g_{0_{\bar{0}}}\otimes A_{0})\oplus(\g_{0_{\bar{1}}}\otimes A_{0})\oplus(\g_{1_{\bar{0}}}\otimes A_{-1})\oplus(\g_{1_{\bar{1}}}\otimes A_{-1}).
    \end{align*}
In particular, $(\n^{-}_{\bar{1}}\otimes A)^{\Gamma}=(\n^{-}_{0_{\bar{1}}}\otimes A_{0})\oplus(\n^{-}_{1_{\bar{1}}}\otimes A_{-1})$.

{\bf Claim I}: $(\n^{-}_{0_{\bar{1}}}\otimes A_{0})w_{\lambda}^{\Gamma}$ is finitely generated $A_{\lambda}^{\Gamma}$-module. Clearly $(\n^{-}_{0_{\bar{1}}}\otimes A_{0})=((\n^{-}_{\Gamma})_{\bar{1}}\otimes A^{\Gamma})$ and let $-\theta$ denote the lowest root of $\g^{\Gamma}$. We have already seen that the  triangular decomposition for $\g^{\Gamma}$ satisfies the $\mathbf{C}$ condition, i.e. $-\theta \in R_{\bar{0}}$. Since $\g^{\Gamma}$ is assumed to be finite dimensional, there exists $k_{0}\in\mathbb{N}$ such that $[x_{1},[x_{2},\ldots[x_{k},x_{\theta}^{-}]\cdots]]=0$ for all $k>k_{0}$ and $x_{1},\ldots,x_{k}\in\n^{+}_{\Gamma}$. Further $\g^{\Gamma}$ is one of the simple Lie superalgebras and $x_{\theta}^{-}$ is the lowest root, we have
\begin{equation*}
\g^{\Gamma}\subseteq~\mbox{span}\{[x_{1},[x_{2},\ldots[x_{k},x_{\theta}^{-}]\cdots]]~|~x_{1}\ldots,x_{k}\in\n^{+}_{\Gamma}~\mbox{and}~0\leq k\leq k_{0}\}.
\end{equation*}
In particular,
\begin{equation}\label{eq:4}
\n_{\Gamma}^{-}\subseteq~\mbox{span}\{[x_{1},[x_{2},\cdots[x_{k},x_{\theta}^{-}]\cdots]]~|~x_{1},\cdots,x_{k}\in\n^{+}_{\Gamma}~ \mbox{and}~0\leq k\leq k_{0}\}.
\end{equation}
Hence by Lemma $\ref{lem:7}$, it is clear that, for each $\alpha\in R^{+}$, the space $((\g^{\Gamma})_{-\alpha}\otimes A^{\Gamma})w_{\lambda}^{\Gamma}$ is finitely generated $\A$- module. Thus $((\n^{-}_{\Gamma})_{\bar{1}}\otimes A^{\Gamma})w_{\lambda}^{\Gamma}$ is finitely generated $\A$- module.
    
{\bf Claim II:} $(\n^{-}_{1_{\bar{1}}}\otimes A_{-1})w_{\lambda}^{\Gamma}$ is finitely generated $\A$-module. For $\sigma \in \Gamma$, consider \[\mathfrak{B}_{1_{\bar{1}}}=\{x_{\beta}^{-}-x_{\sigma(\beta)}^{-}\mid \beta\in \Phi_{\bar{1}},\sigma(\beta)\neq\beta\}\] to be the set of generators for $\n^{-}_{1_{\bar{1}}}$. A basis of $\n_{0_{\bar{1}}}^{-}$ is consisting of elements \[\mathfrak{B}_{0_{\bar{1}}}=\{x_{\beta}^{-}\mid\beta\in \Phi^{+}_{\bar{1}}, \sigma(\beta)=\beta\}\cup\{x_{\beta}^{-}+x_{\sigma(\beta)}^{-}\mid\beta\in \Phi_{\bar{1}}^{+},\sigma(\beta)\neq\beta\}.\]
From Lemma \ref{lem:5}, we know that $A_{s}$ is finitely generated $A^{\Gamma}$ module. So in particular $A_{-1}$ is finitely generated $A^{\Gamma}$ module. Let $\{b_{1},\ldots, b_{k}\}$ be the finite set of generators for $A_{-1}$ as a $A^{\Gamma}$ module and $\{a_{1},\cdots,a_{s}\}$ be the finite set of generators for $A^{\Gamma}$. For any $\alpha\in\Delta$, we can find vectors such that $x_{\alpha}^{-}=x_{\beta_{\alpha}}^{-}+x_{\sigma(\beta_{\alpha})}^{-}$ or $x_{\alpha}^{-}=x_{\beta_{\alpha}}^{-}$ and $h_{\alpha}=h_{\beta_{\alpha}}+h_{\sigma(\beta_{\alpha})}$ or $h_{\alpha}=h_{\beta_{\alpha}}$ for some $\beta_{\alpha}\in \Phi^{+}$. Choose such an $\alpha$ and let $\beta=\beta_{\alpha}$, $\sigma(\beta)\neq\beta$. Then
\begin{equation*}
        \{(x^{-}_{\beta} -x^{-}_{\sigma(\beta)})\otimes a_{1}^{m_{1} }\cdots a_{s}^{m_{s}}b_{i} \mid m_{j}\geq 0,~ 1\leq i\leq k\}\subseteq(\n^{-}_{1_{\bar{1}}}\otimes A_{-1})\subseteq(\n_{\bar{1}}\otimes A)^{\Gamma}.
    \end{equation*}
Also $(h_{\beta}-h_{\sigma({\beta})}\otimes b_{i})\in(\h\otimes A)^{\Gamma}$. Since $(\h\otimes A)^{\Gamma}w_{\lambda}^{\Gamma}=w_{\lambda}^{\Gamma}\A$, then $(h_{\beta}-h_{\sigma(\beta)}\otimes b_{i})w_{\lambda}^{\Gamma}\in w_{\lambda}^{\Gamma}\A$.
\begin{align*}
[h_{\beta}-h_{\sigma(\beta)}, x_{\alpha}^{-}]&=[h_{\beta}-h_{\sigma(\beta)}, x_{\beta}^{-}+x_{\sigma(\beta)}^{-}]\\
&=[h_{\beta}, x_{\beta}^{-}]+[h_{\beta}, x_{\sigma(\beta)}^{-}]-[h_{\sigma(\beta)}, x_{\beta}^{-}]-[h_{\sigma(\beta)}, x_{\sigma(\beta)}^{-}]\\
&= -\beta(h_{\beta})x^{-}_{\beta}-(\sigma(\beta)(h_{\beta}))(x_{\sigma(\beta)}^{-})+\beta(h_{\sigma(\beta)})(x^{-}_{\beta})+(\sigma(\beta)(h_{\sigma(\beta)}))(x^{-}_{\sigma(\beta)})\\
&=-(h_{\beta}, h_{\beta})x^{-}_{\beta}+(h_{\sigma(\beta)}, h_{\sigma(\beta)})x^{-}_{\sigma(\beta)}+(h_{\sigma(\beta)},h_{\beta})x^{-}_{\beta}-(h_{\beta}, h_{\sigma(\beta)})x^{-}_{\sigma(\beta)}\\
&=-(h_{\beta}, h_{\beta})x^{-}_{\beta}+(\sigma(h_{\beta}),\sigma(h_{\beta}))x^{-}_{\sigma(\beta)}+(\sigma(h_{\beta}), h_{\beta})x^{-}_{\beta}-(h_{\beta}, \sigma(h_{\beta}))x^{-}_{\sigma(\beta)}\\
&=\mathbb{K}(x^{-}_{\beta}-x^{-}_{\sigma(\beta)}).
    \end{align*}
This implies that for all $m_{j}\geq 0$,
\begin{align*}
\mathbb{K}((x_{\beta}^{-}-x_{\sigma(\beta)}^{-})&\otimes a_{1}^{m_{1}}\cdots a_{s}^{m_{s}}b_{i})w_{\lambda}^{\Gamma}
            =(h_{\beta}-h_{\sigma(\beta)}\otimes b_{i})(x_{\alpha}^{-}\otimes a_{1}^{m_1}\cdots a_{s}^{m_s})\w\\&-(x_{\alpha}^{-}\otimes a_{1}^{m_1}\cdots a_{s}^{m_s})(h_{\beta}-h_{\sigma(\beta)}\otimes b_{i})\w \\
            &\in(h_{\beta}-h_{\sigma(\beta)}\otimes b_{i})\mbox{span}\{(x_{\alpha}^{-}\otimes a_{1}^{l_1}\cdots a_{s}^{l_s})\w\A-(x_{\alpha}^{-}\otimes a_{1}^{m_1}\cdots a_{s}^{m_s})\w\A \\
&\subseteq(h_{\beta}-h_{\sigma(\beta)}\otimes b_{i})\mbox{span}\{(x_{\alpha}^{-}\otimes a_{1}^{l_1}\cdots a_{s}^{l_s})\w\A\mid 0\leq l_{i}<\lambda(h_{\alpha})\}\\
&+\mbox{span}\{(x_{\alpha}^{-}\otimes a_{1}^{l_1}\cdots a_{s}^{l_s})\w\A\mid 0\leq l_{i}<\lambda(h_{\alpha})~\forall i\}.
\end{align*}
It is clear from the above two cases that $(\n^{-}_{\bar{1}}\otimes A)^{\Gamma}w_{\lambda}^{\Gamma}$ is a finitely generated $\A$-module. 
\end{proof}
Let $\mathbf{U}(\n^{-}\otimes A)^{\Gamma}=\Sigma_{n\geq 0}\mathbf{U}_{n}(\n^{-}\otimes A)^{\Gamma}$ be the filtration on $\mathbf{U}(\n^{-}\otimes A)^{\Gamma}$ induced from the usual grading of the tensor algebra.
 \begin{lemma}\label{lem:15}
Let $\g$ be the Lie superalgebra and $\g^{\Gamma}$ be the fixed subalgebra having the triangular decomposition satisfying the condition $\mathbf{C}$. Then there exists $n_{0}\in\mathbb{N}$ such that 
     \begin{equation*}
         \mathbf{U}_{n}(\n^{-}\otimes A)^{\Gamma}\w\A=W^{\Gamma}(\lambda), \quad \forall~n\geq n_{0}.
     \end{equation*}
 \end{lemma}
 \begin{proof}
$W^{\Gamma}(\lambda)=\mathbf{U}(\n^{-}\otimes A)^{\Gamma}\w\A$. Then by PBW theorem,
     \begin{equation*}
         W^{\Gamma}(\lambda)=\mathbf{U}(\n^{-}_{\bar{1}}\otimes A)^{\Gamma}\mathbf{U}(\n^{-}_{\bar{0}}\otimes A)^{\Gamma}\w\A.
     \end{equation*}
    $\mathbf{U}(\n_{\bar{0}}^{-}\otimes A)^{\Gamma}\w$ is a $(\g_{\bar{0}}\otimes A)^{\Gamma}$-submodule of $W^{\Gamma}(\lambda)$ generated by $\w$. Clearly, it is the quotient of the Weyl $(\g_{\bar{0}}\otimes A)^{\Gamma}$-module of highest weight $\lambda$. That is, it is the quotient of the global Weyl module corresponding to the reductive Lie algebra $\g_{\bar{0}}$. Hence it is clearly a finitely generated $\A$-module and by \cite[Theorem 5.10 ]{FMS15}, there exist $f_{1},\ldots,f_{k}\in\n_{\bar{0}}^{-}\otimes A$ such that 
    \begin{equation*}
        \mathbf{U}(\n_{\bar{0}}^{-}\otimes A)^{\Gamma}\w\A=\sum_{1\leq i_{1}\leq\cdots\leq i_{t}\leq k}f_{i_{1}}\cdots f_{i_{t}}\w\A.
    \end{equation*}
From Lemma \ref{lem:9} , we get that $(\n_{\bar{1}}^{-}\otimes A)^{\Gamma}\w$ is a finitely generated $\A$-module and hence there exists $g_{1},\ldots,g_{l}\in(\n_{\bar{1}}^{-}\otimes A)^{\Gamma}$ such that 
\begin{equation*}
(\n^{-}_{\bar{1}}\otimes A)^{\Gamma}\w\A=\sum_{1\leq j_{1}\leq\cdots\leq j_{s}\leq l}g_{j_{1}}\cdots g_{j_s}\w\A.
\end{equation*}
Using induction on $t$ and $s$, we get 
    \begin{equation*}
        \mathbf{U}(\n_{\bar{1}}^{-}\otimes A)^{\Gamma}\mathbf{U}(\n^{-}_{\bar{0}}\otimes A)^{\Gamma}\w\A= \sum_{\substack{1\leq i_{1}\leq\cdots\leq i_{t}\leq k,\\ 1\leq j_{1}\leq\cdots\leq j_{s}\leq l}}
        g_{j_{1}}\cdots g_{j_s}f_{i_{1}}\cdots f_{i_{t}}\w\A.
    \end{equation*}
 \end{proof}
\begin{thm}\label{thm:1}
Let $\g$ be the Lie superalgebra and $\g^{\Gamma}$ be the fixed subalgebra with triangular decomposition satisfying the condition $\mathbf{C}$. For all $\lambda\in \Lambda^{+}$, the global Weyl module $W^{\Gamma}(\lambda)$ is finitely generated as a right $\A$-module.
\end{thm}
 \begin{proof}
We prove that $\mathbf{U}_{n}(\n^{-}\otimes A)^{\Gamma}\w\A$ is a finitely generated $\A$-module for every $n\geq 0$. Recall that $A^{\Gamma}$ is a finitely generated algebra and let $\{a_{1},\ldots a_{t}\}$ be the set of generators for $A^{\Gamma}$. Just as we had defined $\mathfrak{B}_{0_{\bar{1}}}$ and $\mathfrak{B}_{1_{\bar{1}}}$ as the basis for $\n^{-}_{0_{\bar{1}}}$ and $\n^{-}_{1_{\bar{1}}}$ respectively, we define $\mathfrak{B}_{0_{\bar{0}}}$ to be the basis of $\n_{0_{\bar{0}}}^{-}$, obtained from the right side of \eqref{eq:4}, and $\mathfrak{B}_{1_{\bar{0}}}=\{x^{-}_{\alpha}-x^{-}_{\sigma(\alpha)}\mid\alpha\in\Phi_{\bar{0}},\sigma(\alpha)\neq\alpha\}$ to be the basis for $\n^{-}_{1_{\bar{0}}}$. Define
     \begin{align*}
\mathcal{D}_{0_{\bar{1}}}&=\{x\otimes a_{1}^{l_1}\cdots a_{s}^{l_s}~|~x\in \mathfrak{B}_{0_{\bar{1}}},~0\leq l_{j}<\lambda(h_{\alpha})~\forall j\}\\
\mathcal{D}_{1_{\bar{1}}}&=\{x\otimes a_{1}^{l_1}\cdots a_{s}^{l_s}b_{i}~|~x\in \mathfrak{B}_{1_{\bar{1}}},~0\leq l_{j}<\lambda(h_{\alpha})~\forall j\quad 1\leq i\leq k\}\\
\mathcal{D}_{0_{\bar{0}}}&=\{x\otimes a_{1}^{l_1}\cdots a_{s}^{l_s} \mid x\in \mathfrak{B}_{0_{\bar{0}}},~0\leq l_{j}<\lambda(h_{\alpha})~\forall j\}\\ 
\mathcal{D}_{1_{\bar{0}}}&=\{x\otimes a_{1}^{l_1}\cdots a_{s}^{l_s}b_{i}~|~x\in \mathfrak{B}_{1_{\bar{0}}},~0\leq l_{j}<\lambda(h_{\alpha})~\forall j\quad 1\leq i\leq k\}. 
\end{align*}
Let $\mathfrak{D}=\mathcal{D}_{0_{\bar{1}}}\cup\mathcal{D}_{1_{\bar{1}}}\cup\mathcal{D}_{0_{\bar{0}}}\cup\mathcal{D}_{1_{\bar{0}}}$. Clearly, this forms the basis for $(\n^{-}\otimes A)^{\Gamma}$. Using induction, we claim that 
\begin{equation*}
         \mathbf{U}_{n}(\n^{-}\otimes A)^{\Gamma}\w\subseteq \mathrm{span}\{Y_{1}^{n_1}\cdots Y_{t}^{n_t}\w\A\mid t\geq 0,~Y_{1}\cdots Y_{t}\in\mathfrak{D} ~\mathrm{and}~
         n_{1}+\cdots +n_{t}\leq n\}.
     \end{equation*}
The case for $n=0$ is trivial. For $n=1$ it is clear from the definition of $\mathfrak{D}$ that it is true. We assume that it is true for $n\geq 1$. Let $u\in\mathbf{U}_{1}(\n^{-}\otimes A)^{\Gamma}$ and $u'\in\mathbf{U}_{n}(\n^{-}\otimes A)^{\Gamma}$. Then by assumption $u'\w\in\mathrm{span}\{Y_{1}^{n_1}\cdots Y_{t}^{n_t}\w\A~|~t\geq 0,~ Y_{1}\cdots Y_{t}\in \mathfrak{D} ~\mbox{and}\quad n_{1}+\cdots n_{t}\leq n\}$. Then we have 
\begin{align*}
uu'\w &=[u,u']\w+(-1)^{{|u|}{|u'|}}u'u\w \\
&\in\mathbf{U}_{n}(\n^{-}\otimes A)^{\Gamma}\w\A+\mbox{Span}\{u'Y\w\A\mid  Y\in\mathfrak{D}\} \\
&\subseteq\mathrm{span}\{Y_{1}^{n_1}\cdots Y_{t+1}^{n_{t+1}}\w\A\mid t\geq 0,\;\; Y_{1},\cdots,Y_{t+1}\in \mathfrak{D}\;\;\mbox{and}\;\; n_{1}+\cdots+n_{t+1}\leq n+1\}.
\end{align*}
This shows that $\mathbf{U}_n(\n^{-}\otimes A)^{\Gamma}\w\A$ is a finitely generated $\A$-module. From Lemma \ref{lem:15}, we have seen that there exists $n_{0}\in\mathbb{N}$, such that $W^{\Gamma}(\lambda)=\mathbf{U}_{n}(\n^{-}\otimes A)^{\Gamma}\w\A$ for all $n\geq n_{0}$. Hence the result follows.
\end{proof}
The following corollary follows directly from Theorem \ref{thm:1}.
\begin{cor}\label{cor:7.9}
Let $\g$ be the Lie superalgbera and $\g^{\Gamma}$ be the fixed subalgebra with triangular decomposition that satisfies condition $\mathbf{C}$. If $M$ is a finitely generated $\A$-module (resp. finite dimensional), then $\mathbf{W}_{\lambda}^{\Gamma}M$ is a finitely generated (resp. finite dimensional) $(\g\otimes A)^{\Gamma}$-module.
\end{cor}
\begin{proof}
From Theorem \ref{thm:1}, we know that $W^{\Gamma}(\lambda)$ is a finitely generated $\A$-module. This implies that there exists a finite set of generators $\{w_{1},\cdots,w_{k}\}$ for $W^{\Gamma}(\lambda)$ as an $\A$-module. Hence, for any $w\in W^{\Gamma}(\lambda)$, we have
\[w =c_{1}w_{1}+c_{2}w_{2}+\cdots+c_{k}w_{k},\quad c_{1},\cdots,c_{k}\in\A\] implies
    \begin{align*}
 W^{\Gamma}(\lambda)&=w_{1}\A\oplus w_{2}\A\oplus\cdots\oplus w_{k}\A\\
&\cong\bigoplus_{i=1}^{k}w_{i}\A.
    \end{align*}
From the above isomorphism, we find the surjective map
    \begin{equation}\label{eq:12}
        \bigoplus_{i=1}^{k}\A\rightarrow W^{\Gamma}(\lambda)
    \end{equation}
such that for every $w\in W^{\Gamma}(\lambda)$ we have,
\begin{equation*}
        (c_{1},\ldots, c_{k})\rightarrow w.
    \end{equation*}
Let M be a finitely generated $\A$-module (resp. finite dimensional). We have already seen from the definition of Weyl functor that 
$\mathbf{W}_{\lambda}^{\Gamma}M=W^{\Gamma}(\lambda)\otimes_{\A} M$. Then extending \eqref{eq:12} we get the surjective map,
\begin{equation}\label{eq:13}
       \bigoplus_{i=1}^{k}\A\otimes_{\A} M\rightarrow W^{\Gamma}(\lambda)\otimes_{\A} M.
   \end{equation}
Since $\A\otimes M\cong M$ as the $\A$-module, \eqref{eq:13} becomes the surjective map
\begin{equation*}
       \bigoplus_{i=1}^{k}M\rightarrow W^{\Gamma}(\lambda)\otimes_{\A} M.
\end{equation*}
We have $M$ is finitely generated (resp. finite dimensional), and $\mathbf{W}_{\lambda}^{\Gamma}M$ is a quotient of finite copies of finitely generated (resp. finite dimensional) $\A$-module. So $\mathbf{W}_{\lambda}^{\Gamma}M$ is finitely generated (resp. finite dimensional) $(\g\otimes A)^{\Gamma}$-module.
\end{proof}
\begin{rem}\label{rem:1}
If $\Gamma$ is trivial, then we can recover \cite[Corollary 7.10]{BCM19}.
\end{rem}
\section{Twisted Local Weyl modules}\label{sec:8}
 In this section, we introduce the twisted local Weyl modules, in terms of generators and relations and prove that, they are finite dimensional withrespect to any triangular decomposition of $\g^{\Gamma}$. Here we assume that $A$ is a finitely generated commutative associative unital $\mathbb{K}-$algebra.
\begin{defn}\label{def:7}
Let $\psi\in ((\h\otimes A)^{\Gamma})^{*}$ and $\psi|_{\h^{\Gamma}}=\lambda\in \Lambda^{+}$. The local Weyl module $W^{\Gamma}_{loc}(\psi)$ associated to $\psi$ is defined by the cyclic $(\g\otimes A)^{\Gamma}$-module given as the quotient of $\mathbf{U}(\g\otimes A)^{\Gamma}$ by the left ideal $I$ generated by 
\begin{equation*}
    (\n^{+}\otimes A)^{\Gamma},\quad h-\psi(h),\quad (x_{\alpha}^{-})^{\psi(h_{\alpha})+1},\quad \mbox{for all}~h\in (\h\otimes A)^{\Gamma} \mbox{and}~\alpha\in\Delta_{\bar{0}}.
\end{equation*}
Hence, we get a map $\hat{f}:\mathbf{U}(\g\otimes A)^{\Gamma}\rightarrow\mathbf{U}(\g\otimes A)^{\Gamma}/I$ such that $\hat{f}(1)=w_{\psi}^{\Gamma}$, where $w_{\psi}^{\Gamma}$ is the element that generates $W_{loc}^{\Gamma}(\psi)$ as a $(\g\otimes A)^{\Gamma}$-module. This vector $w_{\psi}^{\Gamma}$ satisfies the following defining relations:
\begin{equation}\label{eq:8}
    (\n^{+}\otimes A)^{\Gamma}w_{\psi}^{\Gamma}=0,\;\; hw_{\psi}^{\Gamma}=\psi(h)w_{\psi}^{\Gamma}, \;\; (x_{\alpha}^{-})^{\psi(h_{\alpha})+1}w_{\psi}^{\Gamma}=0, \;\; \mbox{for all}~h\in(\h\otimes A)^{\Gamma}~\mbox{and}~\alpha\in\Delta_{\bar{0}}.
\end{equation}
\end{defn}
From the defining relations in \eqref{eq:8} and the PBW theorem, it is clear that $W^{\Gamma}_{loc}(\psi)=\mathbf{U}(\n^{-}\otimes A)^{\Gamma}w_{\psi}^{\Gamma}$. If $\Gamma$ is trivial then the definition coincides with the untwisted local Weyl module as discussed in \cite{CLS19}.

\begin{defn}(Highest map weight module)
    A $(\g\otimes A)^{\Gamma}$-module generated by a vector $v_{\psi}$ satisfying the first and second relations in \eqref{eq:8} is called a highest map-weight module with highest map-weight $\psi$. The vector $v_{\psi}$ is called a highest map-weight vector of map-weight $\psi$. 
\end{defn}
Since $\A$ is a commutative associative algebra, every irreducible $\A$-module is a one dimensional module. For $\psi\in((\h\otimes A)^{\Gamma})^{*}$ such that $\psi|_{\h_{\Gamma}}=\lambda\in \Lambda^{+}$, let $\mathbb{K}_{\psi}$ denote the irreducible one dimensional $\A$-module, where $xv=\psi(x)v$ for all $x\in\A$ and $v\in\mathbb{K}_{\psi}$.
\begin{rem}
Clearly $W^{\Gamma}_{loc}(\psi)$ satisfies all the defining relations as in \eqref{eq:3}. By the proof of Proposition \ref{prop:1}, we get $\LW$ is a quotient of $W^{\Gamma}(\lambda)$. As the category $\mathcal{I}_{\lambda}^{\Gamma}$ is closed under taking the quotient, $\LW\in\mathcal{I}_{\lambda}^{\Gamma}$.
\end{rem}
\begin{rem}\label{rem8.3}
We have already defined the twisted restriction functor $\R$. Let us see what happens when this functor acts on the local Weyl module. As defined earlier, the local Weyl module $W_{loc}^{\Gamma}(\psi)$ is generated by the vector $\lw$ such that $(\n^{+}\otimes A)^{\Gamma}\lw=0$ and $h\lw=\psi(h)\lw$ for all $h\in(\h\otimes A)^{\Gamma}$. Then $\R\LW=(\LW)_{\lambda}=\mathbb{K}\lw$. Notice that $\mathbb{K}\lw$ is isomorphic to $\mathbb{K}_{\psi}$ as $\A$-modules.   
\end{rem}
\smallskip
In the next result we will describe local weyl modules using Weyl functors.
\begin{thm}\label{thm:8.4}
Let $\g$ be the Lie superalgebra and let $\psi\in((\h\otimes A)^{\Gamma})^{*}$ such that $\psi|_{\h^{\Gamma}}=\lambda\in\Lambda^{+}$. Then $\mathbf{W}_{\lambda}^{\Gamma}\mathbb{K}_{\psi}\cong\LW$ as $(\g \otimes A)^{\Gamma}$- modules.
\end{thm}
\begin{proof}
By Remark \ref{rem8.3}, we have $\R\LW=\mathbb{K}\lw$. Thus, there exists a unique homomorphism of $(\g\otimes A)^{\Gamma}$-modules $\epsilon_{\LW}:\mathbf{W}_{\lambda}^{\Gamma}\R\LW\rightarrow\LW$ satisfying
    \begin{equation*}
        \epsilon_{\LW}(uw_{\lambda}^{\Gamma}\otimes\lw)=u\lw \quad \mbox{for all}\quad u\in\mathbf{U}(\g\otimes A)^{\Gamma}.
    \end{equation*}
Clearly $\epsilon_{\LW}$ is a surjective map. By definition $\mathbf{W}_{\lambda}^{\Gamma}\R\LW$ is a 
$(\g\otimes A)^{\Gamma}$-module generated by the highest weight vector $1\otimes\lw$, and using Corollary \ref{cor:7.9}, it is finite dimensional. In addition, $1\otimes \lw$ satisfies the relations \eqref{eq:8}. Hence, we get a surjective map of $(\g\otimes A)^{\Gamma}$-modules $\eta:\LW\rightarrow \mathbf{W}_{\lambda}^{\Gamma}\R\LW$ satisfying $\eta(\lw)=1\otimes\lw$. Therefore, $\eta\circ\epsilon_{\LW}=id_{{\mathbf{W}_{\lambda}^{\Gamma}\R\LW}}$ and $\epsilon_{\LW}\circ\eta=id_{\LW}$.
\end{proof}

\begin{thm}
Let $\psi\in((\h\otimes A)^{\Gamma})^{*}$ with $\psi|_{\h_{\Gamma}}\in\Lambda^{+}$. Let $\g$ be the Lie superalgebra such that $\g^{\Gamma}$ has a triangular decomposition satisfying the condition $\mathbf{C}$. Then the local Weyl module is finite dimensional.
\end{thm}
\begin{proof}
    The proof follows directly from Theorem \ref{thm:8.4} and Corollary \ref{cor:7.9}.
\end{proof}
\begin{prop}
Let $\psi\in((\h\otimes A)^{\Gamma})^{*}$ be such that $\psi|_{\h_{\Gamma}}=\lambda\in\Lambda^{+}$. If $\g^{\Gamma}$ has the triangular decomposition satisfying the condition $\mathbf{C}$, then the local Weyl module $\LW$, is the unique (up to isomorphism) finite dimensional object of $\mathcal{I}_{\lambda}^{\Gamma}$ that is generated by a highest map-weight vector of  map-weight $\psi$ and admits a surjective homomorphism to any finite dimensional object of $\mathcal{I}_{\lambda}^{\Gamma}$ also generated by a highest map-weight vector of map-weight $\psi$.
\end{prop}
\begin{proof}
Let $V$ be a finite dimensional object of $\mathcal{I}_{\lambda}^{\Gamma}$ that is generated by a highest map-weight vector $v$ of map-weight $\psi$. It follows from the definition of a highest map-weight module that the vector $v$ satisfies the first two relations in \eqref{eq:8}. Since $\g_{\bar{0}}^{\Gamma}$ is finite dimensional, the $\g_{\bar{0}}^{\Gamma}$ module generated by the vector $v$ should also be finite dimensional. Hence it satisfies the third relation in \eqref{eq:8}. Therefore, there exists a surjective homomorphism $\LW\rightarrow V$ mapping $\lw\rightarrow v$. To show that $\LW$ is the unique finite dimensional object with this property, let us assume that $W$ is any other finite dimensional $(\g\otimes A)^{\Gamma}$-module in $\mathcal{I}_{\lambda}^{\Gamma}$ with the given property. Then we will get surjective maps $\LW\rightarrow W$ and $W\rightarrow\LW$. Hence $W$ is a quotient of $\LW$ and vice-versa. Since both are finite dimensional, we get $\LW\cong W$. 
\end{proof}
\begin{cor}
Let $\psi\in((\h\otimes A)^{\Gamma})^{*}$ be such that $\psi|_{\h_{\Gamma}}=\lambda\in\Lambda^{+}$. If the triangular decomposition for $\g^{\Gamma}$ satisfies condition $\mathbf{C}$, then the local Weyl module $\LW$ is the maximal quotient of the global Weyl module $W^{\Gamma}(\lambda)$ that is a highest map-weight module of highest map-weight $\psi$.
\end{cor}
Consider the Lie superalgebra $\g$ with the fixed subalgebra $\g^{\Gamma}$ having a triangular decomposition not necessarily satisfy condition 
$\mathbf{C}$. Then we give a necessary and sufficient condition for the local Weyl module to be finite dimensional.
\begin{lemma}[\cite{FMS15}]
If $J=\oplus_{\xi\in\Xi}J_{\xi}$ is a $\Gamma-$invariant ideal of $A$ then $A_{\tau}J_{\xi}=J_{\tau+\xi}$ for all $\tau,\xi\in\Xi$.
\end{lemma}
\begin{lemma}[\cite{FMS15}]\label{lem:17}
Suppose that $J_{0}$ is an ideal of $A_{0}=A^{\Gamma}$. Then the ideal of $(\g\otimes A)^{\Gamma}$ generated by $\g^{\Gamma}\otimes J_{0}$ is $(\g\otimes J)^{\Gamma}$, where $J=\oplus_{\xi\in\Xi}A_{\xi}J_{0}$ is the ideal of $A$ generated by $J_{0}$.
\end{lemma}
Let $\g$ be the Lie superalgebra and $W_{A}^{loc}(\psi)$ be the untwisted local Weyl module with the generator $w_{\psi}$. Consider $I_{\alpha}$ to be the kernel of the linear map
\begin{equation*}
    A\longrightarrow \mbox{Hom}_{\mathbb{K}}(\g_{-\alpha}\otimes W_{A}^{loc}(\psi)_{\lambda},(\g_{\alpha}\otimes A)w_{\psi}).
\end{equation*}
\begin{equation*}
    a\longmapsto (u\otimes v\mapsto(u\otimes a)v), \quad a\in A, u\in\g_{-\alpha}, v\in W_{A}^{loc}(\psi)_{\lambda}.
\end{equation*}
The $I_{\alpha}$ so defined will be a finite codimensional ideal of $A$. If $I=\cap_{\alpha\in\Phi_{\bar{0}}^{+}}I_{\alpha}$, then it can be seen that $(\n_{\bar{0}}^{-}\otimes I)W_{A}^{loc}(\psi)_{\lambda}=0$. Since $\Phi_{\bar{0}}^{+}$ is a finite set and $I$ is the intersection of finitely many finite co-dimensional ideals, we find that $I$ is also a finite co-dimensional ideal of $A$. Then we define, for $\psi\in(\h\otimes A)^{*}$ with $\psi|_{\h}\in P^{+}$, $I_{\psi}$ to be the sum of all the ideals $I\subseteq A$ such that $(\n_{\bar{0}}^{-}\otimes I)w_{\psi}=0$.
Let $w_{\psi'}$ be the generator of the untwisted local Weyl module and $w_{\psi'}^{\g}=\{x\in\g \mid (x\otimes a)w_{\psi'}=0,~ \forall a\in A\}$. Note $w_{\psi'}^{\g}$ is a subalgebra of $\g$ and $\n^+ \subseteq w_{\psi'}^{\g}$ \cite{BCM19}.
\begin{lemma}[Lemma 8.11 in \cite{BCM19}] \label{lem:16}
Let $\g$ be a finite dimensional simple Lie superalgebra not of type $\q(n)$, and $\psi'\in(\h\otimes A)^{*}$ be such that $\psi'|_{\h}=\lambda \in P^+$. If $X_{\theta}^{-}$, which is the lowest weight in $\g$, is in the $\g_{\bar{0}}$-submodule of $\g$ generated by $w_{\psi'}^{\g}$, then there exists $n_{\psi'}\in\mathbb{N}$ such that $(\g \otimes I_{\psi'}^{n_{\psi'}})w_{\psi'}=0$.
\end{lemma}
\begin{lemma}\label{lem:8.9}
Let $\g$ be the Lie superalgebra. For $\psi\in((\h\otimes A)^{\Gamma})^{*}$ such that $\psi|_{\h_{\Gamma}}=\lambda\in\Lambda^{+}$, there exists a finite co-dimensional ideal $I$ such that $(\g\otimes I^{n})^{\Gamma}$ annihilates $\LW$ for some positive integer $n\in\mathbb{N}$.
\end{lemma}
\begin{proof}
We have already seen that the Lie superalgebra $\g$ has a $\mathbb{Z}_{2}$ gradation $\g=\g_{0}\oplus\g_{1}$ because of the action of $\Gamma$ on $\g$(here $\g_{0}=\g^{\Gamma}$ and $\g_{1}$ is the set of elements in $\g$ having eigen value $-1$). Under this gradation $\g_{1}$ is going to be an irreducible $\g_{0}$ module \cite{vdL}. Using Lemma \ref{lem:3}, we know that an irreducible finite dimensional $\g_{0}$ module is going to be a highest weight module. Let $x^{+}_{\alpha}$ be the highest root vector corresponding to the highest weight $\alpha$ implies
\begin{equation*}
\n_{0}^{+}x_{\alpha}^{+}=[\n_{0}^{+},x_{\alpha}^{+}]=0.
\end{equation*}
Since $\g_{1}$ is an irreducible $\g_{0}$ module with highest weight $\alpha$, we have
\begin{equation*}
         \g_{1}=\mbox{span}\{[x_{1},[\cdots[x_{k},x_{\alpha}^{+}]]]~|~ x_{1},x_{2},\cdots, x_{k}\in\n_{0}^{-}\}.
\end{equation*}
An ideal $(\g\otimes I)^{\Gamma}$ for $(\g\otimes A)^{\Gamma}$ has the form $(\g_{0}\otimes I_{0})\oplus(\g_{1}\otimes I_{-1})$, where $I_{-1}=A_{-1}I_{0}$. Using Lemma \ref{lem:16}, we know that there already exists a finite codimensional ideal $I_{\psi}^{n}$ for $A_{0}$ such that $(\g^{\Gamma}\otimes I_{\psi}^{n})\lw=0$. In particular, $(\n_{0}^{-}\otimes I_{\psi}^{n})\lw=0$. Then
\begin{align*}
(\g_{1}\otimes A_{-1}I_{\psi}^{n})\lw &=\mbox{span}([\n_{0}^{-},x_{\alpha}^{+}]\otimes A_{-1}I_{\psi}^{n})\lw \\
&=\mbox{span}([\n_{0}^{-}\otimes I_{\psi}^{n},x_{\alpha}^{+}\otimes A_{-1}])\lw \\
&=\mbox{span}((\n_{0}^{-}\otimes I_{\psi}^{n})(x_{\alpha}^{+}\otimes A_{-1})\lw-(-1^{|\n_{0}^{-}||x_{\alpha}^{+}|})(x_{\alpha}^{+}\otimes A_{-1})(\n_{0}^{-}\otimes I_{\psi}^{n})\lw)\\
&=0.
\end{align*}
Now using Lemma  \ref{lem:17}, we denote $(\g\otimes I^{n})^{\Gamma}$ to be the ideal generated by $(\g^{\Gamma}\otimes I_{\psi}^{n})$. Hence \[(\g\otimes I^{n})^{\Gamma}\LW=((\g_{0}\otimes I_{\psi}^{n})\oplus(\g_{1}\otimes A_{-1}I_{\psi}^{n}))\LW=0.\]
Therefore, we see that there exists a finite codimensional ideal $I$ such that $(\g\otimes I^{n})^{\Gamma}\lw=0$.
\end{proof}
\begin{thm}\label{thm:4}
Let $\g$ be the Lie superalgebra and let $\psi\in((\h\otimes A)^{\Gamma})^{*}$ be such that $\psi|_{\h_{\Gamma}}=\lambda\in\Lambda^{+}$ and $A$ be the commutative associative unital algebra of infinite dimension. Then, $\LW$ is finite dimensional for any triangular decomposition of $\g^{\Gamma}$.
\end{thm}
\begin{proof}
We know that, for the Lie superalgebra $\g$, the fixed subalgebra $\g^{\Gamma}$ is going to be of type $\textbf{II}$. We had chosen the triangular decomposition for $\g$ to be consistent with that of the triangular decomposition for $\g^{\Gamma}$. Since it is not possible for type $\textbf{II}$ superalgebra to have a parabolic triangular decomposition, $\g$ will also not have a parabolic triangular decomposition. From the definition of $\LW$, we know that $\LW=\mathbf{U}(\n^{-}\otimes A)^{\Gamma}\lw$. Using Lemma \ref{lem:8.9}, we have a finite co-dimensional ideal $I$ such that $(\n^{-}\otimes I^{n})^{\Gamma}\lw=0$. Hence $\LW=\mathbf{U}(\n^{-}\otimes A/I^{n})^{\Gamma}\lw$.  We show that there exists an $N\in\mathbb{N}$ such that 
\begin{equation*}
\LW=\mathbf{U}_{n}(\n^{-}\otimes A/I^{n})^{\Gamma}\lw \quad \forall\quad n\geq N.
\end{equation*}
We know that the weights of $W^{\Gamma}(\lambda)$ are contained in $\lambda-Q_{\Gamma}^{+}$ and hence only a finite number of weights are dominant integral weights. Since $W^{\Gamma}(\lambda)\in\mathcal{I}_{\lambda}^{\Gamma}$, it is going to be a semisimple $\g_{\bar{0}}^{\Gamma}$ module, This would imply that $W^{\Gamma}(\lambda)$ can be written as the direct sum of finite dimensional $\g_{\bar{0}}^{\Gamma}$ modules and its weights are invariant under the action of the (finite) Weyl group of $\g_{\bar{0}}^{\Gamma}$. We also know that the local Weyl module $W_{loc}^{\Gamma}(\psi)$ for $\psi\in((\h\otimes A)^{\Gamma})^{*}$, such that $\psi|_{\h_{\Gamma}}=\lambda\in\Lambda^{+}$, is a quotient of $W^{\Gamma}(\lambda)$. Hence for all $\psi\in((\h\otimes A)^{\Gamma})^{*}$, such that $\psi|_{\h_{\Gamma}}=\lambda\in\Lambda^{+}$, the set of $\g^{\Gamma}$ weights of $W_{loc}^{\Gamma}(\psi)$ is finite. This gives us the existence of $N\in\mathbb{N}$ such that
\begin{equation*}
        W^{\Gamma}_{loc}(\psi)=\mathbf{U}_{n}(n^{-}\otimes A/I^{n})^{\Gamma}\lw \quad n\geq N.
    \end{equation*}
Since $I^{n}$ is finite codimensional, the Lie superalgebra $(\n^{-}\otimes A/I^{n})^{\Gamma}$ is finite dimensional. Hence we can conclude that $\LW$ is finite dimensional.
\end{proof}
\begin{defn} We define $\mathcal{L}(\h\otimes A)^{\Gamma}$\label{def:8} as follows:
    \[\mathcal{L}(\h\otimes A)^{\Gamma}=\{\psi\in((\h\otimes A)^{\Gamma})^{*}\mid \psi(\h\otimes I)^{\Gamma}=0\; \mbox{for some finite co-dimensional ideal $I$ of $A$}\}.\]
\end{defn}
\begin{cor}
    Let $\psi\in((\h\otimes A)^{\Gamma})^{*}$ such that $\psi|_{\h_{\Gamma}}\in\Lambda^{+}$. If $\psi\notin\mathcal{L}(\h\otimes A)^{\Gamma}$, then $\LW=0$.
\end{cor}
\begin{proof}
From Theorem \ref{thm:4}, we have seen that $\LW$ is finite dimensional. There exists a finite co-dimensional ideal $I$ of $A$ such that $(\g\otimes I)^{\Gamma}\lw=0$. In particular $(\h\otimes a)^{\Gamma}\lw=\psi(\h\otimes a)^{\Gamma}\lw=0$ for all $a\in I$. If $\psi\notin\mathcal{L}(\h\otimes A)^{\Gamma},$ then there exists $a\in I$ such that $\psi(\h\otimes a)^{\Gamma}\neq 0$. As $(\h\otimes a)^{\Gamma}\lw=0$ we get $\lw=0$ and $\LW=0$.
\end{proof}

\section{Twisted Local weyl modules as an image of untwisted local Weyl modules}\label{sec:9}
For a semisimple Lie algebra $\g$ and a commutative associative finitely generated algebra $A$, the local Weyl modules, for equivariant map algebras $(\g\otimes A)^{\Gamma}$, are defined as the restriction of the local Weyl modules for $\g\otimes A$ using twisting functors \cite{FKKS12}.
In this section, we extend this idea to the case of Lie superalgebras. We begin by recalling the concepts of evaluation modules \cite{Sav14}.

Here we assume Lie superalgebra $\g$ such that the triangular decomposition of $\g^{\Gamma}$ satisfies the condition $\mathbf{C}$. Let $X=\mbox{spec}\; A$, be the prime spectrum of $A$. So $X$ is an affine scheme of finite type. Let $X_{rat}$ denote the set of $\mathbb{K}$-rational points of $X$. We say $\textbf{m}\in\mbox{maxspec}A$ is a $\mathbb{K}$-rational point of $X$ if its residue field $A/\textbf{m}$ is $\mathbb{K}$. Thus, $X_{rat}\subseteq\mbox{maxspec}A$ and this becomes an equality when $A$ is finitely generated. Let $\Gamma$ be the cyclic group acting freely on $X$ (equivalently on $A$) and on the Lie superalgebra $\g$. Let $\g\otimes A$ be the map Lie superalgebra of regular maps from $X$ to $\g$ under pointwise multiplication. The equivariant map superalgebra $(\g\otimes A)^{\Gamma}$ consists of the $\Gamma$-fixed points of the diagonal action of $\Gamma$ on $\g\otimes A$. Hence $(\g\otimes A)^{\Gamma}$ is the subalgebra of $\Gamma$-equivariant maps. 

When we refer to the dimension of $A$, we consider $A$ to be a vector space. We say that the ideal $I$ of A is of finite co-dimension, if the quotient $A/I$ is finite dimensional as a vector space over the field $\mathbb{K}$.
\begin{defn}(Support of an ideal $J$)\label{def:9}
      The support of an ideal $J$ of $A$ is defined to be
      \begin{equation*}
          \Supp J=\{\textbf{m}\in\mbox{maxspec}\ A~|~ J\subseteq\textbf{m}\}.
      \end{equation*}
\end{defn}
\begin{defn}(Support of a module)\label{def:10}
    Let V be a $(\g\otimes A)^{\Gamma}$-module. The annihilator of V, denoted by $\Ann_{A}^{\Gamma}V$, is the largest $\Gamma$-invariant ideal $I$ of $A$ such that $(\g\otimes I)^{\Gamma}V=0$. It is also defined to be
    \begin{equation*}
        \Ann^{\Gamma}_{A}V=\langle f\in A~\mid 
        ~uV=0, ~\forall~u\in(\g\otimes A)^{\Gamma}\cap(\g\otimes f)\rangle.
    \end{equation*}
    Then the support of $V$ is defined to be
    \begin{equation*}
        \Supp_{A}^{\Gamma}V= \Supp \Ann_{A}^{\Gamma}V.
    \end{equation*}
\end{defn}
\begin{rem}
    In this section we choose  the root space decomposition for $\g^{\Gamma}$ that satisfies the condition $\mathbf{C}$, as we had defined in Section \ref{sec:2}. We also consider the root space decomposition for $\g$ that is consistent with that of $\g^{\Gamma}$.
\end{rem}

\begin{defn} (Evaluation modules \cite{Sav14})\label{def:11}
Let $X_{*}$ \label{pg:7} be the set of finite subsets of $X_{rat}$ such that no two elements in each of the subsets lie in the same orbit. That is, 
\begin{equation*}
    X_{*}=\{\mathbf{x}\subseteq X_{rat}|\quad\Gamma x\cap\Gamma y=\emptyset,\quad x\neq y\in\mathbf{x}\}.
\end{equation*}
As discussed earlier, $X_{rat}\subseteq\mbox{maxspec}\ A$. Hence, corresponding to every point $x\in\mathbf{x}$, there exists some $\textbf{m}_{x}\in\mbox{maxspec}\ A$. For some $\mathbf{x}\in X_{*}$, the evaluation map for $(\g\otimes A)^{\Gamma}$ is defined as
\begin{equation*}
    \mbox{ev}_{\mathbf{x}}^{\Gamma}:(\g\otimes A)^{\Gamma}\rightarrow(\g\otimes A)/(\g\otimes \prod_{\mathbf{x}}\textbf{m}_{x})\cong\bigoplus_{\mathbf{x}}(\g\otimes(A/\textbf{m}_{x}))\cong\bigoplus_{\mathbf{x}}(\g\otimes\mathbb{K})\cong\bigoplus_{\mathbf{x}}\g.
\end{equation*}
That is, the evaluation map for $(\g\otimes A)^{\Gamma}$ is the restriction of the evaluation map for $\g\otimes A$. We let $\g^{\mathbf{x}}=\bigoplus_{\mathbf{x}}\g$. Hence the evaluation map for $(\g\otimes A)^{\Gamma}$ is 
\begin{equation*}
    \mbox{ev}_{\mathbf{x}}^{\Gamma}:(\g\otimes A)^{\Gamma}\rightarrow\g^{\mathbf{x}},\quad\quad \mbox{ev}_{\mathbf{x}}^{\Gamma}(\beta)=(\beta(x))_{x\in\mathbf{x}}.
\end{equation*}
The map defined above is a Lie superalgebra homomorphism. In particular, it is an epimorphism. Corresponding to an $\mathbf{x}\in X_{*}$ we have the set $\{\rho_{x}:x\in\mathbf{x}\}$ of non zero representations $\rho_{x}:\g\rightarrow End(V_{x})$. Using these, the evaluation representation, denoted by $\mbox{ev}_{\mathbf{x}}^{\Gamma}(\rho_{x})_{x\in\mathbf{x}}$, is defined as 
\begin{equation*}
    (\g\otimes A)^{\Gamma}\xrightarrow{\mbox{ev}_{\mathbf{x}}^{\Gamma}}\bigoplus_{\mathbf{x}}\g^{\mathbf{x}}\xrightarrow{\otimes _{x\in\mathbf{x}}\rho_{x}} End\left(\bigotimes_{x\in\mathbf{x}}V_{x}\right).
\end{equation*}
    
\end{defn}
\begin{defn}(Generalized Evaluation module)\label{def:12}
Suppose $\textbf{m}_{1},\ldots,\textbf{m}_{l}\in \mbox{maxspec}\ A$ are pairwise disjoint and $n_{1},\ldots,n_{l}\in\mathbb{N}$. Then the associated generalized evaluation map is the composition 
    \begin{equation*}
        \mbox{ev}^{\Gamma}_{\textbf{m}_{1}^{n_{1}},\cdots,\textbf{m}_{l}^{n_{l}}}:(\g\otimes A)^{\Gamma}\rightarrow(\g\otimes A)/(\g\otimes \prod_{i=1}^{l}\textbf{m}_{i}^{n_i})\cong\bigoplus_{i=1}^{l}(\g\otimes (A/\textbf{m}_{i}^{n_i})).
    \end{equation*}
    If $V_{i}$ is a finite dimensional $(\g\otimes(A/\textbf{m}_{i}^{n_{i}}))$-module with corresponding representation $\rho_{i}:(\g\otimes(A/\textbf{m}_{i}^{n_{i}}))\rightarrow End (V_{i})$, then the composition
    \begin{equation}
        (\g\otimes A)^{\Gamma}\xrightarrow{\mbox{ev}^{\Gamma}_{\textbf{m}_{1}^{n_{1}},\cdots,\textbf{m}_{l}^{n_{l}}}}\bigoplus_{i=1}^{l}(\g\otimes (A/\textbf{m}_{i}^{n_i}))\xrightarrow{\otimes_{i=1}^{l}\rho_{i}} End(\otimes_{i=1}^{l}V_{i})
    \end{equation}
    is called generalized evaluation representation of $(\g\otimes A)^{\Gamma}$ and is denoted by $\mbox{ev}^{\Gamma}_{\textbf{m}_{1}^{n_1},\cdots,\textbf{m}_{l}^{n_{l}}}(\rho_{1},\cdots,\rho_{l})$.
\end{defn}
In the above definition of the generalized evaluation map, we denote $I_{\eta}=\prod_{x\in \Supp(\eta
)}\textbf{m}_{x}^{\eta(x)}$, where $\eta:X_{rat}\rightarrow\mathbb{N}.$ 
Hence we write
\begin{equation*}
    \mbox{ev}^{\Gamma}_{\eta}:(\g\otimes A)^{\Gamma}\longrightarrow\bigoplus_{x\in   \Supp\eta}(\g\otimes (A/\textbf{m}_{x}^{\eta(x)}))\xrightarrow{\cong}\bigoplus_{x\in \Supp\eta}(\g\otimes A)/(\g\otimes \textbf{m}_{x}^{\eta(x)}).
\end{equation*}
This is nothing but the restriction of the evaluation map $\mbox{ev}_{\eta}$ of $\g\otimes A$ to $(\g\otimes A)^{\Gamma}$ as given in \cite[Eq 5.1]{Sav14}. Hence 
\begin{equation*}
    \mbox{ker ev}_{\eta}^{\Gamma}=(\mbox{ker ev}_{\eta})\cap(\g\otimes A)^{\Gamma}=(\g\otimes I_{\eta})\cap(\g\otimes A)^{\Gamma}.
\end{equation*}
 For a finite subset $\textbf{x}\subset X$, define $I_{\textbf{x}}=I_{\eta}$, where $\eta(x)=1$ for $x\in\textbf{x}$ and $\eta(x)=0$ for $x\notin\textbf{x}$.
\begin{lemma}\label{lem:10}
    If $\eta:X_{rat}\rightarrow\mathbb{N}$ satisfies $\Supp~\eta\in X_{*}$, then
    \begin{equation*}
        (\g\otimes I_{\eta})^{\Gamma}=(\g\otimes \tilde I_{\eta})^{\Gamma}, \quad \mbox{where}\quad \tilde I_{\eta}=\prod_{x\in \Supp\eta}\prod_{\gamma\in\Gamma}\textbf{m}_{\gamma.x}^{\eta(x)}.
    \end{equation*}
\end{lemma}
\begin{proof}
    The proof follows similarly as in \cite[Lemma 2.1]{FKKS12}.
\end{proof}
We rewrite \cite[Lemma 5.6]{Sav14} as the following result.
\begin{thm}\label{thm:9.7}
For $\eta:X_{rat}\rightarrow\mathbb{N}$, with $\Supp~\eta\in X_{*}$, the map $\mbox{ev}^{\Gamma}_{\eta}$ is surjective.
\end{thm}
As a corollary to Theorem \ref{thm:9.7}, we get the following.
\begin{cor}\label{cor:9.8}
    \begin{equation*}
        (\g\otimes A)^{\Gamma}/(\g\otimes I_{\eta})^{\Gamma}\xrightarrow{\cong}(\g\otimes A)/(\g\otimes I_{\eta}).
    \end{equation*}
\end{cor}
Let $R(\g)$ \label{pg:8} denote the set of isomorphism classes of irreducible finite dimensional representations of $\g$. Then, define $\varepsilon(X,\g)$ to be the set of finitely supported functions $\Psi: X_{rat}\rightarrow R(\g)$ and let $\varepsilon(X,\g)^{\Gamma}$ be its subset consisting of all $\Gamma$-equivariant functions. Here the support $\Supp\ \Psi$ is the set of all $x\in X_{rat}$ for which $\Psi(x)\neq 0$, where 0 denotes the isomorphism class of the trivial representation.

For isomorphic representations $\rho$ and $\rho'$ of $\g$, the evaluation representations $\mbox{ev}_{x}\rho$ and $\mbox{ev}_{x}\rho'$ are isomorphic. If $[\rho]\in R(\g)$ is the isomorphism class of $\rho$, we can define $\mbox{ev}_{x}([\rho])$ to be the isomorphism class of $\mbox{ev}_{x}\rho$ and this is independent of the representative $\rho$. Similarly, for a finite subset $\textbf{x}\subseteq
 X_{rat}$ and the representations $\rho_{x}$ of $\g$ for $x\in\textbf{x}$, we define $\mbox{ev}_{\textbf{x}}([\rho_{x}])_{x\in\textbf{x}}$ to be the isomorphism class of $\mbox{ev}_{\textbf{x}}(\rho_{x})_{x\in\textbf{x}}$. For $\Psi\in\varepsilon(X,\g)$ and $\textbf{x}\in X_{*}$, we define the isomorphism class of $(\g\otimes A)$-module, given by $\mbox{ev}_{\Psi}=\mbox{ev}_{\textbf{x}}(\Psi(x))_{x\in\textbf{x}}$.
\begin{defn}(Class $\mbox{ev}_{\Psi}^{\Gamma}$)\label{def:13}
    For $\Psi\in\varepsilon(X,\g)^{\Gamma}$, $\mbox{ev}_{\Psi}^{\Gamma}=\mbox{ev}_{\mathbf{x}}^{\Gamma}(\Psi(x))_{x\in\mathbf{x}}$ where $\mathbf{x}\in X_{*}$ contains one element of each $\Gamma$ orbit in $\Supp\ \Psi$. By \cite[Lemma 5.9]{Sav14}, we have that $\mbox{ev}_{\Psi}^{\Gamma}$ is independent of the choice of $\mathbf{x}$. For different choices of $\mathbf{x}$, we get evaluation representations belonging to the isomorphism class of some representation of $(\g\otimes A)^{\Gamma}$. 
\end{defn}
\begin{prop}[\cite{Sav14,NeherSavageSenesi2012}]
 The map $\Psi\mapsto\mbox{ev}_{\Psi}^{\Gamma}$, from $\varepsilon(X,\g)^{\Gamma}$ to the set of isomorphism classes of evaluation representations of $(\g\otimes A)^{\Gamma}$, is injective.
\end{prop}
The tensor product of irreducible finite dimensional $(\g\otimes A)$-modules with disjoint support is irreducible (see \cite{Sav14}). As a consequence, we get the following proposition.
\begin{prop}[\cite{Sav14}]
If $\g$ is a basic classical Lie superalgebra, then $\mbox{ev}_{\Psi}^{\Gamma}$ is an irreducible representation for all $\Psi\in\varepsilon(X,\g)^{\Gamma}$. Hence the map $\Psi\mapsto\mbox{ev}_{\Psi}^{\Gamma}$ is a bijection from $\varepsilon(X,\g)^{\Gamma}$ to the set of isomorphism classes of irreducible evaluation representations of $(\g\otimes A)^{\Gamma}$. 
\end{prop}

\begin{rem}
    Since $\g\otimes\mathbb{K}\subset\g\otimes A$, any $(\g\otimes A)$-module $V$ can be viewed as a $\g$-module. 
\end{rem}
\begin{lemma}[\cite{NS}]\label{lem:12}
Suppose a finite abelian group $\Gamma$ acts on a unital associative commutative $\mathbb{K}$-algebra $A$ by automorphisms. Let $A=\oplus_{\xi\in\Xi}A_{\xi}$ be the associated grading on $A$. Then the following conditions are equivalent.
    \begin{enumerate}[label=(\alph*)]
        \item $\Gamma$ acts freely on $X$.
        \item  $\prod_{i=1}^{n}I_{\xi_{i}}=(I^{n})_{\Sigma_{i=1}^{n}\xi_{i}}$ for all $\xi_{1},\cdots,\xi_{n}\in\Xi$ and any $\Gamma$ invariant ideal $I$ of $A$.
    \end{enumerate}
\end{lemma}
For a Lie superalgebra $L$, define $L^{n},~n\geq 1$, by
\begin{equation*}
    L^{1}=L, \quad L^{n}=[L,L^{n-1}],\quad n>1.
\end{equation*}
\begin{prop}\label{prop:3}
Every finite dimensional $(\g\otimes A)^{\Gamma}$-module is annihilated by $(\g\otimes I_{\eta})^{\Gamma}$ for some $\eta:X_{rat}\rightarrow\mathbb{N}$ with $\Supp~\eta\subseteq X_{*}$.
\end{prop}
\begin{proof}
Let $V$ be a finite dimensional $(\g\otimes A)^{\Gamma}$-module annihilated by $(\g\otimes I_{\eta})^{\Gamma}$ for some finitely supported $\eta:X_{rat}\rightarrow\mathbb{N}$. It is clear from Lemma \ref{lem:10} that, we can find some $\eta':X_{rat}\rightarrow\mathbb{N}$ with $\Supp~\eta'\subseteq X_{*}$ and $(\g\otimes I_{\eta'})^{\Gamma}\subseteq(\g\otimes I_{\eta})^{\Gamma}$. Hence it is only necessary to prove that every finite dimensional $(\g\otimes A)^{\Gamma}$-module is annihilated by some $(\g\otimes I_{\eta})^{\Gamma}$.
   
We first prove that for any $\Gamma$ invariant ideal $I$ of $A$, $((\g\otimes I)^{\Gamma})^{m}=(\g\otimes I^{m})^{\Gamma},$ for all $m\geq 1.$
The proof is by induction. It is trivial for the case $m=1$. Let it be true for some $m>1$. Then 
\begin{align*}
        ((\g\otimes I)^{\Gamma})^{m+1}&=[(\g\otimes I)^{\Gamma},((\g\otimes I)^{\Gamma})^{m}]
        =[(\g\otimes I)^{\Gamma},(\g\otimes I^{m})^{\Gamma}]\\
       &=[\bigoplus_{\xi\in\Xi}\g_{\xi}\otimes I_{-\xi},\bigoplus_{\tau\in\Xi}\g_{\tau}\otimes (I^{m})_{-\tau}]=\sum
       _{\xi,\tau\in\Xi}[\g_{\xi},\g_{\tau}]\otimes I_{-\xi}(I^{m})_{-\tau}\\
       &=\sum_{\xi,\tau\in\Xi}[\g_{\xi},\g_{\tau}]\otimes (I^{m+1})_{-\xi-\tau}\quad \mbox{(by Lemma \ref{lem:12})}\\
       &=\bigoplus_{\xi'\in\Xi}\g_{\xi'}\otimes(I^{m+1})_{-\xi'}=(\g\otimes I^{m+1})^{\Gamma}.
    \end{align*}
Let $V$ be a finite dimensional $(\g\otimes A)^{\Gamma}$-module. Then there exists a finite  filtration 
    \begin{equation*}
        0=V_{0}\subseteq V_{1}\subseteq\cdots\subseteq V_{n}=V,
    \end{equation*}
such that $V_{i}/V_{i-1}$ is an irreducible finite dimensional $(\g\otimes A)^{\Gamma}$-module for $1\leq i\leq n$. Hence each $V_{i}/V_{i-1}$ is an generalized evaluation module (by \cite[Corollary 8.7 ]{Sav14}). Let $\eta_{i}:X_{rat}\rightarrow\mathbb{N}$ be the characteristic function of the support of $V_{i}/V_{i-1}$. Hence $(\g\otimes I_{\eta_{i}})^{\Gamma} V_{i}/V_{i-1}=0$. In other words $(\g\otimes I_{\eta_{i}})^{\Gamma}V_{i}\subseteq V_{i-1}$.

Let $\nu=\Sigma_{i=1}^{n}\eta_{i}$ and $\eta=n\nu$. Since $I_{\eta}=I_{\nu}^{n}$, we have $((\g\otimes I_{\nu})^{\Gamma})^{n}=(\g\otimes I_{\eta})^{\Gamma}.$ Because $I_{\nu}\subseteq I_{\eta_{i}}$, we have $(\g\otimes I_{\nu})^{\Gamma}V_{i}\subseteq V_{i-1}$ for all $1\leq i\leq n$. Therefore
    \begin{equation*}
        (\g\otimes I_{\eta})^{\Gamma}.V=((\g\otimes I_{\nu})^{\Gamma})^{n}V=0.
    \end{equation*}
\end{proof}
Let $\eta,\eta':X_{rat}\rightarrow\mathbb{N}$ with finite support. We will say that $\eta\leq\eta'$ if $\eta(x)\leq\eta'(x)$ for all $x\in X_{rat}$. Then we get $I_{\eta'}\subseteq I_{\eta}$. This happens because $I_{\eta}$ is an ideal. Hence,
\begin{equation*}
    \eta\leq\eta'\Rightarrow I_{\eta'}\subseteq I_{\eta}\Rightarrow (\g\otimes I_{\eta'})\subseteq(\g\otimes I_{\eta}).
\end{equation*}
Therefore, the projection map
\begin{equation*}
     (\g\otimes A)/(\g\otimes I_{\eta'})\rightarrow(\g\otimes A)/(\g\otimes I_{\eta}), \quad (\g\otimes A)^{\Gamma}/(\g\otimes I_{\eta'})^{\Gamma}\rightarrow(\g\otimes A)^{\Gamma}/(\g\otimes I_{\eta})^{\Gamma}.
\end{equation*}
\begin{lemma}\label{lem:18}
    If $\eta,\eta':X_{rat}\rightarrow\mathbb{N}$ are such that $\eta\leq\eta'$ and $\Supp\ \eta'\subseteq X_{*}$ then the following diagram commutes.\\
    \[
\begin{tikzcd}
(\g \otimes A)^\Gamma / (\g \otimes I_{\eta'})^\Gamma \arrow[r, "\cong"] \arrow[d, two heads] & (\g \otimes A) / (\g \otimes I_{\eta'}) \arrow[d, two heads] \\
(\g \otimes A)^\Gamma / (\g \otimes I_{\eta})^\Gamma \arrow[r, "\cong"] & (\g \otimes A) / (\g \otimes I_{\eta}).
\end{tikzcd}
\]
\end{lemma}
\smallskip

Let $V$ be a finite dimensional $(\g\otimes A)^{\Gamma}$-module. We have seen from Proposition \ref{prop:3} that, there exists $\eta:X_{rat}\rightarrow\mathbb{N}$ with $\Supp~\eta\subseteq X_{*}$, such that $(\g\otimes I_{\eta})^{\Gamma}.V=0$. Using Lemma \ref{lem:18}, we get
\begin{equation}\label{eq:5}
    (\g\otimes A)\rightarrow(\g\otimes A)/(\g\otimes I_{\eta})\cong(\g\otimes A)^{\Gamma}/(\g\otimes I_{\eta})^{\Gamma}\rightarrow\mbox{End}(V)
\end{equation}
which defines the action of $(\g\otimes A)$ on V. Hence we get a $(\g\otimes A)$-module and denote this by $V_{\eta}$ \label{pg:9}.
\begin{lemma}
    Suppose $V$ is a finite dimensional $(\g\otimes A)^{\Gamma}$-module that is annihilated by $(\g\otimes I_{\eta})^{\Gamma}$ and $(\g\otimes I_{\eta'})^{\Gamma}$ for functions $\eta,\eta':X_{rat}\rightarrow\mathbb{N}$ such that $\Supp\eta\cup \Supp\eta'\subseteq X_{*}$. Then $V_{\eta}=V_{\eta'}$ as $(\g\otimes A)$-modules.
\end{lemma}
\begin{proof}
 The proof follows similarly as in \cite[Lemma 2.6]{FKKS12}.
\end{proof}
\smallskip

Suppose $\mathcal{F}$ and $\mathcal{F}^{\Gamma}$ denote the categories of finite dimensional $(\g\otimes A)$ and $(\g\otimes A)^{\Gamma}$- modules respectively. 
Twisting and untwisting functors between full subcategories of $\mathcal{F}$ and $\mathcal{F}^{\Gamma}$ are defined and studied for when $\g$ is finite dimensional semisimple Lie algebra (\cite{FKKS12, FMS15}). Here we extend the definitions to Lie superalgebras. 

\begin{defn}(Categories $\mathcal{F}_{\mathbf{x}}$ and $\mathcal{F}_{\mathbf{x}}^{\Gamma}$) \label{def:14}.
       For $\mathbf{x}\in X_{*}$, let $\mathcal{F}_{\mathbf{x}}$ denote the full subcategory of category $\mathcal{F}$ whose objects are those $V \in \mathcal{F}$ having $\Supp_{A}\ (V)\subseteq\mathbf{x}$. Similarly let $\mathcal{F}_{\mathbf{x}}^{\Gamma}$ be the full subcategory of category $\mathcal{F}^{\Gamma}$ whose objects are those $V \in \mathcal{F}^{\Gamma}$ such that $\Supp_{A}^{\Gamma}\ (V)\subseteq\Gamma\mathbf{x}$.
\end{defn}
 Note that, for a finitely supported 
 $(\g\otimes A)$-module $V$, with $V^{\Gamma}$ being the $(\g\otimes A)^{\Gamma}$-module obtained via restriction, $\Supp^{\Gamma}_{A}V^{\Gamma}=\Gamma \Supp_{A}V$.
\begin{defn}(Twisting functors $\mathbf{T}$ and $\mathbf{T}_{\mathbf{x}})$ \label{def:15}.
The twisting functor $\mathbf{T}$ is from the category of $(\g\otimes A)$-modules $\mathcal{F}$ to $(\g\otimes A)^{\Gamma}$-modules $\mathcal{F}^{\Gamma}$ obtained by restriction. For any $\mathbf{x}\in X_{*}$, we have the induced functor $\mathbf{T}_{\mathbf{x}}:\mathcal{F}_{\mathbf{x}}\rightarrow\mathcal{F}_{\mathbf{x}}^{\Gamma}$.
\end{defn}
\begin{defn}(Untwisting functors)\label{def:16}.
Fix $\mathbf{x}\in X_{*}$. Here we define the untwisting functor $\mathbf{U}_{\mathbf{x}}:\mathcal{F}^{\Gamma}_{\mathbf{x}}\rightarrow\mathcal{F}_{\mathbf{x}}$ that maps the objects $V$ in $\mathcal{F}^{\Gamma}_{\mathbf{x}}$ to $V_{\eta}$ in $\mathcal{F}_{\mathbf{x}}$. Now we try to understand how the morphisms in the categories are mapped. Let $\beta:V\rightarrow W$ be a morphism in $\mathcal{F}^{\Gamma}_{\mathbf{x}}$ mapping two $(\g\otimes A)^{\Gamma}$-modules $V$ and $W$. We know from Proposition \ref{prop:3} that, any $(\g\otimes A)^{\Gamma}$ module $V$ is annihilated by $(\g\otimes I_{\eta})^{\Gamma}$ where $\Supp~\eta\subseteq X_{*}$. Then the action of $(\g\otimes A)^{\Gamma}$ on $V$ factors through $(\g\otimes A)^{\Gamma}/(\g\otimes I_{\eta})^{\Gamma}$. From \eqref{eq:5}, it is clear that this action induces an action by $(\g\otimes A)$ and hence $\beta$ can be considered as a homomorphism between $(\g\otimes A)$ modules $V_{\eta}$ and $W_{\eta}$. This homomorphism is denoted by $\mathbf{U}_{\mathbf{x}}(\beta)$.
\end{defn}
We have defined earlier 
$\varepsilon(X,\g)$ and $\varepsilon(X,\g)^{\Gamma}$ \label{pg:10}, the set of finitely supported function  $\Psi:X_{rat}\rightarrow R(\g)$ and the subset of finitely supported $\Gamma$-equivariant functions, respectively. Each $\Psi\in\varepsilon(X,\g)$ are functions that can also be equivalently defined as maps $\Psi:X_{rat}\rightarrow P^{+}$, where $P^{+}$ denote the dominant integral weights for $\g$. Hence corresponding to a rational point from the support of $\Psi$, we have a weight functional in $P^+$. For ease of notation, we will denote $\varepsilon$ to be $\varepsilon(X,\g)$ and $\varepsilon^{\Gamma}$ to be $\varepsilon(X,\g)^{\Gamma}$. 

For a $\Gamma$ invariant subset $Y$ of $X_{rat}$, let $Y_{\Gamma}$ \label{pg:11} denote the set of subsets of $Y$ containing exactly one point from each $\Gamma$ orbit in $Y$. For $\Psi\in\varepsilon^{\Gamma}$ and $\mathbf{x}\in (\Supp\ \Psi)_{\Gamma}$, we have
\begin{equation*}
    \Psi_{\mbox{x}}:X_{rat}\rightarrow P^{+}, \quad \Psi_{\mathbf{x}}(x)= \begin{cases}
    \Psi(x), & \text{if } x \in \mathbf{x}, \\
0, & \text{if } x \notin \mathbf{x}.

    \end{cases}
\end{equation*}
For $\Psi\in\varepsilon$, define
\begin{equation*}
    \wt(\Psi)=\sum_{x\in \Supp\Psi}\Psi(x)\in P^{+}.
\end{equation*}
For $\Psi\in\varepsilon^{\Gamma}$, define
\begin{equation*}
    \wt_{\Gamma}(\Psi)=(\wt(\Psi_{\mathbf{x}}))|_{\h_{\Gamma}} \quad \mbox{for}~\mathbf{x}\in (\Supp\ \Psi)_{\Gamma}.
\end{equation*}
Let $\lambda\in P^{+},$ then $\lambda=\sum_{i\in\mathcal{J}}k_{i}\alpha_{i}$ where $\alpha_{i}\in\phi$ and $k_{i}\in\mathbb{Z}_{+}$. That is, $\lambda$ can be written as the non negative linear combination of simple  roots of $\g$. We define the height of $\lambda$ as,\label{pg:13}
\begin{equation*}
    \het\ \lambda=\sum_{i\in\mathcal{J}}k_{i}.
\end{equation*}
Now for $\lambda'\in \Lambda^{+}$, we have $\lambda'=\sum_{\textbf{i}\in\mathcal{J}^{\Gamma}}k_{\textbf{i}}\alpha_{\textbf{i}}$ where $\alpha_{\textbf{i}}\in\Delta$ and  $k_{\textbf{i}}\in\mathbb{Z}_{+}.$ We define the height of $\lambda^{'}$ to be
\begin{equation*}
\het_{\Gamma}\ \lambda'=\sum_{\textbf{i}\in\mathcal{J}^{\Gamma}}k_{\textbf{i}}.
\end{equation*}
For $\Psi\in\varepsilon$, define
\begin{equation}
    \het\ \Psi=\het\ (\wt(\Psi)).
\end{equation}
Similarly, for $\Psi\in\varepsilon^{\Gamma}$, we have
\begin{equation}
    \het_{\Gamma}\Psi=\het_{\Gamma}(\wt_{\Gamma}(\Psi)).
\end{equation}
\begin{thm}\label{thm:3}
    For $\mathbf{x}\in X_{*}$, the twisting and the untwisting functors have the following properties.
    \begin{enumerate}[label=(\alph*)]
        \item The twisting functor $\textbf{T}_{\mathbf{x}}$ maps the isomorphism class $\mbox{ev}_{\Psi}$ for $\Psi\in\varepsilon$, $\Supp\Psi\in X_{*}$, to the isomorphism class $\mbox{ev}^{\Gamma}_{\Psi^{\Gamma}}$ for  $\Psi^{\Gamma}\in\varepsilon^{\Gamma}$, where
        \begin{equation*}
            \Psi^{\Gamma}(x)=\sum_{\gamma\in\Gamma}\gamma \Psi(\gamma^{-1}x),\quad x\in X_{rat}.
        \end{equation*}
        \item The untwisting functor $\textbf{U}_{\mathbf{x}}$ maps the isomorphism class of $\mbox{ev}^{\Gamma}_{\Psi}$, $\Psi\in\varepsilon^{\Gamma}$, to the isomorphism class $\mbox{ev}_{\Psi_{\mathbf{x}}}$.
        \item The functor $\textbf{T}_{\mathbf{x}}$ and $\textbf{U}_{\mathbf{x}}$ are mutually inverse isomorphism of categories.
    \end{enumerate}
\end{thm}
\begin{proof}
\begin{enumerate}[label=(\alph*)]
    \item Let $\Psi\in\varepsilon$ and $\Supp\ \Psi\in X_{*}$ and take $\Psi^{\Gamma}$ as 
        $\Psi^{\Gamma}(x)=\sum_{\gamma\in\Gamma}\gamma\Psi(\gamma^{-1}x).$
    For some $\mu\in\Gamma$,
\begin{equation*}
        \mu\Psi^{\Gamma}(x)=\sum_{\gamma\in\Gamma}\mu\gamma\Psi(\gamma^{-1}x)
   \quad \mbox{and} \quad
\Psi^{\Gamma}(\mu x)=\sum_{\gamma\in\Gamma}\gamma \Psi(\gamma^{-1}(\mu x)).
    \end{equation*}
    Let $\mu^{-1}\gamma=\tau$ then $\tau^{-1}=\gamma^{-1}\mu$. Clearly
    \begin{align*}
         \Psi^{\Gamma}(\mu x)&=\sum_{\gamma\in\Gamma}\mu(\mu^{-1}\gamma)\Psi((\mu^{-1}\gamma)^{-1} x)=\sum_{\tau\in\Gamma}\mu\tau\Psi(\tau^{-1}x)\\
        &=\mu\sum_{\tau\in\Gamma}\tau\Psi(\tau^{-1}x)=\mu\Psi^{\Gamma}(x),
    \end{align*}
    and hence $\Psi^{\Gamma}(x)\in\varepsilon^{\Gamma}$. From the definition of $\Psi^{\Gamma}$ we can see that  $\Supp(\Psi^{\Gamma})\subseteq\Gamma.\Supp(\Psi)$. Let $\textbf{x}\in X_{*}$ and $V\in\mathcal{F}_{\textbf{x}}$ is irreducible and corresponds to $\Psi\in\varepsilon$. Let $\mbox{ev}_{\Psi}=(\otimes_{x\in\textbf{x}}\rho_x)\circ\mbox{ev}_{\textbf{x}}$ be the corresponding evaluation representation. Then we have seen that this representation factors through $(\g\otimes A)/(\g\otimes I_{\textbf{x}})$. Using Corollary \ref{cor:9.8} we get $\textbf{T}_{\textbf{x}}(V)$ to be the $(\g\otimes A)^{\Gamma}$-module given by
    \begin{equation}
    \begin{split}
        (\g\otimes A)^{\Gamma}\rightarrow(\g\otimes A)^{\Gamma}/(\g\otimes I_{\textbf{x}})^{\Gamma}\xrightarrow{\cong}(\g\otimes A)/(\g\otimes I_{\textbf{x}})=(\g\otimes A)/(\g\otimes \prod_{x\in\textbf{x}}m_{x})\\\cong\bigoplus_{x\in\textbf{x}}(\g\otimes(A/m_{x}))\xrightarrow{\rho=\otimes_{x\in\textbf{x}}\rho_{x}}End(V).
    \end{split}    
    \end{equation}
    We know already from Lemma \ref{lem:10} that $(\g\otimes I_{\textbf{x}})^{\Gamma}=(\g\otimes \tilde I_{\textbf{x}})^{\Gamma}$ where $\tilde I_{\textbf{x}}=\prod_{x\in\textbf{x}}\prod_{\gamma\in\Gamma}m_{\gamma.x}$. Hence combining this equality with the above equation we get the evaluation representation $(\otimes_{x\in\textbf{x}}\rho_{x})\circ\mbox{ev}^{\Gamma}_{\Gamma.\textbf{x}}$ for $(\g\otimes A)^{\Gamma}$ which is in the isomorphism class of $\mbox{ev}^{\Gamma}_{\Psi^{\Gamma}}$. This completes the proof of (i).
    \item Let $\textbf{x}\in X_{*}$ and $V'\in\mathcal{F}_{\textbf{x}}^{\Gamma}$ be an irreducible module  corresponding to $\Psi\in\varepsilon^{\Gamma}$. Let $\mbox{ev}_{\Psi}^{\Gamma}=(\otimes_{x\in\textbf{x}}\rho_{x})\circ\mbox{ev}_{\textbf{x}}^{\Gamma}$ be the corresponding evaluation representation. Then we get $\textbf{U}_{\textbf{x}}(V')$ to be the $(\g\otimes A)-$ mdoule given by the composition
    \begin{equation*}
        (\g\otimes A)\rightarrow(\g\otimes
        A)/(\g\otimes I_{\textbf{x}})\xrightarrow{\cong}(\g\otimes A)^{\Gamma}/(\g\otimes I_{\textbf{x}})^{\Gamma}\cong\g^{\textbf{x}}\xrightarrow{\otimes_{x\in\textbf{x}}\rho_{x}}End(V').
    \end{equation*}
    This is precisely the evaluation representation of $(\g\otimes A)$ and lies in the isomorphism class of $\mbox{ev}_{\Psi_\textbf{x}}$. Hence (ii) follows.
    \item Suppose $V\in \mathcal{F}_{\textbf{x}}$. Then $V$ is annihilated by some $(\g\otimes I_{\eta})$ and the action of $(\g\otimes A)$ on $\textbf{U}_{\textbf{x}}\textbf{T}_{\textbf{x}}(V)$ is given by
    \begin{equation*}
        (\g\otimes A)\rightarrow(\g\otimes A)/(\g\otimes I_{\eta})\xrightarrow{\cong}(\g\otimes A)^{\Gamma}/(\g\otimes I_{\eta})^{\Gamma}\xrightarrow{\cong}(\g\otimes A)/(\g\otimes I_{\eta})\rightarrow End(V)
    \end{equation*}
    where the two isomorphisms are mutually inverse. Hence $\textbf{U}_{\textbf{x}}\textbf{T}_{\textbf{x}}(V)=V$. Let $\beta$ be a morphism in $\mathcal{F}_{\textbf{x}}$. So $\beta:V\rightarrow W$ be a homomorphism of $(\g\otimes A)-$modules. There exists an $\eta: X_{rat}\rightarrow\mathbb{N}$ with support contained in $\textbf{x}$ such that $(\g\otimes I_{\eta})$ annihilates both $V$ and $W$. By restricting to $(\g\otimes A)^{\Gamma}$, we get $\beta$ to be a morphism in $\mathcal{F}_{\textbf{x}}^{\Gamma}$ denoted by $\textbf{T}_{\textbf{x}}(\beta)$. If we denote $V'$ and $W'$ to be the corresponding $(\g\otimes A)^{\Gamma}$-modules, then again this action of $(\g\otimes A)^\Gamma$ factors through $(\g\otimes A)^{\Gamma}/(\g\otimes I_{\eta})^{\Gamma}\cong(\g\otimes A)/(\g\otimes I_{\eta})$, giving us the morphism $\textbf{T}_{\textbf{x}}(\beta)$ to be a morphism from $V'_{\eta}$ to $W'_{\eta}$. This is exactly $\textbf{U}_{\textbf{x}}\textbf{T}_{\textbf{x}}(\beta)$, morphism of $(\g\otimes A)-$modules. From \eqref{eq:5} it is clear that $V'_{\eta}=V$ and $W'_{\eta}=W$. Hence $\textbf{U}_{\textbf{x}}\textbf{T}_{\textbf{x}}$ is an identity on morphisms and is therefore the identity functor on $\mathcal{F}_{\textbf{x}}$. Similarly $\textbf{T}_{\textbf{x}}\textbf{U}_{\textbf{x}}$ is the identity functor on $\mathcal{F}^{\Gamma}_{\textbf{x}}$. Hence the proof.   
\end{enumerate}
\end{proof}
Let $\alpha_{i}$ be the simple root and $\omega_{i}$ be the fundamental weight of $\g$ corresponding to $i\in \mathcal{J}$. We have already defined $\g^{\Gamma}$ in such a way that the set of nodes $\mathcal{J}^{\Gamma}$ can be identified naturally with the sets of $\Gamma$ orbits in $\mathcal{J}$. For $\textbf{i}\in \mathcal{J}^{\Gamma}$, define
\begin{equation}
    x^{+}_{\textbf{i}}=\sqrt{k_{\textbf{i}}}\sum_{i\in\textbf{i}}X_{i},\quad x^{-}_{\textbf{i}}=\sqrt{k_{\textbf{i}}}\sum_{i\in\textbf{i}}X^{-}_{i},\quad h_{\textbf{i}}=k_{\textbf{i}}\sum_{i\in\textbf{i}}H_{i}
\end{equation}
where $X^{+}_{i}, X^{-}_{i}, H_{i}$ are either even roots of $\g$, that generate $\mbox{sl}_{2}(\mathbb{C})$, or odd roots of $\g$ such that $[X^{+}_{i}, X^{-}_{i}]=H_{i}$. If $X_{i}^{+}, X^{-}_{i}, H_{i}$ are the even roots, then the value of $k_{\textbf{i}}$ depends on $\g^{\Gamma}_{\bar{0}}$. That is,  $k_{\textbf{i}}=2$, if $\g_{\bar{0}}$ is of type $A_{2n}$, $\Gamma$ acts non trivially and $\textbf{i}$ corresponds to the short root of $\g^{\Gamma}_{\bar{0}}$. Otherwise $k_{\textbf{i}}=1$ \cite{FMS15}. All these triples together will generate $\g^{\Gamma}$.
\smallskip

Let $\alpha_{\textbf{i}}$ and $\omega_{\textbf{i}}$ be the simple root and the fundamental weight of $\g^{\Gamma}$, respectively, corresponding to $\textbf{i}\in \mathcal{J}^{\Gamma}$. Thus we get,
\begin{equation}\label{eq:10}
    \Lambda=\bigoplus_{\textbf{i}\in \mathcal{J}^{\Gamma}}\mathbb{Z}\omega_{\textbf{i}},\quad \Lambda^{+}=\bigoplus_{\textbf{i}\in \mathcal{J}^{\Gamma}}\mathbb{N}\omega_{\textbf{i}},\quad Q_{\Gamma}=\bigoplus_{\textbf{i}\in \mathcal{J}^{\Gamma}}\mathbb{Z}\alpha_{\textbf{i}},\quad Q^{+}_{\Gamma}=\bigoplus_{\textbf{i}\in \mathcal{J}^{\Gamma}}\mathbb{N}\alpha_{\textbf{i}}.
\end{equation}
These are the integral weight lattices, dominant integral weight lattice, root lattice and positive root lattice of $\g^{\Gamma}$ respectively.

\begin{lemma}\label{lem:13}
 Let $\Gamma$ be the cyclic automorphism group acting on the Lie superalgebra $\g$ by diagram automorphisms.
\begin{enumerate}[label=(\alph*)]
\item For all $i\in \mathcal{J}$, we have $\alpha_{i}|_{\h_{\Gamma}}=\alpha_{\Gamma i}$ and $\omega_{i}|_{\h_{\Gamma}}=k_{\Gamma i}\omega_{\Gamma i}$.
\item For all $\Psi\in\varepsilon^{\Gamma}$ and $\mathbf{x}\in (\Supp_{\Psi})_{\Gamma}$, we have $\het_{\Gamma}\Psi=\het\Psi_{\mathbf{x}}$.
\item Let $\lambda'\in P^{+}$ and set $\lambda=\lambda'|_{\h_{\Gamma}}$. Then, for $V\in \mathcal{I}_{\lambda'}$, we have $\textbf{T}(V)\in\mathcal{I}_{\lambda}^{\Gamma}$ and $V_{\lambda'}=\textbf{T}(V)_{\lambda}$ as vector spaces.
    \end{enumerate}
\end{lemma}
\begin{proof}
\begin{enumerate}[label=(\alph*)]
\item 
From triangular decompositions, we have $\h_{\Gamma}\subset\h$. Let $\alpha_{i}\in\h^{*},i\in\mathcal{J}$, be a simple root of $\g$. The action of $\gamma\in\Gamma$ on $\alpha_{i}$ can be defined as $\gamma.\alpha_{i}=\alpha_{\gamma(i)}$. Since $\Gamma$ is a group acting on the Lie superalgebra $(\g\otimes A)$, the group action on $\h_{\Gamma}^{*}$ can also be defined as $\gamma(\alpha(h))=\alpha(\gamma^{-1}h)$. For $h\in\h_{\Gamma}$, we get $\alpha_{\gamma(i)}(h)=\gamma.\alpha_{i}(h)=\alpha_{i}(\gamma^{-1}h)=\alpha_{i}(h)$. Therefore, we find that a simple root of $\g$ remains invariant in an orbit of $\mathcal{J}$. We have, $\alpha_{i}|_{\h_{\Gamma}}=\alpha_{\Gamma i}=\alpha_{\textbf{i}}$. Now
for $h_{\textbf{j}}\in\h_{\Gamma}$,
$\omega_{\Gamma i}(h_{\textbf{j}})=\omega_{\textbf{i}}(h_{\textbf{j}})=\delta_{\textbf{i}\textbf{j}}$ and then $\omega_{i}|_{\h_{\Gamma}}$ implies
\begin{equation*}
\omega_{i}\left(k_{\textbf{i}}\sum_{j\in\textbf{i}}H_{j}\right)=k_{\textbf{i}}\sum_{j\in\textbf{i}}\omega_{i} (H_{j})=k_{\textbf{i}}\delta_{ij}=k_{\textbf{i}}\omega_{\textbf{i}} .
\end{equation*}
\item Let $\Psi\in\varepsilon^{\Gamma}$ and $\textbf{x}\in (\Supp\ \Psi)_{\Gamma}$. Denote $\lambda\in P^{+}$ to be the $\wt(\Psi_{\textbf{x}})$. Then, by definition $\lambda=\sum_{i\in\mathcal{J}}k_{\i}\alpha_{i}$, for $k_{i}\in\mathbb{Z}_{+}$ and $\alpha_{i}\in\phi$. Then $\het(\wt\Psi_{\textbf{x}})=\sum k_{i}$. By definition $\het_{\Gamma}\Psi=\het_{\Gamma}((\wt(\Psi_{\textbf{x}}))|_{\h_{\Gamma}})$. Then $\wt(\Psi_{\textbf{x}})|_{\h_{\Gamma}}=\lambda|_{\h_{\Gamma}}=\sum_{i\in\mathcal{J}}k_{i}\alpha_{i}|{_{\h_{\Gamma}}}$. This would imply $\het_{\Gamma}\Psi=\sum k_{i}$. Hence $\het_{\Gamma}\Psi=\het\Psi_{\textbf{x}}$. 
        \item Let $V\in\mathcal{I}_{\lambda'}$. Then the weights of $V$ is less than (or equal to) $\lambda'$. This implies that $(\n^{+}\otimes A)v_{\lambda'}=0$. From part(i) we can say that $(\n^{+}\otimes A)^{\Gamma}v_{\lambda}=0$, for $\lambda=\lambda'|_{\h_{\Gamma}}$, and hence the $\g^{\Gamma}$ weights of $\textbf{T}(V)$ lie in $\lambda-Q_{\Gamma}^{+}$. This implies that $\textbf{T}(V)\in\mathcal{I}_{\lambda}^{\Gamma}$. As a consequence of the definition of $\lambda$, it can also be observed that the highest weight spaces of $V$ and $\textbf{T}(V)$ are the same. That is, $V_{\lambda'}=\textbf{T}(V)_{\lambda}$ as vector spaces.
    \end{enumerate}
\end{proof}
Now we define the untwisted local Weyl modules with respect to $\Psi\in\varepsilon$. The definition of the local Weyl modules, corresponding to a dominant integral weight, has already been given in \cite{BCM19}. Here we give a definition of (untwisted) local Weyl module for Lie superalgebras, analogous to that provided in \cite{FKKS12}, for the case of Lie algebras. 
\begin{defn}(Untwisted local Weyl module)\label{def:1}
Given $\Psi\in\varepsilon$, the untwisted local weyl module $W(\Psi)$ is $(\g\otimes A)-$ module generated by a non zero vector $w_{\Psi}$ satisfying the relations
\begin{equation*}
        (\n^{+}\otimes A) w_{\Psi}=0
    \end{equation*}
\begin{equation}\label{eq:6}
        (x_{i}^{-}\otimes 1)^{\lambda(h_{i})+1} w_\Psi=0, i\in \mathcal{J},\quad \mbox{where}~\lambda=\wt(\Psi)=\sum_{x\in \Supp\Psi}\Psi(x)
    \end{equation}
\begin{equation*}
        \beta w_{\Psi}=\Bigl(\sum_{x\in \Supp\ \Psi}\Psi(x)(\beta(x))\Bigl)w_{\Psi}, \quad \beta\in(\h\otimes A).
    \end{equation*}
\end{defn}
Defining the untwisted local Weyl module in this manner, we will show that the twisted local Weyl module can be found as its image under the twisting functor.

Let $I_{\psi}$,~$\psi\in(\h\otimes A)^{*}$, be the sum of all ideal $I\subseteq A$ such that $(\g\otimes I)w_{\psi}=0$. 
\begin{prop}\label{prop:5}
Let $\psi_{1}, \psi_{2}\in(\h\otimes A)^{*},~\psi_{1}|_{\h}=\lambda,~\psi_{2}|_{\h}=\mu$ and suppose that $\lambda,~\mu\in P^{+}$ are such that $\lambda+\mu\in P^{+}$. If $\Supp(I_{{\psi}_{1}})\cap \Supp(I_{{\psi_{2}}})=\emptyset$, then we have
    \begin{equation*}
        W_{A}^{loc}(\psi_{1}+\psi_{2})\cong W_{A}^{loc}(\psi_{1})\otimes W_{A}^{loc}(\psi_{2})
    \end{equation*}
    as $(\g\otimes A)$-modules.
\end{prop}
\begin{rem}\label{rem:2}
From Definition  \ref{def:1} and Proposition \ref{prop:5}, we can see that the untwisted local Weyl module corresponding to $\Psi\in\varepsilon$, such that $\wt(\Psi)=\lambda$, is isomorphic to the tensor product of local Weyl modules associated to some dominant integral weights.
\end{rem}
We have already defined the twisted local Weyl module corresponding to $\psi\in((\h\otimes A)^{\Gamma})^{*}$ in terms of the generator relations. That is the local Weyl module at a single point. Now we give the definition corresponding to $\Psi\in\varepsilon^{\Gamma}$ in terms of the action of the twisting Weyl functor on an irreducible $\A$-module. In order to do this, we need to introduce certain new modules. These are going to be $(\g\otimes A)$ or $(\g\otimes A)^{\Gamma}$ modules defined with respect to $\Psi\in\varepsilon$ or $\Psi\in\varepsilon^{\Gamma}$, respectively.
\begin{defn}(Module $V^{\Gamma}(\Psi),V(\Psi)$)\label{def:17}
    Let $\Psi\in\varepsilon$ and $x_{1},\ldots,x_{k}\in \Supp\ \Psi$ such that $\wt(\Psi)=\lambda\in P^{+}$. Then, by the definition of $\Psi$, we know that each $\Psi(x_{i})\in P^{+}$. Let ${ (V(\Psi(x_{i})))}_{i=1,\cdots,k}$ be the collection of finite dimensional irreducible $(\g\otimes A)$-module of highest weight $\Psi(x_{i})$ and having disjoint support. Then from Corollary 4.13 and 4.14 in \cite{Sav14}, we know that, $V(\Psi)=V(\Psi(x_{1}))\otimes\cdots\otimes V(\Psi(x_{k}))$ is going to be a finite dimensional irreducible $(\g\otimes A)$-module. Similarly, for $\Psi\in\varepsilon^{\Gamma}$, let $V^{\Gamma}(\Psi)=V^{\Gamma}(\Psi(x_{1}))\otimes\cdots\otimes V^{\Gamma}(\Psi(x_{k}))$ denote the corresponding irreducible finite dimensional $(\g\otimes A)^{\Gamma}$-module.
\end{defn}
\begin{rem}
    From Lemma \ref{lem:3} and \ref{lem:4}, we conclude that the irreducible $(\g\otimes A)^{\Gamma}$-module $V^{\Gamma}(\Psi)$ is the restriction of the irreducible $(\g\otimes A)$-module $V(\Psi)$ to $(\g\otimes A)^{\Gamma}$. 
\end{rem}
\begin{rem}
    Hereafter, for $\Psi'\in\varepsilon$ $(\Psi\in\varepsilon^{\Gamma})$, when we define the untwisted (twisted)  global Weyl module or Weyl functor with respect to $\lambda$, we mean $\lambda=\wt(\Psi')$ $(\lambda=\wt_{\Gamma}(\Psi))$.
\end{rem}
\begin{defn}(Map $V_{\lambda}^{\Gamma}$)\label{def:18}
    We have already seen that, since $\A$ is commutative, any irreducible representation of $\A$ is one dimensional. If $M$ is any such one dimensional $\A$ representation, then let $V_{\lambda}^{\Gamma}M$ denote the unique irreducible quotient of $\mathbf{W}_{\lambda}^{\Gamma}M$. Hence we define a map $V_{\lambda}^{\Gamma}$ from the set of irreducible $\A$ modules to the set of irreducible finite dimensional $(\g\otimes A)^{\Gamma}$ modules.
\end{defn}
\begin{defn}(Modules $M(\Psi')$ and $M^{\Gamma}(\Psi)$)\label{def:2}For $\Psi'\in\varepsilon$, we define $M(\Psi')=R_{\lambda'}V(\Psi')$ and for $\Psi\in\varepsilon^{\Gamma}$, define $M^{\Gamma}(\Psi)=R_{\lambda}^{\Gamma}V^{\Gamma}(\Psi)$, where $R_{\lambda'}$ is the untwisted restriction functor \cite{BCM19} and $R_{\lambda}^{\Gamma}$ is the twisted restriction functor.
\end{defn}
Let $\varepsilon_{\lambda'}$ \label{pg:12} be the set of all $\Psi'\in\varepsilon$ such that $\wt(\Psi')=\lambda'$. Similarly, let $\varepsilon_{\lambda}^{\Gamma}$ be the set of all $\Psi\in\varepsilon^{\Gamma}$ such that $\wt_{\Gamma}(\Psi)=\lambda.$
\begin{prop}\label{prop:4}
    \begin{enumerate}[label=(\alph*)]
\item For $\Psi\in\varepsilon^{\Gamma}_{\lambda}$, we have $V^{\Gamma}(\Psi)\cong V_{\lambda}^{\Gamma}M^{\Gamma}(\Psi)$.
\item For $\lambda\in \Lambda^{+}$, the maps $V_{\lambda}^{\Gamma}$ and $R_{\lambda}^{\Gamma}$ are mutually inverse bijections, from the set of irreducible finite dimensional $(\g\otimes A)^{\Gamma}$-modules whose highest weight as $\g^{\Gamma}$ module is $\lambda$ and the set of irreducible finite dimensional $\A$ modules.
    \end{enumerate}
\end{prop}
\begin{proof}
    \begin{enumerate}[label=(\alph*)]
\item We know that $W_{\lambda}^{\Gamma}M^{\Gamma}(\Psi)=W_{\lambda}^{\Gamma}R_{\lambda}^{\Gamma}V^{\Gamma}(\Psi)$. As in the proof of Theorem \ref{thm:8.4}, we can find a surjective map $\epsilon:W_{\lambda}^{\Gamma}R_{\lambda}^{\Gamma}V^{\Gamma}(\Psi)\rightarrow V^{\Gamma}(\Psi)$. Thus $V^{\Gamma}(\Psi)$ must be isomorphic to the irreducible quotient of $W_{\lambda}^{\Gamma}M^{\Gamma}(\Psi)$, which is $V_{\lambda}^{\Gamma}M^{\Gamma}(\Psi)$.
\item We know that every irreducible finite dimensional $(\g\otimes A)^{\Gamma}$-module having highest weight $\lambda$ is given by $V^{\Gamma}(\Psi)$. Thus by the above part $V_{\lambda}^{\Gamma}R_{\lambda}^{\Gamma}$ is the identity map on the set of such modules. Now for finite dimensional  irreducible $\A$ module $M$, we have that
\begin{equation*}
R_{\lambda}^{\Gamma}V_{\lambda}^{\Gamma}M=(V_{\lambda}^{\Gamma}M)_{\lambda}=(W_{\lambda}^{\Gamma}M)_{\lambda}=W^{\Gamma}(\lambda)_{\lambda}\otimes_{\A} M\cong \A\otimes_{\A} M =M.
\end{equation*}
\end{enumerate}
\end{proof}
From Proposition \ref{prop:4}, we conclude that $M^{\Gamma}(\Psi)$, given in Definition \ref{def:2}, is an irreducible $\A$-module. Just as in Section \ref{sec:8}, we give an equivalent definition for the local Weyl modules in terms of the action of the Weyl functor on an irreducible $\A$-module.  $W(\lambda)$ denotes the untwisted global Weyl module and $W_{\lambda}$ denotes the untwisted Weyl functor (see \cite{BCM19}).  
\begin{defn} [Local Weyl Modules $W(\Psi)$ and $W^{\Gamma}(\Psi)$]\label{def:3}
   For $\Psi\in\varepsilon$ with $wt(\Psi)=\lambda'$, we define the local Weyl module to be $W(\Psi)=W_{\lambda}M(\Psi)$. Similarly, for  $\Psi\in\varepsilon^{\Gamma}$ with $wt_{\Gamma}(\Psi)=\lambda$, we define the twisted local Weyl module to be $W^{\Gamma}(\Psi)=W_{\lambda}^{\Gamma}M^{\Gamma}(\Psi)$.
\end{defn}
As already discussed in Remark \ref{rem:2}, the untwisted local Weyl modules corresponding to $\Psi$ is the tensor product of the local Weyl modules defined at each individual point. We have already proved in Theorem \ref{thm:4} that the local Weyl module $\LW$, for $\psi\in((\h\otimes A)^{\Gamma})^{*}$, is finite dimensional. Finite tensor product of finite dimensional spaces are again finite dimensional. Hence the local Weyl modules, $W(\Psi)$ are finite dimensional $(\g\otimes A)$  modules. As such, from Proposition \ref{prop:3} (taking $\Gamma$ trivial),  we can talk about the existence of annihilating ideals $(\g\otimes I_{\Psi})$, of finite codimension, for $W(\Psi)$. From the definition of $I_{\Psi}$, we get, $W(\Psi)\in\mathcal{F}_{\textbf{x}}$ for $\Psi\in\varepsilon$ and $ \Supp\ (\Psi)\in\textbf{x}$.
\smallskip

In the next proposition, we show that the untwisted local Weyl module given in Definition \ref{def:3} is equivalent to Definition \ref{def:1}. 
 \begin{prop}
Let $\Psi\in\varepsilon$ such that $\wt(\Psi)=\lambda$ and $W(\Psi)$ be the untwisted local Weyl module generated by $w_{\Psi}$ satisfying the relation \eqref{eq:6}. Then $W(\Psi)\cong W_{\lambda}M(\Psi)$, where $W_{\lambda}$ is the (untwisted) Weyl functor corresponding to the dominant integral weight $\lambda$.
\end{prop}
\begin{proof}
From the generator relation definition of the local Weyl module, we know that, $W(\Psi)=\mathbf{U}(\g\otimes A)w_{\Psi}$. Elements of $(\h\otimes A)$ acts on $W(\Psi)$ by scalar multiples and this is given by the equation
    \begin{equation*}
        \beta w_{\Psi}=\biggl(\sum_{x\in Supp\ \Psi}\Psi(x)(\beta(x))\biggl)w_{\Psi}, \quad \beta\in(\h\otimes A).
    \end{equation*}
   Let $V(\Psi)$ be the highest weight $(\g\otimes A)-$module. Then $R_{\lambda}(V(\Psi))=(V(\Psi))_{\lambda}=\mathbf{U}(\h\otimes A)v_{\Psi}=\mathbb{K}v_{\Psi}$, where $v_{\Psi}$ is the highest weight vector for $V(\Psi)$ given by $v_{\Psi}=v_{\Psi(x_{1})}\otimes\cdots\otimes v_{\Psi(x_{k})}$. As $\h\otimes A$ modules, $(W(\Psi))_{\lambda}\xrightarrow{\cong}(V(\Psi))_{\lambda}$, mapping $w_{\Psi}\rightarrow v_{\Psi}$. Hence there exists a $(\g\otimes A)$-module homomorphism $\varepsilon:W_{\lambda}R_{\lambda}V(\Psi)\rightarrow W(\Psi)$ such that 
   \begin{equation*}
       \varepsilon(uw_{\lambda}\otimes v_{\Psi})=uw_{\Psi} \quad \forall~u\in\mathbf{U}(\g\otimes A).
   \end{equation*}
   This is a surjective map.  By Remark \ref{rem:1},  $W_{\lambda}R_{\lambda}V(\Psi)$ is a finite dimensional $(\g\otimes A)$-module, generated by the highest weight vector $1\otimes v_{\Psi}$ and hence satisfies all the relations in \eqref{eq:6}. Hence we get a $(\g\otimes A)$ module homomorphism $\eta:W(\Psi)\rightarrow W_{\lambda}R_{\lambda}V(\Psi)$ such that $\eta(w_{\Psi})=1\otimes v_{\Psi}$. This is surjective. Hence $\eta\circ\varepsilon=id_{W_{\lambda}R_{\lambda}V(\Psi)}$ and $\varepsilon\circ\eta=id_{W(\Psi)} $. Thus we get $W(\Psi)\cong W_{\lambda}M(\Psi)$.
\end{proof}
Now, we give the necessary and sufficient conditions for a $(\g\otimes A)^{\Gamma}$-module to be a local Weyl module. In order to prove this, we need certain homological properties associated to the category $\mathcal{I}_{\lambda}^{\Gamma}$.
\begin{thm}
    Let $V\in \mathcal{I}_{\lambda}^{\Gamma}$. Then $V\cong W_{\lambda}^{\Gamma}R_{\lambda}^{\Gamma}V$ if and only if, for each $U\in\mathcal{I}_{\lambda}^{\Gamma}$ with $U_{\lambda}=0$, we have,
    \begin{equation}
        \Hom_{\mathcal{I}_{\lambda}^{\Gamma}}(V,U)=0, \quad \Ext^{1}_{\mathcal{I}_{\lambda}^{\Gamma}}=0.
    \end{equation}
\end{thm}
\begin{proof}
The proof follows similarly as in \cite[Theorem 4.10]{FMS15}.
\end{proof}
\begin{cor}\label{cor:3}
    Let $V\in \mathcal{I}_{\lambda}^{\Gamma}$. Then $V\cong W_{\lambda}^{\Gamma}R_{\lambda}^{\Gamma}V$ if and only if,
     \begin{equation}
        \Hom_{\mathcal{I}_{\lambda}^{\Gamma}}(V,U)=0, \quad \Ext^{1}_{\mathcal{I}_{\lambda}^{\Gamma}}=0.
    \end{equation}
     for all irreducible finite-dimensional $U\in\mathcal{I}^{\Gamma}_{\lambda}$ with $U_{\lambda}=0$. 
\end{cor}

\begin{lemma}\label{lem:11}
    A $(\g\otimes A)^{\Gamma}-$module $V$ is isomorphic to the local weyl module $W^{\Gamma}(\Psi)$ if and only if it satisfies the following condition:
    \begin{enumerate}[label=(\alph*)]
        \item $V\in\mathcal{I}_{\lambda}^{\Gamma}$, where $\lambda=\wt_{\Gamma}(\Psi)$
        \item $R_{\lambda}^{\Gamma}V\cong M^{\Gamma}(\Psi)$
        \item $\Hom_{\mathcal{I}_{\lambda}^{\Gamma}}(V,U)=0$ and $\Ext^{1}_{\mathcal{I}_{\lambda}^{\Gamma}}(V,U)=0$ for all irreducible finite dimensional $U\in\mathcal{I}_{\lambda}^{\Gamma}$ with $U_{\lambda}=0$.
    \end{enumerate}
\end{lemma}
\begin{proof}
Let V be a $(\g\otimes A)^{\Gamma}-$module such that it satisfies the above conditions. Then by Corollary \ref{cor:3}, we get 
\begin{equation*}
V\cong W_{\lambda}^{\Gamma}R_{\lambda}^{\Gamma}V\cong W_{\lambda}^{\Gamma}M^{\Gamma}(\Psi)=W^{\Gamma}(\Psi),
\end{equation*}
where the second isomorphism follows from (ii) listed above and the third equality follows from the definition of $W^{\Gamma}(\Psi)$.
     
Conversely, the first condition follows directly from the definition of $W^{\Gamma}(\Psi)$. $R_{\lambda}^{\Gamma}W^{\Gamma}(\Psi)=R_{\lambda}^{\Gamma}W_{\lambda}^{\Gamma}M^{\Gamma}(\Psi)\cong M^{\Gamma}(\Psi)$ (by Proposition \ref{prop:4}). This proves (ii).\\ $W_{\lambda}^{\Gamma}R_{\lambda}^{\Gamma}W^{\Gamma}(\Psi)\cong W_{\lambda}^{\Gamma}M^{\Gamma}(\Psi)= W^{\Gamma}(\Psi)$. Hence by Corollary \ref{cor:3}, $W^{\Gamma}(\Psi)$ satisfies (iii) also. This completes the proof.
\end{proof}
Now we prove the main theorem of this section.
\begin{thm}
Suppose that $\Gamma$ acts freely on $X_{rat}$ and on $\g$ by diagram automorphisms. Then $W^{\Gamma}(\Psi)=\textbf{T}_{\mathbf{x}}(W({\Psi}_{\mathbf{x}}))$ for all $\Psi\in\varepsilon^{\Gamma}$ and $\mathbf{x}\in (\Supp\Psi)_{\Gamma}$. 
\end{thm}
\begin{proof}
We prove this using the three equivalent conditions for a twisted local Weyl module $W^{\Gamma}(\Psi)$ given in Lemma \ref{lem:11}.

Let $\Psi\in\varepsilon^{\Gamma}$ and $\textbf{x}\in(\Supp\ \Psi)_{\Gamma}$. Let $\lambda'=\wt\Psi_{\textbf{x}}$ and $\lambda=\wt_{\Gamma}\Psi$, such that $\lambda=\lambda'|_{\h_{\Gamma}}$. Then $W(\Psi_{\textbf{x}})\in\mathcal{I}_{\lambda'}$ and so from Lemma \ref{lem:13}, we have $\textbf{T}_{\textbf{x}}(W(\Psi_{\textbf{x}}))\in\mathcal{I}^{\Gamma}_{\lambda}$.

Let $V^{\Gamma}(\mu),  \mu\in\varepsilon^{\Gamma}$, be an irreducible finite dimensional object in $\mathcal{I}_{\lambda}^{\Gamma}$ with $(V^{\Gamma}(\mu))_{\lambda}=0$. This implies that $\wt_{\Gamma}\mu\in\lambda-Q^{+}_{\Gamma}\in \wt_{\Gamma}\Psi-Q^{+}_{\Gamma}$. Thus we have $\het\mu_{\textbf{x}}<\het\Psi_{\textbf{x}}$. We take $\textbf{x}$ in such a manner that $\Supp~\mu\in\Gamma. x$ and hence $V^{\Gamma}(\mu)\in\mathcal{F}_{\textbf{x}}^{\Gamma}$. For $l=0,1$, we get
\begin{equation*}
     \Ext^{l}_{\mathcal{F}_{\textbf{x}}^{\Gamma}}(\textbf{T}_{\textbf{x}}(W(\Psi_{\textbf{x}})),V^{\Gamma}(\mu))=\Ext^{l}_{\mathcal{F}_{\textbf{x}}^{\Gamma}}(\textbf{T}_{\textbf{x}}(W(\Psi_{\textbf{x}})),\textbf{T}_{\textbf{x}}(V(\mu_{\textbf{x}})))=\Ext^{l}_{\mathcal{F}_{\textbf{x}}^{\Gamma}}(W(\Psi_{\textbf{x}}),V(\mu_{\textbf{x}}))=0
\end{equation*}
 where, the last equality comes from \cite[theorem A.3]{BCM19}. 
 
 The weight space of $W(\Psi_{\textbf{x}})_{\lambda'}$ is isomorphic to the weight space $V(\Psi_{\textbf{x}})_{\lambda'}$ as an $(\h\otimes A)-$module. We had already shown in Theorem \ref{thm:3} that, $\textbf{T}_{\textbf{x}}(V(\Psi_{\textbf{x}}))=V^{\Gamma}(\Psi)$. Now if we restrict the action to $(\h\otimes A)^{\Gamma}$, we get that $(\textbf{T}_{\textbf{x}}(W(\Psi_{\textbf{x}})))_{\lambda}$ is isomorphic to $V^{\Gamma}(\Psi)_{\lambda}$ as $(\h\otimes A)^{\Gamma}-$modules and hence as $A_{\lambda}^{\Gamma}$- modules. This isomorphism is because the twisting functor $\textbf{T}_{\textbf{x}}$ is an isomorphism of categories. Therefore, we get that $R_{\lambda}^{\Gamma}(\textbf{T}_{\textbf{x}}(W(\Psi_\textbf{x})))=(\textbf{T}_{\textbf{x}}(W(\Psi_{\textbf{x}}))_{\lambda}\cong (V^{\Gamma}(\Psi))_{\lambda}\cong M^{\Gamma}(\Psi)$. Thus by Lemma \ref{lem:11}, $W^{\Gamma}(\Psi)=\textbf{T}_{\textbf{x}}(W(\Psi_{\textbf{x}}))$.
\end{proof}

{\bf Acknowledgment:} Authors are supported by a grant from the Anusandhan National Research Foundation (ANRF)(File No.: SRG/2023/002255 ).



\bibliographystyle{plain}
\bibliography{Nayak-SERB-plan}

\end{document}